\documentclass[11pt]{amsart}

\usepackage{mathpple}
\usepackage{amsmath,amsfonts, amscd,graphicx}
\usepackage{amsrefs}
\usepackage{amsthm}
\usepackage[all]{xy}
\usepackage{color}

\usepackage{hyperref}

\usepackage{mathtools}
\usepackage{float}
\usepackage{appendix}

\numberwithin{equation}{section}

\newcommand{\op}{\operatorname}
\newcommand{\C}{\mathbb{C}}

\newcommand{\Q}{\mathbb{Q}}
\newcommand{\Z}{\mathbb{Z}}

\newcommand{\mbf}{\mathbf}

\newcommand{\mc}{\mathcal}

\newcommand{\pa}{\partial}

\newcommand{\into}{\hookrightarrow}

\newcommand{\iso}{\cong}

\renewcommand{\L}{\mathcal L}

\renewcommand{\Im}{\op{Im}}

\usepackage{amssymb}
\usepackage{verbatim}
\usepackage{mathrsfs}

\newcommand{\be}{\mathbf{1}}
\newcommand{\ga}{\gamma}
\newcommand{\LD}{\langle}
\newcommand{\RD}{\rangle}
\newcommand{\LL}{\mathscr{L}}

\DeclareMathOperator{\End}{End}

\DeclareMathOperator{\Sym}{Sym}

\DeclareMathOperator{\Jac}{Jac}
\DeclareMathOperator{\Res}{Res}

\theoremstyle{plain}
\newtheorem{thm}{Theorem}[section]

\newtheorem{thm-defn}{Theorem/Definition}[section]
\newtheorem{lem}[thm]{Lemma}
\newtheorem{lem-defn}[thm]{Lemma/Definition}
\newtheorem{prop}[thm]{Proposition}

\newtheorem{cor}[thm]{Corollary}
\newtheorem{conjecture}[thm]{Conjecture}
\newtheorem{prop-defn}[thm]{Proposition-Definition}
\newtheorem{defn}[thm]{Definition}%[section]

\newtheorem{eg}[thm]{Example}

\newtheorem{thm-alg}[thm]{Theorem/Algorithm}
\newtheorem{rmk}[thm]{Remark}

\newtheorem{ques}[thm]{Question}

\allowdisplaybreaks[4]  %%%%%%%%%%% choose 1-4, where 4 is the strongest desire to breack

\begin{document}
\setlength{\unitlength}{1mm}

 \title{Mirror symmetry for exceptional  unimodular singularities}
  \author{Changzheng Li}
  \address{ Center for Geometry and Physics, Institute for Basic Science (IBS), Pohang 790-784, Republic of   Korea}
\email{czli@ibs.re.kr}

\author{Si Li}
  \address{Mathematical Science Center, Tsinghua University, Beijing, 100084, China}
\email{sli@math.tsinghua.edu.cn}

 \author{Kyoji Saito}
  \address{Kavli Institute for the Physics and Mathematics of the Universe (WPI),
Todai Institutes for Advanced Study, The University of Tokyo,
5-1-5 Kashiwa-no-Ha,Kashiwa City, Chiba 277-8583, Japan}
\email{kyoji.saito@ipmu.jp}

\author{Yefeng Shen}
  \address{Department of Mathematics, Stanford University, Stanford, California, 94305, USA}
\email{yfshen@stanford.edu}
  %\date{\today}
\date{}

  % ����
  \maketitle

%%%%%%%%%%%%%%%%%%%%%%%%%%%%%%
%% ǰ�Բ���
%%%%%%%%%%%%%%%%%%%%%%%%%%%%%%

\begin{abstract}
In this paper, we prove the mirror symmetry conjecture between the Saito-Givental theory of   exceptional unimodular singularities  on Landau-Ginzburg B-side and the Fan-Jarvis-Ruan-Witten theory of their mirror partners  on Landau-Ginzburg A-side.
On the B-side, we develop a perturbative method to compute the genus-zero correlation functions associated to the primitive forms. This is applied to the exceptional unimodular singularities,  and we show that the numerical invariants match the orbifold-Grothendieck-Riemann-Roch and WDVV calculations in FJRW theory on the A-side.
The coincidence of the  full data at all genera is established by reconstruction techniques. Our result establishes the first examples of LG-LG mirror symmetry for weighted homogeneous polynomials of central charge greater than one (i.e. which contain negative degree deformation parameters).
%%%\keywords{Aaaa, bbbb, cccc}
\end{abstract}

%We also verify that there exists choice of normal forms, such that the Arnold's strange duality coincide with BHK mirror construction.

  % Ŀ¼
 \tableofcontents
  % ����Ŀ¼
%  \listoftables
  % ��ͼĿ¼
%  \listoffigures

\addtocontents{toc}{\protect\setcounter{tocdepth}{1}}

\begin{comment}

\end{comment}

\section{Introduction}

Mirror symmetry is a fascinating geometric phenomenon discovered in string theory. The rise of mathematical interest dates back to the early 1990s, when Candelas,  Ossa, Green and Parkes
\cite{CDGP} successfully predicted the number of rational curves on the quintic 3-fold in terms of period integrals on the mirror quintics. Since then, one popular mathematical formulation of mirror symmetry is about the equivalence on the mirror pairs between the Gromov-Witten theory of counting curves  and the theory of variation of Hodge structures.  This is proved in \cite{LLY, G3} for a large class of mirror examples via toric geometry. Mirror symmetry has also deep extensions to open strings incorporating with D-brane constructions \cite{SYZ,Kontsevich}. In our paper, we will focus on closed string mirror symmetry.

Gromov-Witten theory presents the mathematical counterpart of A-twisted supersymmetric nonlinear $\sigma$-models, borrowing the name of A-model in physics terminology.  Its mirror theory is    called the B-model. On either side, there is  a closely related linearized model, called the N=2 Landau-Ginzburg model (or LG model), describing the quantum geometry of singularities. There exist deep connections in physics between nonlinear sigma models on Calabi-Yau manifolds and Landau-Ginzburg models (see \cite{mirror-book} for related literature).

In this paper, we will study the LG-LG mirror symmetry conjecture, which asserts an equivalence of two nontrivial theories of singularities for  mirror pairs $(W, G), (W^T, G^T)$. Here $W$ is an \emph{invertible  weighted homogeneous polynomial} on $\C^n$ with an isolated critical point at the origin, and $G$ is a finite abelian symmetry group of $W$.  The mirror weighted homogeneous polynomial $W^T$ was introduced   by Berglund and H\"ubsch \cite{BH} in early 1990s. For invertible polynomial $W=\sum_{i=1}^{n}\prod_{j=1}^{n}x_j^{a_{ij}},$ the mirror polynomial is  $W^T=\sum_{i=1}^{n}\prod_{j=1}^{n}x_j^{a_{ji}}$.
The mirror group $G^T$ was introduced by Berglund and Henningson \cite{BHe} and Krawitz  \cite{K} independently. Krawitz also constructed a ring isomorphism between two models. Now the mirror symmetry between these LG pairs is also called \emph{Berglund-H\"ubsch-Krawitz mirror} \cite{CR}.
When $G = G_W$ is the   group of diagonal symmetries of $W$, the dual group $G_W^T=\{1\}$ is trivial.   In order to   formulate the conjecture, let us introduce the theories on both sides first. We remark that one of the most general mirror constructions of LG models was proposed by  Hori and Vafa \cite{HV}.

\begin{comment}
and the mirror  group $G^T$ was later constructed by  Krawitz \cite{K}. The pair $(W^T, G^T)$ is sometimes referred to as the BHK mirror of $(W, G)$. When $G=G_{W}$ is the    group of diagonal symmetries of  $W$, the mirror   group $G^T$ is   trivial.    $G$ is a finite abelian symmetry group of $W$
\end{comment}

A geometric candidate of LG A-model   is the Fan-Jarvis-Ruan-Witten theory (or FJRW theory) constructed by Fan, Jarvis and Ruan \cite{FJR,FJR2}, based on a proposal of Witten \cite{W}.
 Several purely algebraic versions of LG A-model have been worked out \cite{PV, CLL}.
The FJRW theory is closely related to the Gromov-Witten theory, in terms of the so-called Landau-Ginzburg/Calabi-Yau correspondence \cite{R, CR2}.
 The purpose of the FJRW theory is to solve the moduli problem for the Witten equations of a LG model $(W,G)$ ($G$ is an appropriate subgroup of $G_W$). The outputs are the \emph{FJRW invariants}. Analogous to the Gromov-Witten invariants, the FJRW invariants are defined via the intersection theory of appropriate \emph{virtual fundamental cycles} with tautological classes on the moduli space of stable curves.
 These invariants  virtually  count the solutions of the Witten equations on orbifold curves. For our purpose later, we consider  $G=G_W$, and summarize the main ingredients of the FJRW theory  as follows (see Section \ref{sec2} for more details):
\begin{itemize}
\item An \emph{FJRW ring} $(H_{W}, \bullet)$. Here $H_{W}$ is the \emph{FJRW state space} given by the $G_W$-invariant Lefschetz thimbles of $W$,  and the multiplication $\bullet$ is defined by an intersection pairing together with the \emph{genus $0$ primary FJRW invariants} with $3$ marked points.

\item A \emph{prepotential} $\mathcal{F}_{0,W}^{\rm FJRW}$ of a {formal} Frobenius manifold structure on $H_{W}$, whose coefficients are all the {\it genus $0$ primary FJRW invariants} $\LD\cdots\RD_0^W$.

\item A \emph{total ancestor potential} $\mathscr{A}_{W}^{\rm FJRW}$   that collects  the FJRW invariants at all genera.
\end{itemize}

\begin{comment}
As have mentioned earlier, we will only study the case $G=G_W$. It turns out that $(G_W)^T=\{\mbox{1}\}$ is the trivial group, and a Saito-Givental theory for $W^T$ (which will be introduced in the next two sections) is an appropriate B-model of the pair $(W^T,(G_W)^T)$. When there is no risk of confusion, we skip the subscripts $G_W$ and $(G_W)^T$.
\end{comment}

A geometric candidate of the LG B-model of $(W^T, G^T)$ is still missing for general $G^T$. When $G=G_W$, then $G^T=\{1\}$   and  a candidate comes from the third author's theory of primitive forms \cite{Saito-primitive}. The starting point here is a germ of holomorphic function ($f=W^T$ to our interest here)
$$
f(\mathbf{x}):(\C^n,\mathbf{0})\to (\C,{0}), \quad \mathbf{x}=\{x_i\}_{i=1,\cdots,n}
$$
with an isolated singularity at the origin $\mathbf{0}$.    We consider its universal unfolding
$$
  (\C^n\times \C^\mu, \mathbf{0}\times \mathbf{0})\to (\C\times \C^\mu, 0\times \mathbf{0}), \quad (\mathbf{x}, \mathbf{s})\to (F(\mathbf{x}, \mathbf{s}), \mathbf{s})
$$
where $\mu=\dim_{\C}\Jac(f)_{\mathbf{0}}$ is the Milnor number, and $\mathbf{s}=\{s_\alpha\}_{\alpha=1,\cdots,\mu}$ parametrize the deformation. Roughly speaking, a primitive form is a relative holomorphic volume form
$$
\zeta=P(\mathbf{x},\mathbf{s})d^n\mathbf{x}, \quad d^n\mathbf{x}=dx_1\cdots dx_n
$$
at the germ $(\C^n\times \C^\mu, \mathbf{0}\times \mathbf{0})$, which induces a Frobenius manifold structure (which is called the flat structure in  \cite{Saito-primitive}) at the germ  $(\C^\mu, \mathbf{0})$. This gives the genus 0 invariants in the LG B-model.  At higher genus, Givental \cite{G2} proposed a remarkable formula (with its uniqueness established by Teleman \cite{T})  of the total ancestor potential for semi-simple Frobenius manifold structures, which can be extended to some non-semisimple boundary points  \cite{M,CI} including $\mathbf{s}=\mathbf{0}$ of our interest. The whole package is now referred to as the Saito-Givental theory. We will call the extended total ancestor potential at $\mathbf{s}=\mathbf{0}$ a Saito-Givental potential and denote it by $\mathscr{A}_{f}^{\rm SG}$.

For $G=G_W$, the  LG-LG mirror conjecture (of all genera) is well-formulated \cite{CR}:
\begin{conjecture}\label{conj}
For a mirror pair $(W,G_W)$ and $(W^T, \{1\})$, there exists a ring isomorphism $(H_{W},\bullet)\cong{\rm Jac}(W^T)$ together with a choice of primitive forms $\zeta$, such that  the FJRW  potential $\mathscr{A}_{W}^{\rm FJRW}$ is identified with the Saito-Givental potential $\mathscr{A}_{W^T}^{\rm SG}$.
\end{conjecture}

For the weighted homogeneous polynomial   $W=W(x_1, \cdots, x_n)$, we have
$$
    W(\lambda^{q_1}x_1,\cdots, \lambda^{q_n}x_n)=\lambda  W(x_1,\cdots, x_n), \,\,\,\forall \lambda\in\C^*,
$$
with each weight  $q_i$ being a unique rational number satisfying   $0<q_i\leq {1\over 2}$   \cite{Saito-quasihomogeneous}.
There is a partial   classification of $W$ using the central charge \cite{Saito-exceptional}
$$
\hat{c}_W:=\sum_{i}(1-2q_i).
$$
   So far, Conjecture \ref{conj} has only been proved for $\hat{c}_W<1$ (i.e., $ADE$ singularities) by Fan, Jarvis and Ruan\cite{FJR} and for $\hat{c}_W=1$ (i.e., simple elliptic singularities) by Krawitz, Milanov and Shen \cite{KS,MS}.
   However, it was  open for $\hat{c}_W>1$,
including exceptional unimodular modular singularities and a wide class of those related to K3 surfaces and CY 3-folds. One of the major obstacle is that computations in the LG B-model require concrete information about the primitive forms. The existence of  the primitive forms for a general isolated singulary has been  proved   by M.Saito \cite{Saito-existence}.
   However, explicit formulas  were only known for weighted homogeneous polynomials of $\hat c_W\leq 1$ \cite{Saito-primitive}. This is due to the difficulty of mixing between positive and negative degree deformations when $\hat c_W>1$.

 The   main objective of the present   paper is to prove that Conjecture \ref{conj} is true when $W^T$ is one of the   exceptional unimodular singularities as in the following table. Here we use variables $x, y, z$ instead of the conventional $x_1, x_2, \cdots, x_n$.
These polynomials are all of central charge larger than 1,  providing the first nontrivial examples with the existence of negative degree deformation (i.e., irrelevant deformation) parameters.
  % \begin{center}

    \begin{table}[h]
        \caption{\label{tab-exceptional-singularities}  Exceptional unimodular singularities
    }
      \resizebox{\columnwidth}{!}{
        \begin{tabular}{|c|c||c|c||c|c||c|c|}
       \hline
         & Polynomial &  & Polynomial &  & Polynomial &   & Polynomial   \\ \hline \hline

    $E_{12}$ & $x^3+y^7$ &  $W_{12}$ & $x^4+y^5$ &$U_{12}$ & $x^3+y^3+z^4$ & &    \\              \hline
       $Q_{12}$ & $ {x^2y+xy^3+z^3}$  &  $Z_{12}$ & $x^3y+y^4x$ & $S_{12}$ & $x^2y+y^2z+z^3x$   &&  \\              \hline
   $E_{14}$ & $ {x^2+xy^4+z^3}$  &  $E_{13}$ & $x^3+xy^5$  & $Z_{13}$  & $ {x^2+xy^3+yz^3}$&   $W_{13}$ & $ {x^2+xy^2+yz^4}$   \\              \hline
       $Q_{10}$ & $x^2y+y^4+z^3$ & $Z_{11}$ & $x^3y+y^5$    &  $Q_{11}$ & $x^2y+y^3z+z^3$ &$S_{11}$ & $x^2y+y^2z+z^4$     \\              \hline
  \end{tabular}
  }
    \end{table}
   %  \end{center}

Originally, the 14 exceptional unimodular singularities by Arnold \cite{Arnold-book
} are one parameter families of singularities with three variables. Each family contains a weighted homogenous singularity characterized by the existence of only one negative degree but no zero-degree deformation parameter \cite{Saito-exceptional}. In this paper, we consider the stable equivalence class of a singularity, and always choose polynomial representatives of the class with no square terms for additional variables.  The FJRW theory with the group of diagonal symmetries is invariant when adding square terms for additional variables.

\subsection*{LG-LG mirror symmetry for   exceptional unimodular singularities}
 Let us explain how we achieve the goal in more details.
Following \cite{K}, we can  specify a ring isomorphism $\Jac(W^T)\iso (H_{W}, \bullet)$.
Then we calculate certain FJRW invariants, by  an orbifold-Grothendieck-Riemann-Roch formula and WDVV equations. More precisely, we have
\begin{prop}\label{thm:FJRW}
Let $W^T$ be  one of the  14 singularities, then
 $$\Psi:\Jac(W^T)\to (H_{W}, \bullet),$$
 defined in \eqref{ring-iso} and \eqref{broad-gen},  generates a ring isomorphism.
Let $M_i^T$ be the $i$-th monomial of $W^T$, and $\phi_\mu$ be   of the highest degree among the specified basis of $\Jac(W^T)$ in Table \ref{invertible-singularities}. Let $q_i$ be the weight of $x_i$ with respect to  $W$.
For each $i$, we have genus $0$ FJRW invariants
\begin{equation}\label{4-point}
\LD \Psi(x_i),\Psi(x_i),\Psi(\frac{M_i^T}{x_i^2}),\Psi(\phi_{\mu})\RD_{0}^{W}=q_i, \mbox{ whenever } M_i^T\neq x_i^2.
\end{equation}
\end{prop}
\noindent Surprisingly,  if $W^T$ belongs to $Q_{11}$ or $S_{11}$, then the ring isomorphism  $\Jac(W^T)\iso (H_{W},\bullet)$ was  not known in the literature. The difficulty comes from that if there is some $q_j=\frac{1}{2}$, then one of the ring generators is a so-called \emph{broad element} in FJRW theory and invariants with broad generators are hard to compute. We overcome this difficulty  for the two cases, using Getzler's relation on $\overline{\mathcal{M}}_{1,4}$. It is quite interesting that   the higher genus structure  detects the ring structure.
  We expect that our   method  works for general unknown cases of   $(H_{W},\bullet)$ as well.

\begin{comment}
that the FJRW theory of  weighted homogeneous polynomials $W$, with respect to the group $G_W$, is equivalent to the Saito-Givental's theory of their BHK mirrors $W^T$
\end{comment}

\begin{comment}
On the other hand, by using the perturbative formula on B-side, we  derive \textbf{Theorem/Algorithm \ref{thm-flat-uptoorder2}}, which provides the four-point (primary) correlators of the Frobenenus manifold structure associated to the primitive form of $W^T$.
\end{comment}

On the B-side,   recently there has appeared a perturbative way to compute the primitive forms for arbitrary weighted homogeneous singularities \cite{LLSaito}. In this paper, we develop the perturbative method to the whole package of the associated Frobenius manifolds, and describe a recursive algorithm to compute the associated flat coordinates and the potential function $\mc F_{0, W^T}^{\rm SG}$ (see section 3.2).
 We apply this perturbative method to compute genus zero invariants of LG B-model associated to the unique primitive forms  \cite{Hertling-classifyingspace, LLSaito}  of the exceptional unimodular singularities, and show that it coincides with the   A-side FJRW invariants for $W$ in Proposition \ref{thm:FJRW} (up to a sign).

\begin{comment}
In the rest of the introduction, we will briefly sketch the mirror theories in LG models and explain our result.
Our aim is to prove the LG-LG mirror symmetry Conjecture \ref{conj} for  $W^T$ being given  by  Arnold's 14 exceptional unimodular singularities.
\end{comment}

In the next step, we establish a reconstruction theorem in such cases (Lemma \ref{reconstruction-lemma}), showing that the WDVV equations are powerful enough to determine the full prepotentials for both sides  from those invariants in \eqref{4-point}. This gives the main result of our paper:
\begin{thm}\label{g=0-mirror}
 Let $W^T$ be  one of the 14 exceptional unimodular singularities in Table \ref{tab-exceptional-singularities}. Then the specified  ring isomorphism $\Psi$ induces an isomorphism of Frobenius manifolds between   $\Jac(W^T)$ (which comes from the primitive form of $W^T$) and $H_{W}$ (which comes from the FJRW theory of $(W, G_W)$).
That is,    the prepotentials are equal to each other:
\begin{equation}\label{thm-0}
\mathcal{F}_{0,W^T}^{\rm SG}=\mathcal{F}_{0,W}^{\rm FJRW}.
\end{equation}

\end{thm}

In general, the computations of FJRW invariants are challenging  due to our very little understanding of virtual fundamental cycles, especially at higher genus. However, according to Teleman \cite{T}  and Milanov \cite{M}, the non-semisimple limit $\mathscr{A}_{W^T}^{\rm SG}$ is fully determined by the genus-0 data on the semisimple points nearby. As a consequence, we upgrade our mirror symmetry statement to higher genus and prove Conjecture \ref{conj} for the   exceptional unimodular singularities.
\begin{cor}\label{main}
Conjecture \ref{conj} is true for $W^T$ being one of the     14 exceptional unimodular singularities in Table \ref{tab-exceptional-singularities}.
 The specified ring isomorphism  $\Psi$ induces the following coincidence of total ancestor potentials:
\begin{equation}\label{thm-g}
\mathscr{A}_{W^T}^{\rm SG}=\mathscr{A}_{W}^{\rm FJRW}.
\end{equation}
\end{cor}
\noindent The choice in Table \ref{tab-exceptional-singularities} has the property that the mirror weighted homogeneous polynomials are again representatives  of the exceptional unimodular singularities. Arnold discovered a strange duality among the 14 exceptional unimodular singularities, which says the Gabrielov numbers of each coincide with the Dolgachev numbers of its strange dual   \cite{Arnold-strangduality}. The strange duality is also reproved algebraically in \cite{Saito-duality}.
The choices in Table \ref{tab-exceptional-singularities} also represent  Arnold's strange duality: the first two rows are strange dual to themselves, and the last two rows are dual to each other. For example, $E_{14}$ is strange dual to $Q_{10}.$ Beyond the choices in Table \ref{tab-exceptional-singularities}, we also discuss the LG-LG mirror symmetry for other invertible polynomial representatives (some of whose mirrors may no longer be exceptional exceptional singularities) where equality \eqref{thm-g} still holds. The results are summarized in \textbf{Theorem \ref{addtional7-FJRW}} and  \textbf{Remark \ref{rmk.4.5}}.
Our method has the advantage of being applicable to general invertible polynomials with more involved WDVV techniques developed.

  The present paper is organized as follows.  In section 2, we give a brief review of the FJRW theory and compute those initial FJRW invariants as in Proposition \ref{thm:FJRW}.
   In section 3, we develop the perturbative method for computing the Frobenius manifolds in the LG B-model following \cite{LLSaito}. In section 4, we prove Conjecture \ref{conj} when the B-side is given by one of  the exceptional unimodular singularities. We also discuss the more general case when either side is given by an arbitrary weighted homogeneous polynomial representative  of the exceptional unimodular singularities.  Finally in the appendix, we provide   detailed descriptions of the  specified isomorphism $\Psi$ as well as a complete list of the genus-zero four-point functions on the B-side for all the  exceptional unimodular singularities.
    We would like to point out that section 2 and section 3 are completely independent of each other. Our readers can choose either sections to start  first.

\section{A-model: FJRW-theory}\label{sec2}

\subsection{FJRW-theory}
In this section, we give a brief review of FJRW theory. For more details, we refer the readers to \cite{FJR, FJR2}.
%\subsubsection{Cohomology space}
We start with a \emph{nondegenerate weighted homogeneous polynomial} $W=W(x_1,\cdots,x_n)$, where the nondegeneracy means that $W$ has isolated critical point at the origin $\mbf{0}\in\C^n$ and contains no monomial of the form $x_ix_j$ for $i\neq j$. This implies that each $x_i$ has a unique weight $q_i\in\mathbb{Q}\cap (0,\frac{1}{2}]$ \cite{Saito-quasihomogeneous}. Let $G_W$ be the \emph{group of diagonal symmetries},
\begin{equation}\label{eq:diag}
G_{W}:=
\left\{(\lambda_1,\dots,\lambda_n)\in(\mathbb{C}^*)^n\Big\vert\,W(\lambda_1\,x_1,\dots,\lambda_n\,x_n)=W(x_1,\dots,x_n)\right\}.
\end{equation}
In this paper, we will only consider the FJRW theory for the pair $(W,G_W)$. In general, the FJRW theory also works for any subgroup $G\subset G_W$ where $G$ contains the {\em exponential grading element}
\begin{equation}\label{exp-grading}
J=\left(\exp(2\pi\sqrt{-1}q_1),\cdots,\exp(2\pi\sqrt{-1}q_n)\right)\in G_W.
\end{equation}
\begin{defn}
The {\rm FJRW} state space $H_W$ for $(W,G_W)$ is defined to be the direct sum of all $G_W$-invariant relative cohomology:
\begin{equation}\label{FJRW-vec}
H_{W}:=
\bigoplus_{\gamma\in G_W}H_{\gamma},\quad
H_\gamma:=H^{N_\gamma}({\rm Fix}(\gamma);W^{\infty}_{\gamma};\mathbb{C})^{G_W}.
\end{equation}
Here ${\rm Fix}(\gamma)$ is the fixed points set of $\gamma$, and $\C^{N_{\gamma}}\cong{\rm Fix}(\gamma)\subset\C^n$. $W_{\gamma}$ is the restriction of $W$ to
${\rm Fix}(\gamma)$. $W^{\infty}_{\gamma}:=(\mathrm{Re}W_{\gamma})^{-1}(M,\infty)$ with $M\gg0$, where  ${\rm Re}W_{\gamma}$ is the real part of $W_{\gamma}$.
\end{defn}
Each $\gamma\in G_W$ has a unique form
\begin{equation}\label{theta}
\gamma=\left(\exp(2\pi\sqrt{-1}\Theta_1^{\ga}),\cdots,\exp(2\pi\sqrt{-1}\Theta_n^{\ga})\right)\in(\mathbb{C}^*)^n, \quad \Theta_i^{\gamma}\in[0,1)\cap\mathbb{Q}.
\end{equation}
Thus $H_{W}$ is a graded vector space, where for each nonzero $\alpha\in H_{\gamma}$, we assign its   degree
$$\deg\alpha={N_{\gamma}\over2}+\sum_{i=1}^n(\Theta_i^{\ga}-q_i).$$
We call
$H_\ga$ a \emph{narrow sector} if ${\rm Fix}(\ga)$ consists of  ${\bf 0}\in\mathbb{C}^n$ only, or a \emph{broad sector} otherwise.

The {\rm FJRW} vector space $H_{W}$ is equipped with a symmetric and nondegenerate pairing $$\LD\ ,\ \RD:=\sum_{\gamma\in G_W}\LD\ ,\ \RD_{\gamma},$$
where each $\LD\ ,\ \RD_{\gamma}:H_{\gamma}\times H_{\gamma^{-1}}\to\mathbb{C}$ is induced from the intersection pairing of Lefchetz thimbles. The pairing between $H_{\ga_1}$ and $H_{\ga_2}$ is nonzero only if $\ga_1\ga_2=1$. Moreover, there is a canonical isomorphism (see section 5.1 of \cite{FJR}, Appendix A of \cite{CIR}, and references therein)
\begin{equation}\label{eq:pair}
\left(H_{W},\LD\ ,\ \RD\right)\cong\left(\bigoplus_{\ga\in G_W}(\Jac(W_\ga)\omega_\ga)^{G_W}, \sum_{\ga\in G_W}\LD\ ,\RD_{{\rm res}, \ga}\right).
\end{equation}
Here $\omega_\ga$ is a volume form on ${\rm Fix}(\ga)$ of the type $dx_{j_1}\wedge\cdots\wedge dx_{j_{N_{\gamma}}}$, where we mean $\omega_\gamma=1$ if $N_\gamma=0$. $G_W$ acts on both   $x_i$ and $dx_i$. Let $(\Jac(W_\ga)\omega_\ga)^{G_W}$ be the $G_W$-invariant part of the action.
\begin{comment}
Each $g=(\lambda_1,\cdots,\lambda_n)\in G_W$ acts on $\Jac(W_\ga)\omega_\ga$ as
$$g\cdot (x_{j_1}^{m_{j_1}}\cdots x_{j_{N_{\gamma}}}^{m_{N_{\gamma}}}\omega_{\gamma})
=\left(\prod_{i=1}^{N_{\ga}}\lambda_{j_i}^{-m_{j_i}-1}\right)x_{j_1}^{m_{j_1}}\cdots x_{j_{N_{\gamma}}}^{m_{N_{\gamma}}}\omega_{\ga}.$$
\end{comment}
We choose a generator
\begin{align}\label{eqn-notationEgamma}
   \be_\ga:=\omega_\gamma\in H^{N_{\gamma}}({\rm Fix}(\ga);W^{\infty}_\ga;\mathbb{C}).
\end{align}
If $H_\ga$ is narrow, then $H_\ga\cong(\Jac(W_\ga)\omega_\ga)^{G_W}\cong\C$ is generated by $\be_\ga$.
If $H_{\ga}$ is broad, we denote generators of $H_\ga$ by $\phi\,\be_{\gamma}$ via $\phi\,\omega_{\ga}\in(\Jac(W_\ga)\omega_\ga)^{G_W}$. Finally, the residue pairing $\LD\ ,\RD_{{\rm res}, \ga}$
is defined from the standard residue ${\rm Res}_{W_\ga}$ of $W_\ga$,
$$\LD f\omega_\ga ,g\omega_\ga\RD_{{\rm res}, \ga}:={\rm Res}_{W_\ga}(fg):={\rm Residue}_{x=0}
{fg\omega_\ga\over
{\partial W_\ga\over\partial x_{j_1}}
\cdots {\partial W_\ga\over\partial x_{j_{N_\ga}}
}
}.$$

%\subsubsection{Cohomological field theory and FJRW invariants}\label{sec-FJRW}
It is highly nontrivial to construct a virtual cycle for the moduli of solutions of Witten equations.
For the details of the construction, we refer to the original paper of Fan, Jarvis and Ruan \cite{FJR2}.
Let $\mathscr{C}:=\mathscr{C}_{g,k}$ be a stable genus-$g$ orbifold curve  with marked points $p_1,\dots,p_k$ (where $2g-2+k>0$). We only allow orbifold points at marked points and nodals. Near each orbifold point $p$, a local chart is given by $\C/G_p$ with $G_p\cong\Z/m\Z$ for some positive integer $m$. Let $\LL_1,\dots,\LL_n$ be orbifold line bundles over $\mathscr{C}$. Let $\sigma_i$ be a $C^{\infty}$-section of $\LL_i$.
We consider the $W$-structures, which can be thought of as the background data to be used to set up the Witten equations
\begin{equation*}
\bar{\partial}\sigma_i+\overline{\frac{\partial W}{\partial\sigma_i}}=0.
\end{equation*}
For simplicity, we only discuss cases that $W=M_1+\dots+M_{n}$,  with $M_i=\prod_{j=1}^{n}x_j^{a_{ij}}$.
\begin{comment}
We write $W$ as a sum of monomial basis, $W=\sum_{j=1}^{s}c_jM_j$ where $M_j=x_1^{a_{j,1}}\dots x_n^{a_{j,n}}$.
\end{comment}
Let $K_C$ be the canonical bundle for the underlying curve $C$ and $\rho:\mathscr{C}\to C$ be the forgetful morphism.
A \emph{$W$-structure} $\mathfrak{L}$ consists of $(\mathscr{C}, \LL_1, \cdots, \LL_n, \varphi_1,\cdots, \varphi_n)$ where $\varphi_i$ is an isomorphism of orbifold line bundles
\begin{equation*}
\varphi_i:\bigotimes_{j=1}^{n}\LL_j^{\otimes a_{i,j}}\longrightarrow \rho^*(K_{C, \rm log}), \quad K_{C, \rm log}:=K_C\otimes\bigotimes_{j=1}^{k} \mathscr{O}(p_j).
\end{equation*}
A $W$-structure induces a representation $r_p:G_{p}\to G_W$ at each point $p\in\mathscr{C}$. We require it to be faithful.
The moduli space of pairs $\mathfrak{C}=(\mathscr{C},\mathfrak{L})$ is called the \emph{moduli of stable $W$-orbicurves} and denoted by $\overline{\mathscr{W}}_{g,k}$.
According to \cite{FJR}, $\overline{\mathscr{W}}_{g,k}$ is a Deligne-Mumford stack, and the forgetful morphism ${\rm st}: \overline{\mathscr{W}}_{g,k}\to \overline{\mathcal{M}}_{g,k}$ to the moduli space of stable curves is flat, proper and quasi-finite.
$\overline{\mathscr{W}}_{g,k}$ can be decomposed into open and closed stacks by decorations on each marked point,
$$\overline{\mathscr{W}}_{g,k}=\sum_{(\ga_1,\dots,\ga_k)\in (G_W)^k}\overline{\mathscr{W}}_{g,k}(\ga_1,\dots,\ga_k),\quad \gamma_j:=r_{p_j}(1).$$
Furthermore, let $\Gamma$ be the dual graph of the underlying curve $C$. Each vertex of $\Gamma$ represents an irreducible component of $C$, each edge represents a node, and each half-edge represents a marked point. Let $\sharp E(\Gamma)$ be the number of edges in $\Gamma$. We decorate the half-edge representing the point $p_j$ by an element $\ga_j\in G_W$. We denote the decoration by $\Gamma_{\ga_1,\dots,\ga_k}$ and call it a {\em $G_W$-decorated dual graph}.
We further call it \emph{fully $G_W$-decorated} if we also assign some $\gamma_+\in G_W$ and $\gamma_-=(\gamma_+)^{-1}$ on the two sides of each edge.
The stack $\overline{\mathscr{W}}_{g,k}(\ga_1,\dots,\ga_k)$ is stratified, where each closure in $\overline{\mathscr{W}}_{g,k}(\ga_1,\dots,\ga_k)$ of the stack of stable $W$-orbicurves with fixed decorations $(\ga_1,\dots,\ga_k)$ on $\Gamma$ is denoted by $\overline{\mathscr{W}}_{g,k}(\Gamma_{\ga_1,\dots,\ga_k})$.
%We have the decomposition
%$$\overline{\mathscr{W}}_{g,k}(\ga_1,\dots,\ga_k)=\sum_{\Gamma}\overline{\mathscr{W}}_{g,k}(\Gamma_{\ga_1,\dots,\ga_k}).$$
%Let $\rho:\mathscr{C}\to C$ be the forgetful morphism to the underlying curve.

If $\overline{\mathscr{W}}_{g,k}(\ga_1, \cdots, \ga_k)$ is nonempty, then the \emph{Line bundle criterion} follows(\cite[Proposition 2.2.8]{FJR}):
\begin{equation}\label{eq:line bdle}
\deg (\rho_*\LL_i)=(2g-2+k)q_i-\sum_{j=1}^{k}\Theta_i^{\ga_j}\in \mathbb{Z},\,i=1, \cdots, n.
\end{equation}

\begin{comment}

\end{comment}

In \cite{FJR2}, Fan, Jarvis and Ruan perturb the polynomial $W$ to polynomials of Morse type and construct virtual cycles from the solutions of the perturbed Witten equations. Those virtual cycles transform in the same way as the Lefschetz thimbles attched to the critical points of the perturbed polynomials.
As a consequence, they construct a virtual cycle
$$[\overline{\mathscr{W}}_{g,k}(\Gamma_{\ga_1,\dots,\ga_k})]^{\rm vir}\in H_{*}(\overline{\mathscr{W}}_{g,k}(\Gamma_{\ga_1,\dots,\ga_k}),\mathbb{C})\otimes\prod_{j=1}^{k}H_{N_{\gamma_j}}({\rm Fix}(\ga_j),W^{\infty}_{\ga_j},\mathbb{C})^{G_W},$$
which has total degree
\begin{equation}\label{degree-vir}
2\left((\hat{c}_W-3)(1-g)+k-\sharp E(\Gamma)-\sum_{j=1}^{k}\sum_{i=1}^{n}(\Theta_i^{\gamma_j}-q_i)\right).
\end{equation}
Based on this, they obtain a cohomological field theory $\{\Lambda^W_{g,k}:(H_W)^{\otimes k}\longrightarrow H^*(\overline{\mathcal{M}}_{g,k},\mathbb{C})\}$ with a flat identity. Each $\Lambda^W_{g,k}$ is defined by extending the following map linearly to $H_W$,
$$\Lambda^W_{g,k}(\alpha_1,\dots,\alpha_k):=\frac{|G_W|^g}{\deg({\rm st})}{\rm PD}\ {\rm st}_*\left([\overline{\mathscr{W}}_{g,k}(\ga_1,\dots,\ga_k)]^{\rm vir}\cap\prod_{j=1}^{k}\alpha_j\right), \quad \alpha_j\in H_{\ga_j}.$$

\begin{defn}\label{defn:FJRW}
Let $\psi_j$ be the $j$-th psi class in $H^*(\overline{\mathcal{M}}_{g,k})$. Define FJRW invariants (or correlators)
\begin{equation}\label{FJRW inv}
\LD\tau_{\ell_1}(\alpha_1),\dots,\tau_{\ell_k}(\alpha_k)\RD_{g,k}^W=\int_{\overline{\mathcal{M}}_{g,k}}\Lambda_{g,k}^W(\alpha_1,\dots,\alpha_k)\prod_{j=1}^{k}\psi_j^{\ell_j}, \quad \alpha_j\in H_{\ga_j}.
\end{equation}
The FJRW invariants in \eqref{FJRW inv} are called primary if all $\ell_j=0$. We simply denoted them by $\LD\alpha_1,\dots,\alpha_k\RD^W_{g}.$
We call $(H_W,\bullet)$ an {\rm FJRW} ring where the multiplication $\bullet$ on $H_W$ is defined by
\begin{equation}\label{FJRW-ring}
\LD\alpha\bullet\beta, \ga\RD=\LD\alpha,\beta,\gamma\RD^W_{0}.
\end{equation}

\end{defn}

If the invariant in \eqref{FJRW inv} is nonzero, the intergrand should be a top degree element in $H^*(\overline{\mathcal{M}}_{g,k}).$
Then using the total degree formula \eqref{degree-vir} and the definition of the cohomological field theory, it is not hard to see that
\begin{equation}\label{selection-non-primary}
\sum_{j=1}^k\deg\alpha_j+\sum_{j=1}^k\ell_j=(\hat{c}_W-3)(1-g)+k.
\end{equation}

\begin{comment}
Let us fix a basis $\{\alpha_j\}_{j=1}^{\mu}$ of $H_W$ with $\alpha_1$ be the identity and introduce the following notation
$$\mbf{q}(z)=\sum_{m\geq0} \mbf{q}_m\,z^m=\sum_{m\geq0}\sum_{j=1}^{\mu}q_m^j\,\alpha_j\,z^m.$$

\end{comment}

Let us fix a basis $\{\alpha_j\}_{j=1}^{\mu}$ of $H_W$, with $\alpha_1$ being the identity. Let
$\mbf{t}(z)=\sum_{m\geq0}\sum_{j=1}^{\mu}t_{m, \alpha_j}\,\alpha_j\,z^m.$
\begin{comment}
For a fixed genus $g\geq0$, we introduce the \emph{genus-g {\rm FJRW} generating function}
$$\mathscr{F}^{\rm FJRW}_{g,W}=\sum_{\begin{subarray}{l}
        k\geq0,  k>2-2g
      \end{subarray}}\frac{1}{k!}\LD\mbf{t}(\psi_1)+\psi_1,\dots,\mbf{t}(\psi_k)+\psi_k\RD_{g,k}^{W}, \quad \mbf{t}(z)=\sum_{m\geq0}\sum_{j=1}^{\mu}t_m^{\alpha_j}\,\alpha_j\,z^m.$$
\end{comment}
The \emph{{\rm FJRW} total ancestor potential} is defined to be
\begin{equation}\label{FJRW-anc}
\mathscr{A}_{W}^{\rm FJRW}
%=\exp\left(\sum_{g\geq0}\hbar^{g-1}\mathscr{F}^{\rm FJRW}_{g,W}\right)
=\exp\left(\sum_{g\geq0}\hbar^{g-1}\sum_{\begin{subarray}{l}
        k\geq0
      \end{subarray}}\frac{1}{k!}\LD\mbf{t}(\psi_1)+\psi_1,\dots,\mbf{t}(\psi_k)+\psi_k\RD_{g,k}^{W}\right).
\end{equation}
There is a {formal} Frobenius manifold structure on $H_W$, in the sense of Dubrovin \cite{D}. Its \emph{prepotential} is given by
$$\mathcal{F}^{\rm FJRW}_{0,W}=\sum_{\begin{subarray}{l}
        k\geq3
      \end{subarray}}\frac{1}{k!}\LD\mbf{t}_0,\dots,\mbf{t}_0\RD_{0,k}^{W}, \quad
\mbf{t}_0=\sum_{j=1}^{\mu}t_{0, \alpha_j}\,\alpha_j.$$
The prepotential satisfies the WDVV (Witten-Dijkgraaf-Verlinde-Verlinde) equations:
\begin{equation}\label{WDVV}
\sum_{i,j}\frac{\partial^3 \mathcal{F}_{0, W}^{\scriptsize\mbox{FJRW}}}{\partial t_{\alpha_a} \partial t_{\alpha_d}  \partial t_{\alpha_i} }
\eta^{ij}
\frac{\partial^3 \mathcal{F}_{0, W}^{\scriptsize\mbox{FJRW}}}{\partial t_{\alpha_j}  \partial t_{\alpha_b}  \partial t_{\alpha_c} }
=
\sum_{i,j}\frac{\partial^3 \mathcal{F}_{0, W}^{\scriptsize\mbox{FJRW}}}{\partial t_{\alpha_a}  \partial t_{\alpha_b}  \partial t_{\alpha_i} }
\eta^{ij}
\frac{\partial^3 \mathcal{F}_{0, W}^{\scriptsize\mbox{FJRW}}}{\partial t_{\alpha_j}  \partial t_{\alpha_c}  \partial t_{\alpha_d} }, \quad t_{\alpha}:=t_{0, \alpha},
\end{equation}
where $\big(\eta^{ij}\big)$ is the inverse of the matrix $\big(\LD \alpha_i, \alpha_j\RD\big)$. It implies (\cite[Lemma 6.2.6]{FJR})
\begin{eqnarray}\label{eq:WDVV}
\begin{split}
\LD\dots,\alpha_a,\alpha_b\bullet\alpha_c,\alpha_d\RD_{0,k}
=&S_k+\LD\dots,\alpha_a\bullet\alpha_b,\alpha_c,\alpha_d\RD_{0,k}+\LD\dots,\alpha_a,\alpha_b,\alpha_c\bullet\alpha_d\RD_{0,k}\\
&-\LD\dots,\alpha_a\bullet\alpha_d,\alpha_b,\alpha_c\RD_{0,k}.
\end{split}
\end{eqnarray}
where $k\geq3,$ $S_k$ is a linear combination of products of correlators with number of marked points no greater than $k-1$. Moreover, both $S_3=S_4=0.$

\begin{comment}
Let us end this subsection by the following property of the FJRW invariants.   We refer our readers to \cite{FJR} for a complete list of the axioms that FJRW invariants satisfy.

\begin{lem}\label{lem-sel-string} \begin{enumerate}  \item{\upshape (Selection rule)} A correlator  $\LD\tau_{\ell_1}(\alpha_1),\dots,\tau_{\ell_k}(\alpha_k)\RD_{g,k}^W$ is nonzero only if
\begin{align}\label{eqn-virdim}
\sum_{j=1}^k\deg \alpha_j+\sum_{j=1}^k \ell_j=(\hat c_W-3)(1-g)+k.
\end{align}
   \item{\upshape (A consequence of the string equation)} A primary  correlator  {\upshape$\LD  \mbox{id}, \alpha_2, \dots, \alpha_k\RD_{g,k}^W=0$} unless $(g, k)=(0, 3)$, where   {\upshape id} denotes the unit of the FJRW ring $(H_{W, G}, \bullet)$.
\end{enumerate}
\end{lem}

\end{comment}

Another important tool is the Concavity Axiom \cite[Theorem 4.1.8]{FJR}.
Consider the universal $W$-structure $(\LL_1,\cdots,\LL_n)$ on the universal curve $\pi:\mathscr{C}\to\overline{\mathscr{W}}_{g,k}(\Gamma_{\ga_1,\dots,\ga_k})$.
\begin{equation}\label{cond:concave}
\textrm{If all } H_{\gamma_i} \textrm{ are\ narrow\ and\ } \pi_*(\oplus_{i=1}^{n} \LL_i)=0,
\end{equation}
then $R^1\pi_{*}(\oplus_{i=1}^{n}\LL_i)$ is a vector bundle of constant rank, denoted by $D$, and
\begin{equation}\label{top-euler}
[\overline{\mathscr{W}}_{g,k}(\Gamma_{\ga_1,\dots,\ga_k})]^{\rm vir}\cap\prod_{i=1}^{k}\be_{\ga_i}=(-1)^{D}c_{D}\left(R^1\pi_{*}(\oplus_{i=1}^{n}\LL_i)\right)\cap[\overline{\mathscr{W}}_{g,k}(\Gamma_{\ga_1,\dots,\ga_k})].
\end{equation}
It can be calculated by the orbifold Grothendieck-Riemann-Roch  formula \cite[Theorem 1.1.1]{C}.
As a consequence, if the codimention $D=1$, we have
\begin{equation}\label{eq:OGRR}
\Lambda^W_{0,4}(\be_{\ga_1},\dots,\be_{\ga_4})=
\sum_{i=1}^n\left(\frac{B_{2}(q_i)}{2}\kappa_1
-\sum_{j=1}^4\frac{B_{2}(\Theta_{i}^{\ga_j})}{2}\psi_j
+\sum_{\Gamma_{\rm cut}}
\frac{B_{2}(\Theta_{i}^{\ga_{\Gamma_{\rm cut}}})}{2}[\Gamma_{\rm cut}]\right).
\end{equation}
Here $B_2(x):=x^2-x+\frac{1}{6}$ is the second Bernoulli polynomial. $\kappa_1$ is the $1$-st kappa class on $\overline{\mathcal{M}}_{0,4}$. Here the graphs $\Gamma_{\rm cut}$ are fully $G_W$-decorated on the boundary of $\overline{\mathscr{W}}_{0,4}(\ga_1,\dots,\ga_4)$. Each $\Gamma_{\rm cut}$ has exactly one edge which seperates the graph into two components. Two sides of the edge are decorated by some $\gamma_+\in G_W$ and $\gamma_-:=(\gamma_+)^{-1}$ such that each component of $\Gamma_{\rm cut}$ satisfies the line bundle criterion \eqref{eq:line bdle}. Finally, $[\Gamma_{\rm cut}]$ denotes the boundary class in $H^*(\overline{\mathcal{M}}_{0,4},\C)$ that corresponds to the underlying undecorated graph of $\Gamma_{\rm cut}.$

We call a correlator   \emph{concave} if it satisfies \eqref{cond:concave}. Otherwise we call it is \emph{nonconcave}. Nonconcave correlator may contain broad sectors. In this paper, we will use WDVV to compute the  nonconcave correlators. Some other methods are described in \cite{CLL, Gu}.

\subsection{FJRW invariants}\label{sec:ring-iso}
In this subsection, we will prove Proposition \ref{thm:FJRW}. Let us first describe the construction of the mirror polynomial $W^T$.
Let $W=M_1+\dots+M_{n}$, with $M_i=\prod_{j=1}^{n}x_j^{a_{ij}}.$ We call such a polynomial $W$ {\em invertible} because its {\em exponent matrix} $E_W:=\big(a_{ij}\big)$ is invertible. Berglund and H\"ubsch \cite{BH} introduced a mirror polynomial $W^T$,
\begin{equation}\label{mirror-poly}
W^T:=\sum_{i=1}^{n}\prod_{j=1}^{n}x_j^{a_{ji}}.
\end{equation}
Its exponent matrix $E_{W^T}$ is just the transpose matrix of $E_W$, i.e. $E_{W^T}=(E_W)^T$.
In \cite{KreS}, Kreuzer and Skarke proved that every invertible $W$ is a direct sum of three {\em atomic types} of singularities: Fermat, chain and loop. If $W$ is of atomic type, then $W^T$ belongs to the same atomic type.
We list the three atomic types (with $q_i\leq\frac{1}{2}$) and a $\mathbb{C}$-basis of their Jacobi algebra as follows. The table also contains an element $\phi_{\mu}$ of highest degree.
    \begin{table}[h]
        \caption{\label{invertible-singularities}  Invertible singularities
    }
        \begin{tabular}{|c|c|c|c|}
       \hline
 & Polynomial $f$ & $\mathbb{C}$-basis of ${\rm Jac}(f)$ & $\phi_{\mu}$  \\ \hline \hline
       $m$-Fermat & $x_1^{a_1}+\dots+x_m^{a_m}$                                  &  $\prod_{i=1}^mx_i^{k_i}, k_i< a_i-1$  & $\prod_{i=1}^mx_i^{a_i-2}$  \\   \hline
       $m$-Chain:  & $x_1^{a_1}x_2+x_2^{a_2}x_3+\dots+x_m^{a_m}$      & $\{\prod_{i=1}^mx_i^{k_i}\}_{\mathbf{k}}$ &$x_1^{a_1-2}\prod_{i=2}^mx_i^{a_i-1}$   \\  \hline
       $m$-Loop:   & $x_1^{a_1}x_2+x_2^{a_2}x_3+\dots+x_m^{a_m}x_1$ & $\prod_{i=1}^mx_i^{k_i}, k_i< a_i$  &$\prod_{i=1}^mx_i^{a_i-1}$    \\ \hline

     \end{tabular}
\end{table}

\noindent Here in the case of $m$-Chain, $\mathbf{k}=(k_1,\cdots,k_m)$ satisfies (1)
$k_j\leq  a_j-1$ for all $j$ and (2) the property that   $\mathbf{k}$  is not of the form $ (a_1-1,0,a_{3}-1,0,\cdots, a_{2l-1}-1,i,*,\cdots,*)$ with $i\geq 1$.
%%%\end{array}\right\}}$.
%%%$\Big\vert
%%\begin{array}{l}
%%%k_i< a_i, (k_1,\cdots, k_m)\ \mbox{is not of the form}\  (a_1-1,0,a_{3}-1,0,\cdots, a_{2l-1}-1,i,*,\cdots,*), i\geq 1.
%%%\end{array}\right\}}$.

A first step towards the LG-LG mirror symetry Conjecture \ref{conj} is a ring isomorphism between $(H_{W},\bullet)$ and ${\rm Jac}(W^T)$. For computation convenience later, we use the following normalized residue defined by the normalized residue pairing $\eta_{W^T}$ (which is to be explained in \eqref{normalized-residue22})
\begin{equation}\label{normalized-residue}
\widetilde{\rm Res}_{W^T}(\phi_\mu):=\eta_{W^T}(dx_1\cdots dx_n, \phi_\mu dx_1\cdots dx_n)=1.
\end{equation}

The ring isomorphism has been studied  in   \cite{FJR, A, FS, K, KP+} for various  examples.
According to the Axiom of Sums of singularities \cite[Theorem 4.1.8 (8)]{FJR} in FJRW theory, the FJRW ring $(H_W,\bullet)$ is a tensor product of the FJRW ring of each direct summand.
Krawitz constructed a ring isomorphism for each atomic type if all $q_i<\frac{1}{2}$ \cite{K}.   For our purpose, if $W$ is a polynomial in Table \ref{tab-exceptional-singularities}, then it is already known that $(H_{W},\bullet)$ is isomorphic to ${\rm Jac}(W^T)$ except for $W=x^2+xy^q+yz^r$, $(q,r)=(3,3), (2,4)$. We will  give the new constructions for the  two exceptional cases, and will also  briefly introduce the earlier constructions for the other 12 cases.

Since $E_W$ is invertible, we can write $E_W^{-1}$ using column vectors $\rho_k$,
$$E_W^{-1}=\Big(\rho_1 | \cdots | \rho_n\Big), \quad \rho_k:=
\left(\varphi_1^{(k)},\cdots, \varphi_n^{(k)}\right)^T, \quad \varphi_i^{(k)}\in\mathbb{Q}.$$
We can view $\rho_{k}$ as an element in $G_{W}$ by defining the action
$$\rho_{k}=(\exp(2\pi\sqrt{-1}\varphi_1^{(k)}), \cdots, \exp(2\pi\sqrt{-1}\varphi_n^{(k)}))\in G_W.$$
Thus $\rho_i J\in G_W$, with $J$ the exponential grading element in \eqref{exp-grading}.
\begin{prop}[\cite{K}]\label{Krawitz}
For any $n$-variable invertible polynomial $W$ with each degree $q_i<\frac{1}{2}$, there is a degree-preserving ring isomorphism $\Psi: \Jac(W^T)\to(H_W,\bullet)$. In particular, if   $\rho_i J$ is narrow for $i=1,\cdots, n$, then $\Psi$ is generated by
\begin{equation}\label{ring-iso}
\Psi(x_i)={\bf 1}_{\rho_i J}, \quad i=1,\cdots, n
\end{equation}

\end{prop}

\begin{eg}
Let $W=x^p+y^q$, $p,q>2$.
Denote $\gamma_{i,j}=\left(\exp({2\pi\sqrt{-1}\,i\over p}), \exp({2\pi\sqrt{-1}\,j\over q})\right)$.
 The FJRW ring  $(H_W,\bullet)$ is generated by $\{\be_{\gamma_{2,1}},\be_{\gamma_{1,2}}\}$. Then $W^T=W$ and the ring isomorphism $\Psi: \Jac(W^T)\overset{\cong}{\to}(H_W,\bullet)$ generated by \eqref{ring-iso} extends as
  \begin{align}\label{mirror-2fermat}
     \Psi(x^{i-1}y^{j-1})=\be_{\gamma_{i,j}}, \quad 1\leq i< p,1\leq j< q.
  \end{align}
\end{eg}
For 2-Loop singularities, $\rho_i J$ may not be narrow for some $i\in\{1,2\}$. However, ring isomorphisms still exist. According to \cite{A,K}, we have
\begin{eg}
For $W=x^2y+xy^3+z^3\in Q_{12}$, $G_W\cong\mu_{15}$.
A ring isomorphism $\Psi: \Jac(W^T)\overset{\cong}{\to}(H_W,\bullet)$ is obtained by extending \eqref{ring-iso} from
\begin{equation}\label{iso:Q12}
\Psi(x)=x\be_{J^{10}}.
\end{equation}
The corresponding vector space isomorphism $\Psi:{\rm Jac}(W^T)\to H_W$ is as follows:
  \begin{table}[h]
         %\caption{\label{mirror-Q12} $\Psi: \Jac(W^T)\overset{\cong}{\to}(H_W,\bullet)$  with $W^T\in Q_{12}$}
      \resizebox{\columnwidth}{!}{
\begin{tabular}{|c||c|c|c|c|c|c|c|c|c|c|c|c|}
   \hline
 $H_W$  & $\be_{J}$ & $\be_{J^{13}}$ & $\be_{J^{11}}$  & $x\be_{J^{10}}$ & $y^2\be_{J^{10}}$  & $\be_{J^{8}}$ & $\be_{J^7}$ & $x\be_{J^5}$ & $y^2\be_{J^5}$ & $\be_{J^4}$  & $\be_{J^{2}}$  & $\be_{J^{14}}$ \\
   \hline
 ${\rm Jac}(W^T)$  & $1$ & $y$ & $z$ & $x$ & $y^2$ & $yz$ & $xy$ & $xz$ & $y^2z$ & $xy^2$ & $xyz$ & $xy^2z$  \\
   \hline
 \end{tabular}
 }
 \end{table}
\end{eg}

Now we discuss if there exists $q_i=\frac{1}{2}$ for $W$. Without loss of  generality, we assume $W$ is of the atomic type:
$W=x_1^2+x_1x_2^{a_2}+\cdots+x_{m-1}x_m^{a_m}$. Then ${\rm Fix}(\rho_1 J)=\{(x_1,\cdots,x_m)\in\mathbb{C}^k|x_i=0, i>2\}$. Thus $H_{\rho_1 J}$ is generated by a broad element $x_2^{a_2-1}\be_{\rho_1 J}$, which is a ring generator of $H_W$. If $m=2$, it is known   \cite{FS} that $\Psi:{\rm Jac}(W^T)\to (H_W,\bullet)$ generates a ring isomorphism, by $\Psi(x_1)=a_2x_2^{a_2-1}\be_{\rho_1 J}$ and $\Psi(x_2)=\be_{\rho_2 J}$. The key point is that the residue formula in \eqref{eq:pair} implies
$$\LD x_2^{a_2-1}\be_{\rho_1 J}, x_2^{a_2-1}\be_{\rho_1 J}, \be_{\rho_2^{1-a_2}J^{-1}}\RD_0^W
=\LD x_2^{a_2-1}\be_{\rho_1 J}, x_2^{a_2-1}\be_{\rho_1 J}, \be_{J}\RD_0^W=-\frac{1}{a_2}.$$
Inspired from this, for $m\geq3$, we consider
$$K:=\LD x_2^{a_2-1}\be_{\rho_1 J}, x_2^{a_2-1}\be_{\rho_1 J}, \be_{\rho_2^{1-a_2}\rho_{3}^{-1}J^{-1}}\RD_0^W.$$
If $K\neq0$, then it is possible to define
\begin{equation}\label{broad-gen}
\Psi(x_1)=\left(\sqrt{-\frac{a_2}{K}}\right)\,x_2^{a_2-1}\be_{\rho_1 J}.
\end{equation}
In Section \ref{sec:non-K}, using Getzler's relation, we will prove the following nonvanishing lemma,
\begin{lem}\label{lm:non-Krawitz}
Let $W=x^2+xy^q+yz^r, (q,r)=(3,3), (2,4)$. Then
\begin{equation}\label{broad3}
K_{q,r}:=\LD y^{q-1}\be_{\rho_1 J}, y^{q-1}\be_{\rho_1 J}, \be_{\rho_2^{1-q}\rho_{3}^{-1}J^{-1}}\RD_0^W\neq0.
\end{equation}
\end{lem}
As a direct consequence of Lemma \ref{lm:non-Krawitz}, it is not hard to check the following statement.
\begin{prop}\label{non-Krawitz}
Let $W^T$ be one of the exceptional unimodular singularities in Table \ref{tab-exceptional-singularities}, then the map $\Psi$ in \eqref{ring-iso} and \eqref{broad-gen} generates a degree-preserving ring isomorphism
$$\Psi:{\rm Jac}(W^T)\cong (H_W,\bullet).$$
\end{prop}
\begin{proof}
We only need to consider $W=x^2+xy^q+yz^r, (q,r)=(3,3), (2,4)$. We will check that $\Psi$ gives a vector space isomorphism which preserves the degree and the pairing on both side. We will also check that the generators in $H_W$ satisfy exactly the algebra relations as in $\Jac(W^T)$, by computing all the genus-0, 3-point correlators. We remark that we use the normalized residue in $\Jac(W^T)$, i.e.,
$$\widetilde{\rm Res}_{W^T}(y^{q-1}z^{r-1})=1.$$
Lemma \ref{lm:non-Krawitz} allows us to
extend $\Psi$ by defining $\Psi(x)$ as in \eqref{broad-gen}.
Then we can check directly that
$$\Psi(x)\bullet\Psi(x)=\left(-\frac{q}{K_{q,r}}\LD y^{q-1}\be_{\rho_1 J}, y^{q-1}\be_{\rho_1 J}, \be_{\rho_2^{1-q}\rho_{3}^{-1}J^{-1}}\RD_0^W\right)\be_{\rho_2^{q-1}\rho_{3}J}
=-q\Psi(y^{q-1}z).$$
This coincides with $x^2+qy^{q-1}z=0$ in $\Jac(W^T)$. We notice that the multiplication $\Psi(x)\bullet\Psi(z)$ can be computed via
$$\LD\Psi(x),\Psi(x), \Psi(z^{r-2})\RD_{0,3}^W=\LD\Psi(x)\bullet\Psi(x), \Psi(z^{r-2})\RD=-q.$$
For $r=4,$ we use the WDVV equation once to get $\Psi(x)\bullet\Psi(z)$. The preimages of the broad sectors are in the form of $cxz^{j}$, $j=1,\cdots, r-2$, where the constant $c$ is fixed by the constant in \eqref{broad3} and the normalized residue pairing.

We have $\Psi(x)\bullet\Psi(y)=0$ by simply checking the formula \eqref{selection-non-primary}. This coincides with $xy=0$ in $\Jac(W^T)$.

The rest of the proof are the same as that in Lemma 4.5, Lemma 4.6, and Lemma 4.7 in \cite{K}. For the reader's convenience, we sketch a proof here for $W=x^2+xy^3+yz^3$. The other case can be treated similarly. By \eqref{ring-iso}, we get
$$\Psi(y)=\be_{J^{15}}, \quad \Psi(z)=\be_{J^{13}}.$$
According to \eqref{selection-non-primary}, the nonzero $\LD\cdots\RD_{0,3}^W$ with narrow insertions only is one of following:
\begin{equation}\label{pairing-233}
\LD\be_{J}, \be_{J^j}, \be_{J^{18-j}}\RD_{0,3}^W, \quad j \text{\ is odd}
\end{equation}
or
\begin{equation}\label{3-point-233}
\LD\be_{J^{15}}, \be_{J^{15}}, \be_{J^{7}}\RD_{0,3}^W, \quad \LD\be_{J^{15}}, \be_{J^{13}}, \be_{J^{9}}\RD_{0,3}^W, \quad \LD\be_{J^{13}}, \be_{J^{13}}, \be_{J^{11}}\RD_{0,3}^W, \quad \LD\be_{J^{15}}, \be_{J^{11}}, \be_{J^{11}}\RD_{0,3}^W.
\end{equation}
All the correlators listed above are concave. Furthermore, we apply \eqref{eq:line bdle} to get the line bundle degrees. Except for the last correlator in \eqref{3-point-233}, we have
$$\deg(\rho_*\mathscr{L}_i)=-1, \quad i=1,2,3$$
This implies all the bundles $R^1\pi_{*}(\oplus_{i=1}^{n}\LL_i)$ have rank zero.  Applying \eqref{top-euler} for $D=0$, the values of those correlators all equal to 1.
We use those correlators to get, for example,
$$\Psi(y)\bullet\Psi(y)=\LD\be_{J^{15}},\be_{J^{15}},\be_{J^{7}}\RD_{0,3}^W\,\eta^{\be_{J^7},\be_{J^{11}}}\,\be_{J^{11}}=\be_{J^{11}}.$$
Here $\eta^{-,-}$ is defined in \eqref{WDVV}. Similarly, we obtain
$$\Psi(yz)=\be_{J^{9}}, \quad \Psi(z^2)=\be_{J^{7}}, \quad \Psi(y^2z)=\be_{J^{5}}, \quad \Psi(yz^2)=\be_{J^{3}}, \quad \Psi(y^2z^2)=\be_{J^{17}}.$$
The correlators in \eqref{pairing-233} match the normalized residue pairing. For the last correlator
$\LD\be_{J^{15}}, \be_{J^{15}}, \be_{J^{11}}\RD_{0,3}^W,$ we have
$$\deg(\rho_*\mathscr{L}_1)=-1, \quad \deg(\rho_*\mathscr{L}_2)=-2, \quad \deg(\rho_*\mathscr{L}_3)=0.$$
Thus for each fiber (isomorphic to $\mathbb{CP}^1$) of the universal curve $\mathscr{C}$ over $\overline{\mathscr{W}}_{0,3}(J^{15}, J^{15}, J^{11})$, we have
$$H^0(\mathbb{CP}^1,\bigoplus\mathscr{L}_i)=0\oplus0\oplus\C, \quad H^1(\mathbb{CP}^1, \bigoplus\mathscr{L}_i)=0\oplus\C\oplus0.$$
According to the Index Zero axiom in Theorem 4.1.8 \cite{FJR}, this corrlelator equals to the degree of the so-called Witten map form $H^0$ to $H^1$, which sends $(x,y,z)$ to $(\overline{{\partial W\over \partial x}},\overline{{\partial W\over \partial y}}, \overline{{\partial W\over \partial z}})$. In this case, we get
$$\LD\be_{J^{15}}, \be_{J^{11}}, \be_{J^{11}}\RD_{0,3}^W=-3.$$
From this, we check that
$$\Psi(y)\bullet\Psi(y^2)=-3\Psi(z^2).$$
This coincides with the last relation in $\Jac(W^T)$, i.e., $y^3+3z^2=0.$

Finally, we list the table for each vector space isomorphism.

If
 $W=x^2+xy^3+yz^3$, the vector space isomorphism $\Psi:\Jac(W^T)\overset{\cong}{\to}H_W$ is
  \begin{table}[H]
         %\caption{\label{mirror-3-chain-233} $\Psi: \Jac(W^T)\overset{\cong}{\to}(H_W,\bullet)$  with $W^T$ of type $Q_{11}$ or $S_{11}$}
         \resizebox{\columnwidth}{!}{
\begin{tabular}{|c||c|c|c|c|c|c|c|c|c|c|c|}
   \hline
  $H_W$  & $\be_{J}$ &  $\be_{J^{15}}$ & $\be_{J^{13}}$& $\sqrt{-\frac{3}{K_{3,3}}}y^2\be_{J^{12}}$ & $\be_{J^{11}}$  & $\be_{J^9}$ & $\be_{J^7}$& $\sqrt{-3^3K_{3,3}}y^2\be_{J^{6}}$ & $\be_{J^5}$ & $\be_{J^3}$   & $\be_{J^{17}}$   \\
   \hline
   ${\rm Jac}(W^T)$  & $1$ & $y$ & $z$ & $x$ & $y^2$  & $yz$ & $z^2$ & $xz$ & $y^2z$ & $yz^2$ & $y^2z^2$  \\
   \hline
 \end{tabular}
 }
 \end{table}

If $W=x^2+xy^2+yz^4$, then the vector space isomorphism is given by
  \begin{table}[H]
         %\caption{\label{mirror-3-chain-233} $\Psi: \Jac(W^T)\overset{\cong}{\to}(H_W,\bullet)$  with $W^T$ of type $Q_{11}$ or $S_{11}$}
         \resizebox{\columnwidth}{!}{
\begin{tabular}{|c||c|c|c|c|c|c|c|c|c|c|c|}
   \hline
   $H_W$  & $\be_{J}$ & $\be_{J^{13}}$ & $\sqrt{-\frac{2}{K_{2,4}}}y\be_{J^{12}}$ & $\be_{J^{11}}$ & $\be_{J^9}$  & $y\be_{J^{8}}$ & $\be_{J^7}$ & $\be_{J^5}$ & $\sqrt{-2^3K_{2,4}}y\be_{J^4}$ & $\be_{J^3}$  & $\be_{J^{15}}$   \\
   \hline
 ${\rm Jac}(W^T)$  & $1$ & $z$ & $x$ & $y$ & $z^2$ & $xz$ & $yz$ & $z^3$ & $xz^2$ & $yz^2$ & $yz^3$  \\
   \hline
 \end{tabular}
 }
 \end{table}

\end{proof}

We will give explicit formulas of the isomorphism $\Psi$ of all other cases in the appendix. Those isomorphisms $\Psi$ turn out to identify the ancestor total potential of the FJRW theory of $(W,G_W)$ with that of Saito-Givental theory of $W^T$ up to a rescaling.

Next we compute the FJRW invariants in Proposition \ref{thm:FJRW}.
We introduce a new notation
\begin{equation}\label{ring-notation}
\be_{\phi}:=\Psi(\phi), \quad \phi\in \Jac(W^T).
\end{equation}
Due to the above conventions, the second part of Proposition \ref{thm:FJRW} is simplified as follows.
\begin{comment}
\noindent\textbf{Proposition \ref{thm:FJRW}. }{\itshape   Let $M_i^T$  be the $i$-th monomial of $W^T$ with the ordering in Table \ref{tab-exceptional-singularities}. We have
\begin{equation}\label{4-point-prop}
\LD \be_{x_i},\be_{x_i},\be_{M_i^T/x_i^2},\be_{\phi_\mu}\RD_{0,4}^{W}={q_i}, \quad \forall i=1, \cdots, n.
\end{equation}
 }
\end{comment}
\begin{prop}\label{thm-reformulathm1}
Let $M_i^T$  be the $i$-th monomial of $W^T$ with the ordering in Table \ref{tab-exceptional-singularities}. We have
\begin{equation}\label{4-point-prop}
\LD \be_{x_i},\be_{x_i},\be_{M_i^T/x_i^2},\be_{\phi_\mu}\RD_{0}^{W}={q_i}, \quad \forall i=1, \cdots, n.
\end{equation}
\end{prop}
\begin{proof}
We classify all the correlators in \eqref{4-point-prop} into concave correlators and nonconcave correlators. For the concave correlators, we use \eqref{eq:OGRR} to compute. For the nonconcave correlators, we use WDVV to reconstruct them from concave correlators and again use \eqref{eq:OGRR}.  We will freely interchange the notation
\begin{align}
      (x_1, x_2, x_3)=(x, y, z).
\end{align}

Let us start with concave correlators. As an example, we compute $\LD\be_{x},\be_x,\be_{x^{p-2}},\be_{\phi_\mu}\RD_0^W$ for $W=x^p+y^q$. The computation of all the other concave corrlators in \eqref{4-point-prop} follows similarly. For $W=x^p+y^q$, we recall that for
$\gamma_{i,j}\in G_W\cong\mu_p\times\mu_q,$ we have
$\Theta_1^{\gamma_{i,j}}=\frac{i}{p}, \Theta_2^{\gamma_{i,j}}=\frac{j}{q}.$  All the sectors are narrow and $\be_{\gamma_{i,j}}=\be_{x^{i-1}y^{j-1}}$ with our notation conventions. According to the line bundle criterion \eqref{eq:line bdle}, we know for $\LD\be_{x},\be_x,\be_{x^{p-2}},\be_{\phi_\mu}\RD_0^W$,  $$\deg\rho_*\LL_1=-2, \deg\rho_*\LL_2=-1.$$  Thus $\pi_*\LL_1=\pi_*\LL_2=0$ and the correlator is concave. Moreover, $R^1\pi_*\LL_2=0$ and the nonzero contribution of the virtual cycle only comes from $R^1\pi_*\LL_1$.
Now we can apply \eqref{eq:OGRR}.
There are three decorated dual graphs in $\mathbf{\Gamma}_{\rm cut}$, where we simply denote $\be_{i,j}:=\be_{\gamma_{i,j}}$,
\begin{figure}[h]
\centering
\setlength{\unitlength}{1.0mm}
\begin{picture}(120,20)

	% Lines
	\put(0,9){\line(-3,4){5}}
    \put(0,9){\line(-3,-4){5}}
	\put(0,9){\line(1,0){20}}
    \put(20,9){\line(3,4){5}}
    \put(20,9){\line(3,-4){5}}
	%\multiput(20,-.5)(0,1){2}{\line(1,0){10}}
	%\put(30,0){\line(1,0){10}}
	
	% Labels
	\put(-6,17){$\be_{2,1}$}
	\put(-6,0){$\be_{2,1}$}
	\put(24,17){$\be_{p-1,1}$}
	\put(24,0){$\be_{p-1,q-1}$}
	\put(1,11){$\ga_{\Gamma_1}$}
	\put(13,11){$\ga_{\Gamma_1}^{-1}$}

	% Lines
	\put(50,9){\line(-3,4){5}}
    \put(50,9){\line(-3,-4){5}}
	\put(50,9){\line(1,0){20}}
    \put(70,9){\line(3,4){5}}
    \put(70,9){\line(3,-4){5}}
	%\multiput(20,-.5)(0,1){2}{\line(1,0){10}}
	%\put(30,0){\line(1,0){10}}
	
	% Labels
	\put(44,17){$\be_{2,1}$}
	\put(44,0){$\be_{p-1,1}$}
	\put(74,17){$\be_{2,1}$}
	\put(74,0){$\be_{p-1,q-1}$}
	\put(51,11){$\ga_{\Gamma_2}$}
	\put(63,11){$\ga_{\Gamma_2}^{-1}$}

	% Lines
	\put(100,9){\line(-3,4){5}}
    \put(100,9){\line(-3,-4){5}}
	\put(100,9){\line(1,0){20}}
    \put(120,9){\line(3,4){5}}
    \put(120,9){\line(3,-4){5}}
	%\multiput(20,-.5)(0,1){2}{\line(1,0){10}}
	%\put(30,0){\line(1,0){10}}
	
	% Labels
	\put(94,17){$\be_{2,1}$}
	\put(94,0){$\be_{p-1,q-1}$}
	\put(124,17){$\be_{2,1}$}
	\put(124,0){$\be_{p-1,1}$}
	\put(101,11){$\ga_{\Gamma_3}$}
	\put(113,11){$\ga_{\Gamma_3}^{-1}$}

\end{picture}
\end{figure}

The decorations of the boundary classes are $\Theta_1^{\ga_{\Gamma_i}}=\frac{p-3}{p},0,0$ for $i=1,2,3$. We obtain
\begin{eqnarray*}
&&\LD\be_{2,1},\be_{2,1},\be_{p-1,1},\be_{p-1,q-1}\RD_{0}^W\\
&=&\int_{\overline{\mathcal{M}}_{0,4}}\Lambda_{0,4}^W(\be_{2,1},\be_{2,1},\be_{p-1,1},\be_{p-1,q-1})\\
&=&\frac{1}{2}\left(B_2(\frac{1}{p})-2B_2(\frac{2}{p})-2B_2(\frac{p-1}{p})+2B_2(0)+B_2(\frac{p-3}{p})\right)\\
&=&\frac{1}{p}.
\end{eqnarray*}

All the nonconcave correlators in \eqref{4-point-prop} are listed as follows:
\begin{itemize}

\item $\LD\be_y,\be_y,\be_{z},\be_{\phi_\mu}\RD_{0}^{W}$ for 3-Chain $W=x^2+xy^2+yz^4$.

\item $\LD\be_{x},\be_x,\be_{y},\be_{\phi_\mu}\RD_{0}^{W}$ for 3-Chain $W=x^2+xy^q+yz^r$, $(q,r)=(3,3)$ or $(2,4)$.

\item $\LD \be_x,\be_x,\be_y,\be_{\phi_{\mu}}\RD_{0}^{W}$ and $\LD \be_y,\be_y,\be_z,\be_{\phi_\mu}\RD_{0}^{W}$ for 3-Loop $W=x^2z+xy^2+yz^3$.

\item $\LD\be_{x},\be_{x},\be_{y},\be_{\phi_\mu}\RD_{0}^{W}$ for $W=x^2+xy^4+z^3$.

\item $\LD\be_x,\be_x,\be_y,\be_{\phi_{\mu}}\RD_{0}^{W}$ for $W=x^2y+xy^3+z^3$.

\item $\LD \be_x,\be_y,\be_y,\be_{\phi_{\mu}}\RD_{0}^{W}$ for $W=x^2y+y^2+z^4$.
\end{itemize}

For the nonconcave correlators, we will use the WDVV equations and the ring relations to reconstruct them from concave correlators. Let us start with the value of $\LD\be_y,\be_y,\be_{y^{q-2}z},\be_{\phi_\mu}\RD_{0}^{W}$ in a 3-Chain $W=x^2+xy^2+yz^4$. Since $\phi_\mu=yz^3\in {\rm Jac}(W^T)$ and $\be_y\bullet\be_{yz}=0$, we get
\begin{equation*}
\LD\be_z,\be_y,\be_{yz}\bullet\be_{z^2},\be_{y}\RD_{0}^{W}=\LD\be_z,\be_y,\be_{yz},\be_{z^2}\bullet\be_y\RD_{0}^{W}-\LD\be_z,\be_y\bullet\be_{y},\be_{yz},\be_{z^2}\RD_{0}^{W}=0-(-4)\frac{1}{16}=\frac{1}{4}.
\end{equation*}
The first equality follows from the WDVV equation\eqref{eq:WDVV}. We also use  $\be_y\bullet\be_{yz}=0$. Both $\LD\be_z,\be_y,\be_{yz},\be_{z^2}\bullet\be_y\RD_{0}^{W}$ and $\LD\be_z,\be_y\bullet\be_{y},\be_{yz},\be_{z^2}\RD_{0}^{W}$ are concave correlators and can be computed by \eqref{eq:OGRR}. For other nonconcave correlators, we will list the WDVV equations. The concavity computation is checked easily.
 For 3-Chain $W=x^2+xy^q+yz^r$, $(q,r)=(3,3)$ or $(2,4)$, $\be_{\phi_\mu}=\be_{y^{q-1}z^{r-1}}$.
$$\LD\be_{y},\be_x,\be_{\phi_\mu},\be_{x}\RD_{0}^{W}=-\LD\be_y,\be_x\bullet\be_{x},\be_y,\be_{y^{q-2}z^{r-1}}\RD_{0}^{W}=q\LD\be_{y},\be_{y^{q-1}z},\be_{y},\be_{y^{q-2}z^{r-1}}\RD_{0}^{W}=\frac{1}{2}.$$
For 3-Loop $W=x^2z+xy^2+yz^3$, $\be_{\phi_\mu}=\be_{xyz^2}$. We get
$$
\begin{array}{l}
\LD \be_y,\be_x,\be_{xy}\bullet\be_{z^2},\be_x\RD_{0}^{W}=\LD\be_y,\be_x,\be_{xy},\be_{z^2}\bullet\be_x\RD_{0}^{W}-\LD \be_y,\be_x\bullet\be_x,\be_{xy},\be_{z^2}\RD_{0}^{W}=\frac{1}{13}-(-2)\frac{2}{13}=\frac{5}{13}.
\\
\LD\be_z,\be_y,\be_z\bullet\be_{xyz},\be_y\RD_{0}^{W}=\LD\be_z,\be_y\bullet\be_z,\be_{xyz},\be_y\RD_{0}^{W}-\LD \be_z,\be_y\bullet\be_y,\be_z,\be_{xyz}\RD_{0}^{W}=\frac{1}{13}-(-3)\frac{1}{13}=\frac{4}{13}.
\end{array}
$$
For $W=x^2+xy^4+z^3$, $\be_{x}$ is broad. However,
$$\LD\be_{y},\be_{x},\be_{\phi_\mu},\be_{x}\RD_{0}^{W}=-\LD\be_y,\be_x\bullet\be_x,\be_y,\be_{y^{2}z}\RD_{0}^{W}=4\LD\be_y,\be_{y^{3}},\be_y,\be_{y^{2}z}\RD_{0}^{W}=\frac{1}{2}.$$
For $W=x^2y+xy^3+z^3$, we get
$$
\begin{array}{l}
\LD \be_y,\be_x,\be_{xy}\bullet\be_{yz},\be_x\RD_{0}^{W}+\LD\be_y,\be_x\bullet\be_{x},\be_{xy},\be_{yz}\RD_{0}^{W}
=\LD \be_y,\be_x,\be_{xy},\be_{yz}\bullet\be_{x}\RD_{0}^{W}\\
=-\frac{1}{2}\LD \be_y,\be_x,\be_{y}\bullet\be_{y^2z},\be_{xy}\RD_{0}^{W}
=-\frac{1}{2}\LD \be_y,\be_{xy},\be_{xy},\be_{y^2z}\RD_{0}^{W}=-\frac{1}{2}\LD \be_y,\be_{xy},\be_{y}\bullet\be_{yz},\be_{xy}\RD_{0}^{W}\\
=-\LD\be_y,\be_{xy},\be_{yz},\be_{xy^2}\RD_{0}^{W}.
\end{array}
$$
The first, third and last equalities are WDVV equations. Finally, we get
$$\LD\be_x,\be_x,\be_y,\be_{xy^2z}\RD_{0}^{W}=-\LD\be_y,\be_x\bullet\be_{x},\be_{xy},\be_{yz}\RD_{0}^{W}-\LD\be_y,\be_{xy},\be_{yz},\be_{xy^2}\RD_{0}^{W}=-(-\frac{1}{5})-(-\frac{1}{5})=\frac{2}{5}.$$
For $W=x^2y+y^2+z^4$, we get
$$
\begin{array}{l}
\LD \be_x,\be_y,\be_{xy}\bullet\be_{z^2},\be_y\RD_{0}^{W}
=\LD\be_x,\be_y,\be_{xy},\be_{z^2}\bullet\be_y\RD_{0}^{W}-\LD \be_x,\be_y\bullet\be_y,\be_{xy},\be_{z^2}\RD_{0}^{W}\\
=\left(\LD\be_y,\be_x,\be_x,\be_{y}\bullet\be_{yz^2}\RD_{0}^{W}
-\LD\be_y,\be_{x}\bullet\be_{yz^2},\be_x,\be_y\RD_{0}^{W}\right)-\LD \be_x,\be_y\bullet\be_y,\be_{xy},\be_{z^2}\RD_{0}^{W}.
\end{array}
$$
Combining  this equation and $y^2=-2x$, we get
$$
\LD \be_x,\be_y,\be_{xyz^2},\be_y\RD_{0}^{W}=
-\LD\be_y,\be_x,\be_x,\be_{xz^2}\RD_{0}^{W}
+\LD \be_x,\be_x,\be_{xy},\be_{z^2}\RD_{0}^{W}
=-(-\frac{1}{8})+\frac{1}{4}=\frac{3}{8}.
$$
\end{proof}

\subsection{Nonvanishing invariants}\label{sec:non-K}
In this subsection, we will prove Lemma \ref{lm:non-Krawitz}.
Our tool is the Getzler's relation \cite{Get}, which is a linear relation between codimension two cycles in $H_*(\overline{\mathcal{M}}_{1,4},\mathbb{Q})$. Let us briefly introduce this relation here. Consider the dual graph,
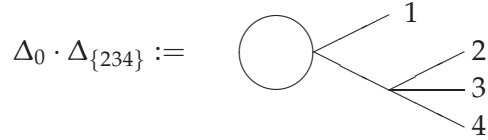
\begin{center}

\begin{picture}
(50,20)

	\put(-20,9){$\Delta_{0}\cdot\Delta_{\{234\}}:=$}

    % Circles
    \put(15,10){\circle{10}}

	% Lines

    \put(20,10){\line(2,1){10}}
    \put(20,10){\line(2,-1){10}}
    \put(30,5){\line(2,1){10}}
    \put(30,5){\line(2,-1){10}}

    \put(30,5){\line(1,0){10}}

% Markings
    \put(32,14){1}
    \put(41,9){2}
    \put(41,4){3}
    \put(41,-1){4}

\end{picture}
\end{center}
This graph represents a codimension-two stratum in $\overline{\mathcal{M}}_{1,4}$: A vertex represents a genus-0 component. An edge connecting two vertices (including a circle connecting the same vertex) represents a node, a tail (or half-edge) represents a marked point on the component of the corresponding vertex. Let $\Delta_{0,3}$ be the $S_4$-invariant of the codimension-two stratum in $\overline{\mathcal{M}}_{1,4}$,
\begin{equation*}
\Delta_{0,3}=\Delta_{0}\cdot\Delta_{\{123\}}+\Delta_{0}\cdot\Delta_{\{124\}}+\Delta_{0}\cdot\Delta_{\{134\}}+\Delta_{0}\cdot\Delta_{\{234\}}.
\end{equation*}
We denote $\delta_{0,3}=[\Delta_{0,3}]$ the corresponding cycle in $H_4(\overline{\mathcal{M}}_{1,4},\mathbb{Q})$. We list the corresponding unordered dual graph for other strata below. A filled circle (as a vertex) represents a genus-1 component. See \cite{Get} for more details.

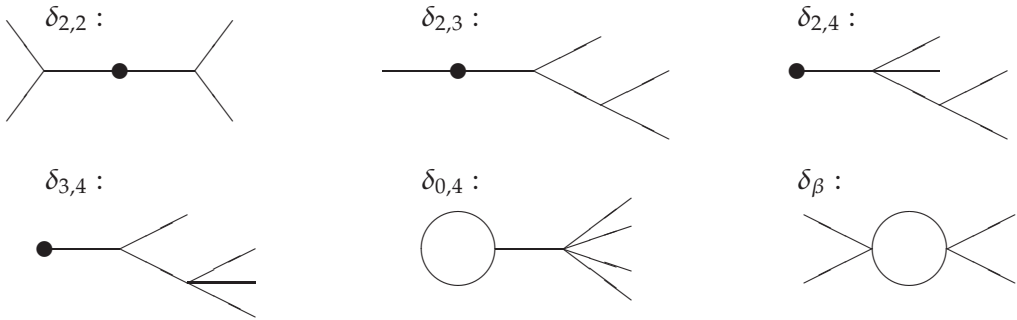
\begin{center}
\begin{picture}(50,20)

	\put(-40,15){$\delta_{2,2}:$}

    % Circles
    \put(-30,9){\circle*{2}}

	% Lines
	\put(-40,9){\line(-3,4){5}}
    \put(-40,9){\line(-3,-4){5}}
	\put(-40,9){\line(1,0){9}}
	\put(-29,9){\line(1,0){9}}
    \put(-20,9){\line(3,4){5}}
    \put(-20,9){\line(3,-4){5}}

	%\multiput(20,-.5)(0,1){2}{\line(1,0){10}}
	%\put(30,0){\line(1,0){10}}

    \put(15,9){\circle*{2}}
	\put(10,15){$\delta_{2,3}:$}

	% Lines
	\put(5,9){\line(1,0){9}}
	\put(16,9){\line(1,0){9}}
    \put(25,9){\line(2,1){9}}
    \put(25,9){\line(2,-1){9}}
    \put(34,4.5){\line(2,1){9}}
    \put(34,4.5){\line(2,-1){9}}
    % Circles

    \put(60,9){\circle*{2}}
	\put(60,15){$\delta_{2,4}:$}

	% Lines

	\put(61,9){\line(1,0){9}}
    \put(70,9){\line(2,1){9}}
	\put(70,9){\line(1,0){9}}
    \put(70,9){\line(2,-1){9}}
    \put(79,4.5){\line(2,1){9}}
    \put(79,4.5){\line(2,-1){9}}
	%\multiput(20,-.5)(0,1){2}{\line(1,0){10}}
	%\put(30,0){\line(1,0){10}}

\end{picture}
\end{center}

\begin{center}

\begin{picture}
(50,20)

	\put(-40,17){$\delta_{3,4}:$}
	\put(10,17){$\delta_{0,4}:$}
	\put(60,17){$\delta_{\beta}:$}

    % Circles
    \put(-40,9){\circle*{2}}

	% Lines

	\put(-21,4.5){\line(1,0){9}}
	\put(-39,9){\line(1,0){9}}
    \put(-30,9){\line(2,1){9}}
    \put(-30,9){\line(2,-1){9}}
    \put(-21,4.5){\line(2,1){9}}
    \put(-21,4.5){\line(2,-1){9}}

        % Circles
    \put(15,9){\circle{10}}

	% Lines

    \put(29,9){\line(4,3){9}}
    \put(29,9){\line(3,1){9}}
	\put(20,9){\line(1,0){9}}
    \put(29,9){\line(4,-3){9}}
    \put(29,9){\line(3,-1){9}}

	%\multiput(20,-.5)(0,1){2}{\line(1,0){10}}
	%\put(30,0){\line(1,0){10}}

    \put(75,9){\circle{10}}

	% Lines
    \put(80,9){\line(2,1){9}}
    \put(80,9){\line(2,-1){9}}
    \put(70,9){\line(-2,1){9}}
    \put(70,9){\line(-2,-1){9}}
\end{picture}
\end{center}

In \cite{Get}, Getzler found the following identity:
\begin{equation}\label{eq:Getzler}
  12\delta_{2,2}+4\delta_{2,3}-2\delta_{2,4}+6\delta_{3,4}+\delta_{0,3}+\delta_{0,4}-2\delta_{\beta}=0\in H_4(\overline{\mathcal{M}}_{1,4},\mathbb{Q}).
\end{equation}

\noindent{\bf Proof of Lemma \ref{lm:non-Krawitz}:}
We start with $W=x^2+xy^2+yz^4$. We normalize
$$u:=y\be_{J^{12}}, \quad v=\sqrt{-2}y\be_{J^{8}}, \quad w=-2y\be_{J^{4}}.$$
The nonvanishing pairings between these broad elements are $\LD u, w\RD=1, \LD v,v\RD=1.$

We integrate $\Lambda_{1,4}^W(\be_{J^9},\be_{J^9},\be_{J^9},\be_{J^9})$ over the Getzler's relation \eqref{eq:Getzler}.
The Composition law \cite[Theorem 4.1.8 (6)]{FJR} in FJRW theory implies
\begin{eqnarray*}
&&\int_{\delta_{0,3}}\Lambda_{1,4}^W(\be_{J^9},\be_{J^9},\be_{J^9},\be_{J^9})\\
&&=4\LD \be_{J^9},\be_{J^9},\be_{J^9},\be_{J^7}\RD_0^W\left(\sum_{\alpha,\beta}\eta^{\alpha,\beta}\LD\be_{J^9},\be_{J^9},\alpha,\beta\RD_0^W\right)\\
&&=4\LD \be_{J^9},\be_{J^9},\be_{J^9},\be_{J^7}\RD_0^W
\left(
\begin{array}{c}
2\LD\be_{J^9},\be_{J^9},\be_{J^{13}},\be_{J^3}\RD_0^W+2\LD\be_{J^9},\be_{J^9},\be_{J^{11}},\be_{J^5}\RD_0^W+\\
2\LD\be_{J^9},\be_{J^9},\be_{J^9},\be_{J^7}\RD_0^W+2\LD\be_{J^9},\be_{J^9},u,w\RD_0^W+\LD\be_{J^9},\be_{J^9},v,v\RD_0^W
\end{array}
\right).
\end{eqnarray*}
The factor $4$ comes from that there are $4$ strata in $\Delta_{0,3}$ which contribute. We have the factor $2$ for $\LD\be_{J^9},\be_{J^9},\be_{J^{13}},\be_{J^3}\RD_0^W$ since both $\alpha=\be_{J^{13}}$ and $\alpha=\be_{J^{3}}$ give the same correlator. Finally, $\be_{J}$ is the identity, and the string equation implies
$\LD\be_{J^9},\be_{J^9},\be_{J^{15}},\be_{J}\RD_0^W=0$.
There are two correlators contain broad sectors, we simply denote
$$C_1:=\LD\be_{J^9},\be_{J^9},v,v\RD_0^W, \quad C_2:=\LD\be_{J^9},\be_{J^9},u,w\RD_0^W.$$
We can calculate the concave correlators using orbifold-GRR formula in \eqref{eq:OGRR} and get
$$\LD\be_{J^9},\be_{J^9},\be_{J^{13}},\be_{J^3}\RD_0^W=\frac{1}{4}, \quad\LD\be_{J^9},\be_{J^9},\be_{J^{11}},\be_{J^5}\RD_0^W=-\frac{1}{8},\quad\LD\be_{J^9},\be_{J^9},\be_{J^9},\be_{J^7}\RD_0^W=\frac{1}{8}.$$
This implies
$$\int_{\delta_{0,3}}\Lambda_{1,4}^W(\be_{J^9},\be_{J^9},\be_{J^9},\be_{J^9})=C_2+\frac{C_1}{2}+\frac{1}{4}.$$
Similarly, we get
$$\int_{\delta_{\beta}}\Lambda_{1,4}^W(\be_{J^9},\be_{J^9},\be_{J^9},\be_{J^9})=6C_2^2+3C_1^2+\frac{9}{16},\quad \int_{\delta_{0,4}}\Lambda_{1,4}^W(\be_{J^9},\be_{J^9},\be_{J^9},\be_{J^9})=\frac{165}{128}.$$
The last equality requires the computation for a genus-$0$ correlator with $5$ marked points. It is reconstructed from some known $4$-point correlators by WDVV equations.
On the other hand, using the homological degree \eqref{degree-vir}, we conclude the vanishing of the integration of $\Lambda_{1,4}^W(\be_{J^9},\be_{J^9},\be_{J^9},\be_{J^9})$ over those strata which contain genus-1 component. Thus
$$\int_{12\delta_{2,2}+4\delta_{2,3}-2\delta_{2,4}+6\delta_{3,4}}\Lambda_{1,4}^W(\be_{J^9},\be_{J^9},\be_{J^9},\be_{J^9})=0.$$
Now apply Getzler's relation \eqref{eq:Getzler}, we get
\begin{equation}\label{gel}
-12C_2^2+C_2-6C_1^2+\frac{C_1}{2}+\frac{53}{128}=0.
\end{equation}
On the other hand, since $\be_{J^9}=\be_{J^{13}}\bullet\be_{J^{13}}$, we apply WDVV equations and get
\begin{equation*}
\left\{
\begin{array}{l}
\LD u,u,\be_{J^9}\RD_0^W=\Big(\LD\be_{J^{13}},u,v\RD_0^W\Big)^2,\\
\LD\be_{J^9},\be_{J^9},v,v\RD_0^W+\LD\be_{J^{13}},\be_{J^{13}},\be_{J^{9}},\be_{J^{15}}\RD_0^W=2\LD\be_{J^9},\be_{J^{13}},v,w\RD_0^W\LD\be_{J^{13}},u,v\RD_0^W,\\
\LD\be_{J^9},\be_{J^9},u,w\RD_0^W+\LD\be_{J^{13}},\be_{J^{13}},\be_{J^{9}},\be_{J^{15}}\RD_0^W=\LD\be_{J^9},\be_{J^{13}},v,w\RD_0^W\LD\be_{J^{13}},u,v\RD_0^W.
\end{array}
\right.
\end{equation*}
If $\LD u,u,\be_{J^9}\RD_0^W=0$, then $\LD\be_{J^{13}},u,v\RD_0^W=0$ and the rest two equations above implies
$$C_1=C_2=-\LD\be_{J^{13}},\be_{J^{13}},\be_{J^{9}},\be_{J^{15}}\RD_0^W=-\frac{3}{16},$$
where the last equality follows from \eqref{eq:OGRR}. However, this contradicts with formula \eqref{gel}.

Next we consider $W=x^2+xy^3+yz^3$. We denote
$$\left\{
\begin{array}{l}
u:=y^2\be_{J^{12}}, \quad w:=-3y^{2}\be_{J^6},\\
C_1:=\LD\be_{J^{13}},\be_{J^{13}},w,w\RD_0^W, \quad C_2:=\LD\be_{J^7},\be_{J^{13}},u,w\RD_0^W, \quad C_3:=\LD\be_{J^7},\be_{J^7},u,u\RD_0^W.
\end{array}
\right.$$
We integrate $\Lambda_{1,4}^W(\be_{J^{13}},\be_{J^{13}},\be_{J^7},\be_{J^7})$ over the Getzler's relation \eqref{eq:Getzler} and get
\begin{equation}\label{Q11-gel}
-8C_2^2-\frac{2C_2}{3}-2C_1C_3+\frac{8}{81}=0
\end{equation}
On the other hand, since $\be_{J^7}=\be_{J^{13}}\bullet\be_{J^{13}}$, the WDVV equations imply
\begin{equation*}
\left\{
\begin{array}{l}
\LD\be_{J^7},\be_{J^{13}},u,w\RD_0^W+\LD\be_{J^{13}},\be_{J^{13}},\be_{J^{13}},\be_{J^{17}}\RD_0^W=\LD\be_{J^{13}},\be_{J^{13}},w,w\RD_0^W\LD\be_{J^{13}},u,u\RD_0^W,\\
\LD\be_{J^7},\be_{J^7},u,u\RD_0^W=\LD\be_{J^7},\be_{J^{13}},u,w\RD_0^W\LD\be_{J^{13}},u,u\RD_0^W.
\end{array}
\right.
\end{equation*}
Now $\LD\be_{J^{13}},u,u\RD_0^W=0$ implies $C_2=-\frac{5}{18}$ and $C_3=0$. This contradicts with \eqref{Q11-gel}.
\qed

\section{B-model: Saito's theory of primitive form}\label{sec3}
Throughout this section, we consider the Landau-Ginzburg B-model defined by
$$
f: X=\C^n\to \C,
$$
where $f$ is a weighted homogeneous polynomial with isolated singularity at the origin:
$$
   f(\lambda^{q_1}x_1,\cdots, \lambda^{q_n}x_n)=\lambda f(x_1,\cdots, x_n).
$$
Recall that $q_i$ are called the weights of  $x_i$, and    the central charge of $f$ is defined by
$$
  \hat c_f=\sum_i(1-2q_i).
$$
Associated to $f$, the third author has introduced the concept of a primitive form  \cite{Saito-primitive}, which, in particular, induces  a Frobenius manifold structure (sometimes called a flat structure) on  the local universal deformation space of $f$.   This gives rise to the genus zero correlation functions in the Landau-Ginzburg B-model, which are conjectured to be equivalent to the FJRW-invariants on the mirror singularities.

The general existence of primitive forms for local isolated singularities is proved by M.Saito \cite{Saito-existence} via Deligne's mixed Hodge theory. In the case for $f$ being a weighted homogeneous polynomial, the existence problem is greatly simplified due to the semi-simplicity of the monodromy \cite{Saito-primitive, Saito-existence}. However,  explicit formulas of   primitive forms were only known for $ADE$ and simple elliptic singularities \cite{Saito-primitive} (i.e., for $\hat c_f\leq 1$). This led to the difficulty of computing  correlation functions in the Landau-Ginzburg B-model, and has become one of the main obstacles toward proving mirror symmetry between Landau-Ginzburg models.

Based on the recent idea of perturbative approach to primitive forms \cite{LLSaito}, in this section we will develop a general perturbative method to compute the Frobenius manifolds in the Landau-Ginzburg B-model. This is applied to the 14 exceptional unimodular singularities. With the help of certain reconstruction type theorem from the WDVV equation (see e.g. Lemma \ref{reconstruction-lemma}), it completely solves the computation problem in the Landau-Ginzburg B-model at genus zero.

\subsection{Higher residue and good basis}\label{subsection-higher-residue}
 Let $\mbf{0}\in X=\C^n$ be the origin. Let $\Omega^k_{X,\mbf{0}}$ be the germ of holomorphic $k$-forms at $\mbf{0}$. In this paper we will work with the following space \cite{Saito-residue}
$$
  \mc{H}_f^{(0)}:=\Omega^n_{X,\mbf{0}}[[z]]/(df+zd)\Omega^{n-1}_{X,\mbf{0}}
$$
which is a formally completed version of the \emph{Brieskorn lattice} associated to $f$. Given a differential form $\varphi \in \Omega^n_{X,\mbf{0}}$, we will use $[\varphi]$ to represent its class in $\mc H_f^{(0)}$.

 There is a natural semi-infinite Hodge filtration on $ \mc{H}_f^{(0)}$ given by
$\mc{H}_f^{(-k)}:=z^k \mc{H}_f^{(0)}$, with graded pieces
$$
    \mc{H}_f^{(-k)}/  \mc{H}_f^{(-k-1)} \iso \Omega_f, \quad \mbox{where}\ \Omega_f:= \Omega_{X,\mbf{0}}^n/ df\wedge \Omega_{X,\mbf{0}}^{n-1}.
$$
In particular, $\mc{H}_f^{(0)}$ is a free $\C[[z]]$-module of rank $\mu=\dim_{\C}\Jac(f)_o$, the Milnor number of $f$. We will also denote the extension to Laurent series by
$$
   \mc{H}_f:= \mc{H}_f^{(0)}\otimes_{\C[[z]]}\C((z)).
$$
There is a natural $\Q$-grading on $\mc{H}_f^{(0)}$ defined by assigning the degrees
$$
   \deg (x_i)=q_i, \quad \deg(dx_i)=q_i, \quad \deg (z)=1.
$$
Then for a homogeneous element of the form
$
  \varphi= z^k g(x_i) dx_1\wedge\cdots\wedge dx_n,
$
we have
$$
  \deg(\varphi)=\deg(g)+k+\sum_i q_i.
$$

In \cite{Saito-residue}, the third author constructed a \emph{higher residue pairing}
$$
   K_f: \mc{H}_f^{(0)}\otimes \mc{H}_f^{(0)}\to z^{n}\C[[z]]
$$
which satisfies the following properties
\begin{enumerate}
\item $K_f$ is equivariant with respect to the $\Q$-grading, i.e.,
$$
   \deg (K_f(\alpha, \beta))=\deg(\alpha)+\deg(\beta)
$$
for homogeneous elements  $\alpha, \beta\in \mc{H}_f^{(0)}$.
\item $K_f(\alpha, \beta)=(-1)^n \overline{K_f(\beta,\alpha)}$, where the $-$ operator takes $z\to -z$.
\item $K_f(v(z)\alpha, \beta)=K_f(\alpha, v(-z)\beta)=v(z)K_f(\alpha,\beta)$ for $v(z)\in \C[[z]]$.
\item The leading $z$-order of $K_f$ defines a pairing
$$
    \mc{H}_f^{(0)}/z  \mc{H}_f^{(0)}\otimes  \mc{H}_f^{(0)}/z  \mc{H}_f^{(0)} \to \C, \quad \alpha\otimes \beta\mapsto \lim_{z\to 0}z^{-n}K_f(\alpha,\beta)
$$
which coincides with the usual residue pairing
$$
   \eta_f: \Omega_f\otimes \Omega_f \to \C.
$$
\end{enumerate}
We remark that the classical residue pairing $\eta_f$ is intrinsically defined up to a nonzero constant.
In the case of weighted homogeneous singularities (for instance for the exceptional unimodular singularities), we will always specify a top degree element $\phi_\mu$ in a weighted homogeneous basis of $\Jac(f)$, and will fix the   constant  such that
  \begin{equation}  \label{normalized-residue22}
      \eta_f (dx_1\cdots dx_n, \phi_\mu dx_1\cdots dx_n)=1.
  \end{equation} We will call it the \textit{normalized} residue pairing.

The last property implies that $K_f$ defines a semi-infinite extension of the residue pairing, which explains the name ``higher residue". It is naturally extended to
$$
    K_f:  \mc{H}_f \otimes  \mc{H}_f \to \C((z))
$$
which we denote by the same symbol. This defines a symplectic pairing $\omega_f$ on $\mc{H}_f$ by
$$
   \omega_f(\alpha, \beta):= \Res_{z=0} z^{-n} K_f(\alpha, \beta)dz,
$$
with $\mc{H}_f^{(0)}$ being a maximal isotropic subspace. Following \cite{Saito-primitive},
\begin{defn}\label{defn-good-section} A good section $\sigma$ is defined by a splitting of the quotient $ \mc{H}_f^{(0)}\to \Omega_f$,
$$
  \sigma: \Omega_f\to  \mc{H}_f^{(0)},
$$
such that: (1) $\sigma$ preserves the $\Q$-grading; (2) %% $\Im(\sigma)$ have vanishing higher residue beyonds the leading order:
$K_f(\Im(\sigma), \Im(\sigma))\subset z^n \C$.

A basis of the image $\Im(\sigma)$ of a good section $\sigma$ will be referred to as a good basis of $\mc H_f^{(0)}$.
\end{defn}

\begin{defn} A good opposite filtration $\L$  is defined by a splitting
$$
      \mc{H}_f= \mc{H}_f^{(0)}\oplus \L
$$
such that: (1) $\L$ preserves the $\Q$-grading;  (2) $\L$ is an isotropic subspace; (3) $z^{-1}: \L\to \L$.
\end{defn}
\begin{rmk} Here for $f$ being weighted homogeneous, (1) is a convenient and equivalent statement to the conventional condition that $\nabla^{GM}_{z\pa_z}$ preserves $\mc L$ (see e.g. \cite{LLSaito} for an exposition).

\end{rmk}
\noindent The above two definitions are equivalent. In fact,   a good opposite filtration $\L$ defines the splitting
$
       \sigma: \Omega_f\overset{\iso}{\to} \mc{H}_f^{(0)}\cap z \L
$. Conversely, a good section $\sigma$ gives rise to the good opposite filtration $\L=z^{-1}\Im(\sigma)[z^{-1}]$.
As shown in  \cite{Saito-primitive, Saito-existence},
  the primitive forms   associated to the weighted homogeneous singularities are in one-to-one correspondence with   good sections   (up to a nonzero scalar).
  Therefore, we only introduce the notion of good sections, and refer our readers to loc. cite for precise notion of the primitive forms. We remark that for general isolated singularities, we need the notion of \emph{very good sections} \cite{Saito-existence, Saito-uniqueness} in order to incorporate with the monodromy.

\subsection{The perturbative equation} We start with a good basis $\{[\phi_\alpha d^n \mathbf{x}]\}_{\alpha=1}^\mu$  of $\mc H_f^{(0)}$, where $ d^n\mathbf{x}:=dx_1\cdots dx_n$. In this subsection, we will formulate the perturbative method of \cite{LLSaito} for computing its associated primitive form, flat coordinates and the potential function. The construction works for general $f$ after the replacement of  a good basis by a very good one (see also \cite{Saito-uniqueness}). We will focus on $f$ being weighted homogeneous since in such case it leads to a very effective computation algorithm in practice. In the following discussion we will then assume  $\{\phi_\alpha\}_{\alpha=1}^\mu$ to be   weighted homogeneous polynomials in $\C[\mathbf{x}]$ that  represent  a basis of the Jacobi algebra $\Jac(f)$ and $\phi_1=1$.
\subsubsection{The exponential map}
Let $F$ be a local universal unfolding of $f(\mathbf{x})$ around $\mbf{0}\in \C^\mu$:
$$
   F: \C^n\times \C^\mu \to \C, \quad F(\mathbf{x},\mathbf{s}):=f(\mathbf{x})+\sum_{\alpha=1}^\mu s_\alpha \phi_\alpha(\mathbf{x}), \quad \mathbf{s}=(s_1,\cdots, s_\mu).
$$
  The polynomial  $F$ becomes weighted homogeneous of total degree $1$  after   the assignment
$$
   \deg(s_\alpha):=1-\deg(\phi_\alpha).
$$
The higher residue pairing is also defined for $F$ as the family version, but we will not use it explicitly in our discussion (although implicitly used essentially).

Let $B:=\text{Span}_{\C}\{[\phi_\alpha d^n\mathbf{x}]\}\subset \mc H_f^{(0)}$ be spanned by the chosen good basis.  Then
$$
   \mc H_f^{(0)}=B[[z]], \quad \mc H_f=B((z)).
$$
Let $B_F:=\text{Span}_{\C}\{\phi_\alpha d^n\mathbf{x}\}$ be another copy of the vector space spanned by the forms $\phi_\alpha d^n\mathbf{x}$. We use a different notation to distinguish it with $B$, since $B_F$ should be viewed as a subspace of the Brieskorn lattice for the unfolding $F$. See \cite{LLSaito} for more details.

 Consider the following exponential operator \cite{LLSaito}
$$
   e^{(F-f)/z}: B_F\to B((z))[[\mathbf{s}]]
$$
defined as a $\C$-linear map on the basis of $B_F$ as follows.  Let $\C[\mathbf{s}]_k:=\Sym^k(\text{Span}_\C\{s_1,\cdots, s_\mu\})$ denote the space of $k$-homogeneous polynomial in $\mathbf{s}$ (not to be confused with the weighted homogeneous polynomials).
 As elements in $\mc H_f\otimes\C[\mathbf{s}]_k$, we can decompose
$$
  [z^{-k}(F-f)^k \phi_\alpha d^n\mathbf{x}]= \sum_{m\geq -k}\sum_{\beta} h^{(k)}_{\alpha \beta,m}z^m [\phi_\beta d^n\mathbf{x}],
$$
where $h^{(k)}_{\alpha \beta,m}\in \C[\mathbf{s}]_k$. Then we define
$$
  e^{(F-f)/z}(\phi_\alpha d^n\mathbf{x}):=\sum_{k=0}^\infty  \sum_{\beta} \sum_{m\geq -k} h^{(k)}_{\alpha \beta,m}  {z^m \over k! } [\phi_\beta d^n\mathbf{x}] \in B((z))[[\mathbf{s}]]
$$
\begin{prop}
The exponential map extends to a $\C((z))[[\mathbf{s}]]$-linear isomorphism
$$
e^{(F-f)/z}: B_F((z))[[\mathbf{s}]]\to B((z))[[\mathbf{s}]].
$$
\end{prop}
\begin{proof} Clearly, $e^{(F-f)/z}$ extends to a $\C((z))[[\mathbf{s}]]$-linear map on $B_F((z))[[\mathbf{s}]]$. The statement follows by noticing
$
e^{(F-f)/z}\equiv 1 \mod (\mathbf{s})$ under the manifest identification between $B$ and $B_F$.
\end{proof}

We will use the same symbol
$$
  K_f: B((z))[[\mathbf{s}]]\times B((z))[[\mathbf{s}]]\to \C((z))[[\mathbf{s}]]
$$
to denote the $\C[[\mathbf{s}]]$-linear extension of the higher residue pairing to $\mc H_f[[\mathbf{s}]]=B((z))[[\mathbf{s}]]$.
\begin{lem}\label{exp-isotropy}
For any $\varphi_1, \varphi_2\in B_F$, we have
$$
 K_f(e^{(F-f)/z} \varphi_1, e^{(F-f)/z}\varphi_2)\in z^n \C[[z,\mathbf{s}]]
$$
In particular, $e^{(F-f)/z}$ maps $B_F[[z]]$ to an isotropic subspace of $\mc H_f[[\mathbf{s}]]$.
\end{lem}
\begin{proof} Let $K_F$ denote  the higher residue pairing for the unfolding $F$ \cite{Saito-residue}. The exponential operator $e^{(F-f)/z}$ gives an isometry (with respect to the higher residue pairing) between the Brieskorn lattice for the unfolding $F$ and the trivial unfolding $f$ \cite{LLSaito, Saito-uniqueness}. That is,  $K_f(e^{(F-f)/z} \varphi_1, e^{(F-f)/z}\varphi_2)=K_F(\varphi_1, \varphi_2)\in z^n \C[[z,\mathbf{s}]]$, where $\varphi_1$, $\varphi_2$ are treated as elements of  Brieskorn lattice for the unfolding $F$.
\end{proof}

\begin{rmk}
The above lemma can also be proved directly via an explicit formula of $K_f$ described in \cite{LLSaito}. By such a formula,  there exists a compactly supported differential operator $P({\pa\over \pa \bar{x_i}}, z{\pa\over \pa {x_i}}, \lrcorner \pa_{x_i}, \wedge d\bar x_i)$  on smooth differential forms  composed of ${\pa\over \pa \bar{x_i}}, z{\pa\over \pa {x_i}}, \lrcorner \pa_{x_i}, \wedge d\bar x_i$ and some cut-off function such that
$$
   K_f(e^{(F-f)/z} \varphi_1, e^{(F-f)/z}\varphi_2)=z^n \int_X e^{(F-f)/z}\varphi_1\wedge P({\pa\over \pa \bar{x_i}}, z{\pa\over \pa {x_i}}, \lrcorner \pa_{x_i}, \wedge d\bar x_i) (e^{-(F-f)/z}\varphi_2).
$$
Since $P$ will not introduce negative powers of $z$ when passing through $e^{(f-F)/z}$, the lemma follows.
\end{rmk}

\begin{thm}\label{thm-decomp}
 Given a good basis  $\{[\phi_\alpha d^n \mathbf{x}]\}_{\alpha=1}^\mu\subset \mc H_f^{(0)}$,  there exists a unique  pair $(\zeta, \mc J)$ satisfying the following:
  $(1)\, \zeta \in B_F[[z]][[\mathbf{s}]], \quad (2)\, \mc J \in [d^n\mathbf{x}]+z^{-1}B[z^{-1}] [[\mathbf{s}]]\subset \mc H_f[[s]],$ and
\begin{equation}\label{pert-eqn}
   e^{(F-f)/z} \zeta= \mc J. \tag{$\star$}
\end{equation}
 Moreover, both $\zeta$ and $\mc J$ are weighted homogeneous.
\end{thm}
\begin{proof}
We will solve $ \zeta(\mathbf{s})$ recursively with respect to the order in $\mathbf{s}$. Let
$$
   \zeta=\sum_{k=0}^\infty \zeta_{(k)}= \sum_{k=0}^\infty \sum_\alpha \zeta_{(k)}^\alpha \phi_\alpha d^n\mathbf{x}, \quad \zeta^\alpha_{(k)}\in \C[[z]]\otimes_{\C}\C[\mathbf{s}]_k.
$$
Since $e^{(F-f)/z}\equiv 1\mod (\mathbf{s})$, the leading order of \eqref{pert-eqn} is
$$
   \zeta_{(0)}\in [d^n\mathbf{x}]+z^{-1}B[z^{-1}]
$$
which is uniquely solved by  $
   \zeta_{(0)}=\phi_1 d^n\mathbf{x}
$.
Suppose we have solved \eqref{pert-eqn} up to order $N$, i.,e,  $\zeta_{(\leq N)}:=\sum_{k=0}^N \zeta_{(k)}$ such that
$$
    e^{(F-f)/z} \zeta_{(\leq N)}\in [d^n\mathbf{x}]+ z^{-1}B[z^{-1}] [[\mathbf{s}]] \mod (\mathbf{s}^{N+1}).
$$
Let $R_{N+1}\in B((z))\otimes_{\C}\C[\mathbf{s}]_{(N+1)}$ be the $(N+1)$-th order component of $e^{(F-f)/z} \zeta_{(\leq N)}$. Let
$$
  R_{N+1}= R_{N+1}^+ + R_{N+1}^-
$$
where $R_{N+1}^+\in B[[z]]\otimes_{\C}\C[\mathbf{s}]_{(N+1)}, R_{N+1}^-\in z^{-1}B[z^{-1}]\otimes_{\C}\C[\mathbf{s}]_{(N+1)}$. Let $\tilde R_{N+1}^+\in B_F[[z]]\otimes_{\C}\C[\mathbf{s}]_{(N+1)}$ correspond to $R_{N+1}^+$ under the manifest identification between $B$ and $B_F$. Then
$$
  \zeta_{(\leq N+1)}:=\zeta_{(\leq N)}-\tilde R_{N+1}^+
$$
gives the unique solution of \eqref{pert-eqn} up to order $N+1$. This algorithm allows us to solve $\zeta, \mc J$ perturbatively to arbitrary order. The weighted homogeneity follows from the fact that \eqref{pert-eqn} respects the weighted degree.
\end{proof}

\begin{rmk}\label{rmk-primitive}
In  \cite{LLSaito},  it is shown that the volume form $$
  \sum_{k=0}^\infty \sum_\alpha \zeta_{(k)}^\alpha \phi_\alpha d^n\mathbf{x}
$$
gives the power series expansion of a representative of  the   primitive form associated to the good basis  $\{[\phi_\alpha d^n \mathbf{x}]\}_{\alpha=1}^\mu$. In particular,  this is a perturbative way   to compute the primitive form  via a formal solution of the Riemann-Hilbert-Birkhoff problem.

\end{rmk}

\subsubsection{Flat coordinates and potential function} Let $(\zeta, \mc J)$ be the unique solution of \eqref{pert-eqn}. As shown in \cite{LLSaito}, $\zeta$ represents the power series expansion of a primitive form. However for the purpose of mirror symmetry, it is more convenient to work with $\mc J$, which plays the role of Givental's J-function (see \cite{G-tutorial} for an introduction). This allows us to read off the flat coordinates and the potential function of the associated Frobenius manifold structure.

With the natural embedding $z^{-1}\C[z^{-1}][[\mathbf{s}]]\into z^{-1}\C[[z^{-1}]][[\mathbf{s}]]$,  we decompose
$$
    \mc J=[d^nx]+\sum_{m=-1}^{-\infty} z^{m} \mc J_{m}, \quad \text{where}\ \mc J_m=\sum_\alpha \mc J_{m}^\alpha [\phi_\alpha d^n\mathbf{x}], \mc J_{m}^\alpha\in \C[[\mathbf{s}]].
$$
We denote the $z^{-1}$-term by
$$
   t_\alpha(\mathbf{s}):=\mc J_{-1}^\alpha(\mathbf{s}).
$$
It is easy to see that $t_\alpha$ is weighted homogeneous of the same degree as $s_\alpha$ such that $t_\alpha=s_\alpha+O(\mathbf{s}^2)$. Therefore $t_\alpha$ defines a set of new homogeneous local coordinates on the (formal) deformation space of $f$.

\begin{prop}\label{quantum-diff} The function  $\mc J=\mc J(\mathbf{s}(\mathbf{t}))$ in coordinates $t_\alpha$ satisfies
$$
  \pa_{t_\alpha}\pa_{t_\beta} \mc J=z^{-1}\sum_\gamma A_{\alpha\beta}^\gamma(\mathbf{t}) \pa_{t_\gamma}\mc J
$$
for some homogeneous $A_{\alpha\beta}^\gamma(\mathbf{t})\in \C[[\mathbf{t}]]$ of weighted degree $\deg \phi_\alpha+\deg \phi_\beta-\deg \phi_\gamma$. Moreover, for any $\alpha, \beta, \gamma, \delta$,
$$
   \pa_{t_\alpha}A_{\beta\gamma}^\delta= \pa_{t_\beta}A_{\alpha\gamma}^\delta, \quad \sum_{\sigma} A_{\alpha\sigma}^\delta A_{\beta\gamma}^\sigma= \sum_{\sigma} A_{\beta\sigma}^\delta A_{\alpha\gamma}^\sigma
$$
\end{prop}
\begin{proof} Consider the splitting
$$
   \mc H_f[[\mathbf{s}]]=B((z))[[\mathbf{s}]]=\mc H_+\oplus \mc H_-,
   $$
where
$$
 \mc H_+:=e^{(F-f)/z}(B_F[[z]][[\mathbf{s}]])\subset B((z))[[\mathbf{s}]], \quad \mc H_-:=z^{-1}B[z^{-1}][[\mathbf{s}]].
$$
 Let $
\mathfrak{B}_F:= \mc H_+\cap z \mc H_-
$. Equation \eqref{pert-eqn} implies that $
  z\pa_{t_\alpha}\mc J\in \mathfrak{B}_F
$, with $z$-leading term of constant coefficient
$$
z\pa_{t_\alpha} \mc J \in [\phi_\alpha d^n\mathbf{x}]+\mc H_-.
$$
In particular, $\{z\pa_{t_\alpha}\mc J\}$ form  a $\C[[\mathbf{s}]]$-basis of $\mathfrak{B}_F$.

Similarly, $z^2\pa_{t_\alpha}\pa_{t_\beta}\mc J=z^2\pa_{t_\alpha}\pa_{t_\beta}(e^{(F-f)/z}\zeta)\in \mc H_+$, and $z^2\pa_{t_\alpha}\pa_{t_\beta}\mc J \in z\mc H_-$ by the above property of leading constant coefficient. Therefore $z^2\pa_{t_\alpha}\pa_{t_\beta}\mc J \in \mathfrak{B}_F$. This implies the existence of functions $A_{\alpha\beta}^\gamma=A_{\alpha\beta}^\gamma(\mathbf{s}(\mathbf{t}))$ such that
$$
z^2\pa_{t_\alpha}\pa_{t_\beta}\mc J=\sum_{\gamma}zA_{\alpha\beta}^\gamma(\mathbf{t}) \pa_{t_\gamma}\mc J
$$
The homogeneous degree follows from the fact that $\mc J$ is weighted homogeneous.

Let $\mc A_\alpha$ denote the linear transformation on $\mathfrak{B}_F$ by
$$
 \mc A_\alpha: z\pa_{\beta}\mc J \to \sum_{\gamma}A_{\alpha\beta}^\gamma z\pa_{t_\gamma}\mc J.
$$
We can rewrite the above equation as $
   (\pa_{t_\alpha}-{z^{-1}}\mc A_\alpha)\pa_{t_\beta}\mc J=0.
$ We notice that
$$
  \big[\pa_{t_\alpha}-{z^{-1}}\mc A_\alpha, \pa_{t_\beta}-{z^{-1}}\mc A_\beta  \big]=0\ \text{on}\ \mathfrak{B}_F \quad, \forall \alpha, \beta.
$$
Therefore the last equations in the proposition hold.
\end{proof}

\begin{lem} In terms of the coordinates $t_\alpha$, we have
$$
  K_f(z\pa_{t_\alpha}\mc J, z\pa_{t_\beta}\mc J)=z^n g_{\alpha\beta}.
$$
Here $g_{\alpha\beta}$ is the constant equal to the residue pairing  $\eta_f(\phi_\alpha d^n\mathbf{x},  \phi_\beta d^n\mathbf{x})$.
\end{lem}
\begin{proof} We adopt the same notations as in the above proof. Since $z\pa_{t_\alpha}\mc J\in \mc H_+$,
$$
   K_f(z\pa_{t_\alpha}\mc J, z\pa_{t_\beta}\mc J)\in z^n \C[[z]][[\mathbf{s}]]
$$
by Lemma \ref{exp-isotropy}. Since also $z\pa_{t_\alpha}\mc J=[\phi_\alpha d^n\mathbf{x}]+\mc H_-\in z\mc H_-$, we have $$
K_f(z\pa_{t_\alpha}\mc J, z\pa_{t_\beta}\mc J)\in z^ng_{\alpha\beta}+ z^{n-1}\C[z^{-1}][[\mathbf{s}]].
$$
The lemma follows from the above two properties.
\end{proof}

\begin{cor}
Let $A_{\alpha\beta\gamma}(\mathbf{t}):=\sum_{\delta}A_{\alpha\beta}^\delta g_{\delta\gamma}$. Then $A_{\alpha\beta\gamma}$ is symmetric in $\alpha, \beta, \gamma$.
\end{cor}
\begin{proof}
By the previous lemma, $\pa_{t_\gamma} K_f(z\pa_{t_\alpha}\mc J, z\pa_{t_\beta}\mc J)=0$. The corollary now follows from Proposition \ref{quantum-diff}.
\end{proof}

The   properties in the propositions of this subsection  can be summarized as follows. The triple $(\pa_{t_\alpha}, A_{\alpha\beta}^\gamma, g_{\alpha\beta})$ defines a (formal) Frobenius manifold structure on a neighborhood $S$ of the origin with $\{t_\alpha\}$ being the flat coordinates, together with the potential function  $\mc F_0(\mathbf{t})$   satisfying
$$
A_{\alpha\beta\gamma}(\mathbf{t})=\pa_{t_\alpha}\pa_{t_\beta}\pa_{t_\gamma}\mc F_0(\mathbf{t}).
$$
It is not hard to see that $\mc F_0(\mathbf{t})$ is homogeneous of degree $3-\hat{c}_f$.
As in the next proposition, the potential function $\mc F_0(\mathbf{t})$ can also be computed perturbatively. Let
$$
\mc F_0(\mathbf{t})=\mc F_{0,(\leq N)}(\mathbf{t})+O(\mathbf{t}^{N+1}).
$$
\begin{prop}\label{prop-order} The potential function  $\mc F_0$ associated to the unique  pair $(\zeta, \mc J)$ satisfies
 $$
     \pa_{t_\alpha}\mc F_0(\mathbf{t})=\sum_\beta g_{\alpha \beta} \mc J_{-2}^\beta(\mathbf{s}(\mathbf{t})).
$$
Moreover,   $\mc F_0^{(\leq N)}(\mathbf{t})$ is determined by $\zeta_{(\leq N-3)}(\mathbf{s})$.
\end{prop}

\begin{proof}
The first statement follows directly from Proposition \ref{quantum-diff}.

Recall $\zeta(\mathbf{s})=\zeta_{(\leq N)}(\mathbf{s})+O(\mathbf{s}^{N+1})$. Let $\mc J^\alpha_{m}(\mathbf{s})= \mc J^\alpha_{m, (\leq N)}(\mathbf{s})+O(\mathbf{s}^{N+1})$.
It is easy to see that $\mc F_0^{(\leq N)}(\mathbf{t})$ only depends on $\mc J^\alpha_{-1, (\leq N-2)}(\mathbf{s})$, $\mc J^\alpha_{-2, (\leq N-1)}(\mathbf{s})$, and $\mc J^\alpha_{m, (\leq N)}(\mathbf{s})$ only depends on $\zeta_{(\leq N+m)}(\mathbf{s})$. Hence, the second statement follows.
\end{proof}
\begin{rmk} By Remark \ref{rmk-primitive}, $\zeta$ is in fact an analytic primitive form. Therefore,  both $t_\alpha$ and $\mc F_0(\mathbf{t})$ are in fact analytic functions of $\mathbf{s}$ at the germ $\mathbf{s}=0$.
\end{rmk}

\subsection{Computation for exceptional unimodular singularities}\label{sec-14}  We start with the next proposition, which    follows from a related statement  for Brieskorn lattices \cite{Hertling-classifyingspace}. An explicit calculation of the moduli space of good sections for general weighted homogenous polynomials is  also given in \cite{LLSaito,Saito-uniqueness}.  For exposition, we include a proof here.
\begin{prop}\label{good-basis-uniqueness}
If $f$ is one of the 14 exceptional unimodular singularities, then there exists a unique good section $\{[\phi_\alpha d^n\mathbf{x}]\}_{\alpha=1}^\mu$, where $\{\phi_\alpha\}\subset \C[\mathbf{x}]$ are (arbitrary) weighted homogeneous representatives of a basis of the Jacobi algebra $\Jac(f)$.
\end{prop}
\begin{proof}
  We give the details for  $E_{12}$-singularity. The other 13 types are established similarly.

The $E_{12}$-singularity is given by $f=x^3+y^7$ with $\deg x={1\over 3}, \deg y={1\over 7}$, and central charge $\hat c_f={22\over 21}$.  We consider the weighted homogeneous monomials
$$
  \{\phi_1, \cdots, \phi_{12}\}=\{1,y,y^2, x, y^3, xy, y^4, xy^2, y^5, xy^3, xy^4, xy^5\}\subset \C[x,y]
$$
which represent a basis of $\Jac(f)$. The normalized residue pairing $g_{\alpha\beta}$  between $\phi_\alpha, \phi_\beta$ is equal to $1$ if $\alpha+\beta=13$, and $0$ otherwise. Since $K_f$ preserves the $\Q$-grading,
$$
    \deg K_f([\phi_\alpha dxdy], [\phi_\beta dxdy])=\deg\phi_\alpha +\deg \phi_\beta+2-\hat c_f,
$$
which has to be an integer for a non-zero pairing. A simple degree counting implies that
$$
   K_f([\phi_\alpha dxdy], [\phi_\beta dxdy])=z^{2} g_{\alpha\beta}
$$
and therefore $\{[\phi_\alpha dxdy]\}$ constitutes a good basis.

Let $\{\phi_\alpha^\prime\}$ be another set of weighted homogeneous polynomials such that $\{[\phi_\alpha^\prime dxdy]\}$  gives a good basis. We can assume $\phi_\alpha^\prime\equiv \phi_\alpha$ as elements in $\Jac(f)$ and $\deg \phi_\alpha^\prime=\deg \phi_\alpha$. Since $[\phi_\alpha dxdy]$ forms a $\C[[z]]$-basis of $\mc H_f^{(0)}$, we can decompose
$$
   [\phi_\alpha^\prime dxdy]=\sum_\beta R_\alpha^\beta [\phi_\beta dxdy], \quad R_\alpha^\beta\in \C[[z]].
$$
By the weighted homogeneity, $R_\alpha^\beta$ is homogeneous of degree $\deg \phi_\alpha-\deg \phi_\beta$, which is not an integer unless $\alpha=\beta$. Thus $[\phi^\prime_\alpha dxdy]=[\phi_\alpha dxdy]$, and hence the uniqueness.
\end{proof}

 Let $\mc F_0$ be the potential function of the associated Frobenius manifold structure. Then $\mc F_0$ is an analytic function, as an immediate consequence of  the above uniqueness   together with the existence of the (analytic) primitive form.
 As will be shown in Lemma \ref{reconstruction-lemma}, we only need to compute $\mc F_{0, (\leq 4)}$ to prove mirror symmetry.

 We illustrate the perturbative calculation for the $E_{12}$-singularity $f=x^3+y^7$. The full result is summarized in the appendix by similar calculations.
We adopt the same notations as in the proof of  Proposition \ref{good-basis-uniqueness}.  By Proposition \ref{prop-order}, we only need $\zeta_{(\leq 1)}$ to compute $\mc F_{0, (\leq 4)}$, which is
$$
  \zeta_{(\leq 1)}=dxdy.
$$
Using the equivalence relation in $\mc H_f$, we can expand
$$
e^{(F-f)/z}(\zeta_{(\leq 1)})=\sum_{k=0}^3{(F-f)^k\over k!} z^{-k}\zeta_{(\leq 1)}+O(\mathbf{s}^4)
$$
in terms of the good basis $\{\phi_\alpha\}$.  We find the flat coordinates up to order $2$
$$\begin{array}{lll}
   t_1\doteq  s_1-\frac{s_5 s_7}{7}-\frac{s_3 s_9}{7},&
   t_2\doteq s_2-\frac{s_7^2}{7}-\frac{2 s_5 s_9}{7}, &
   t_3\doteq s_3-\frac{3 s_7 s_9}{7},\\
   t_4\doteq s_4-\frac{s_8 s_9}{7}-\frac{s_7 s_{10}}{7}-\frac{s_5 s_{11}}{7}-\frac{s_3
                 s_{12}}{7},&
   t_5\doteq  s_5-\frac{2 s_9^2}{7},&
   t_6\doteq  s_6-\frac{2 s_9s_{10}}{7}-\frac{2 s_7 s_{11}}{7}-\frac{2 s_5 s_{12}}{7},\\
   t_7\doteq s_7,&
   t_8\doteq  s_8-\frac{3 s_9 s_{11}}{7}-\frac{3 s_7 s_{12}}{7}, &
   t_9\doteq s_9,\\
   {}\!\!t_{10}\doteq  s_{10}-\frac{4 s_9s_{12}}{7},&
   {}\!\!t_{11} \doteq s_{11},&
   {}\!\!t_{12}\doteq s_{12}.
\end{array}$$

This allows us to solve the inverse function $s_\alpha=s_{\alpha}(\mathbf{t})$ up to order $2$. An straight-forward but tedious computation of the $z^{-2}$-term shows that in terms of flat coordinates
$$
\mc F_{0, (\leq 4)}=\mc F_{0}^{(3)}+\mc F_{0}^{(4)},
$$
where $\mc F_{0}^{(3)}$ is the third order term representing the algebraic structure of $\Jac(f)$
$$
\pa_{t_\alpha}\pa_{t_\beta}\pa_{t_\gamma}\mc F_{0}^{(3)}=\eta_f([\phi_\alpha\phi_\beta\phi_\gamma dxdy],[dxdy]).
$$
The fourth order term $\mc F_{0}^{(4)}$, which we call the \emph{4-point function}, is computed by
\begin{align*}
   -\mc F_{0}^{(4)}&={1\over 14} t_5 t_6 t_7^2+{1\over 18} t_6^3t_8+{1\over 7}t_5^2t_7t_8+{1\over 7} t_3t_7^2t_8+{1\over 6} t_4t_6t_8^2+{1\over 14} t_5^2t_6t_9+{1\over 7} t_3t_6t_7t_9\\
                 &\quad +{1\over7} t_3t_5t_8t_9+{1\over 7} t_2t_7t_8t_9+{1\over 14}t_2t_6t_9^2+{1\over 14} t_5^3t_{10}+{1\over 6} t_4t_6^2t_{10}+{2\over 7}t_3t_5t_7t_{10}+{1\over 14}t_2t_7^2 t_{10}\\
                 &\quad  +{1\over 6}t_4^2t_8t_{10}+{1\over 14}t_3^2 t_9t_{10}+{1\over 7}t_2t_5t_9t_{10}+{1\over 7}t_3t_5^2t_{11}+{1\over 6}t_4^2t_6t_{11}+{1\over 7}t_3^2t_7 t_{11}+{1\over 7}t_2t_5t_7t_{11} \\
               &\quad +{1\over 7}t_2t_3t_9t_{11}+{1\over 18}t_4^3t_{12} +{1\over 14} t_3^2t_5t_{12}+{1\over 14}t_2t_5^2t_{12}+{1\over 7}t_2t_3t_7t_{12}+{1\over 14}t_2^2t_9t_{12}.
\end{align*}

In particular, for our later use, we can read off
  $$\pa_{t_4}\pa_{t_4}\pa_{t_4}\pa_{t_{12}}\mathcal{F}_{0}|_{\mathbf{t}=\mathbf{0}}=-{1\over 3},\qquad  \pa_{t_2}\pa_{t_2}\pa_{t_9}\pa_{t_{12}}
  \mathcal{F}_{0}|_{\mathbf{t}=\mathbf{0}}=-{1\over 7}. $$

\section{Mirror Symmetry for exceptional unimodular singularities}\label{sec4}
In this section, we use two reconstruction results to prove the mirror symmetry conjecture between the 14 exceptional unimodular singularities and their FJRW mirrors both at genus $0$ and higher genera.

\subsection{Mirror symmetry at genus zero}
Throughout this subsection, we assume  $W^T$ to be one of    the 14 exceptional unimodular singularities in Table \ref{tab-exceptional-singularities}. We will consider  the  ring isomorphism
 $\Psi:  \Jac(W^T)\to (H_W, \bullet)$ defined in Proposition \ref{non-Krawitz}.
We will also denote the specified basis of $\Jac(W^T)$ therein by   $\{\phi_1, \cdots, \phi_\mu\}$ such that $\deg \phi_1\leq \deg \phi_2\leq \cdots\leq  \deg \phi_\mu$.   As have mentioned, there is a  {formal} Frobenius manifold structure on the FJRW ring $(H_W, \bullet)$ with a prepotential $\mc F_{0, W}^{\rm FJRW}$. We have also shown in the previous section that there is a Frobenius manifold structure with flat coordinates $(t_1, \cdots, t_\mu)$ associated to (the primitive form) $\zeta$ therein, whose prepotential will be denoted as $\mc F_{0, W^T}^{\rm SG}$ from now on.
   We introduce the primary correlators $\langle\cdots\rangle_{0, k}^{W^T, \rm SG}$ associated to the Frobenius manifold structure on B-side.
   The primary  correlators, up to linear combinations, are given by
     \begin{align}
        \langle\phi_{i_1}, \cdots, \phi_{i_k}\rangle_{0, k}^{W^T, \rm SG}={\pa^k \mc F_{0, W^T}^{\rm SG}\over \pa t^{i_1}\cdots \pa t^{i_k}}(0).
     \end{align}
As from the specified ring isomorphism $\Psi$ and \eqref{FJRW-ring}, we have
$$\langle \be_{\phi_i}, \be_{\phi_j}, \be_{\phi_k}\rangle_{0, 3}^{W}=\langle \phi_i, \phi_j, \phi_k\rangle_{0, 3}^{W^T, \rm SG}.$$
 As from Proposition \ref{thm-reformulathm1} and the computation   in section \ref{sec-14} and in the appendix, we have
  $$
   \langle \be_{x_i},   \be_{x_i}, \be_{M_i^T/x_i^2}, \be_{\phi_{\mu}}\rangle_{0, 4}^{W}=
       -\langle x_i, x_i, M_i^T/x_i^2, \phi_{\mu}\rangle_{0, 4}^{W^T, \rm SG}.$$
To deal with the sign, we will do the following modifications, as in  \cite[section 6.5]{FJR}.
We simply denote $(-1)^r:= e^{\pi \sqrt{-1} r}$. Let $\tilde{\mc F}_0^{\rm SG}$ denote the potential function of the Frobenius manifold structure $\tilde \zeta:=(-1)^{-\hat{c}_{W^T}}\zeta$. Set
$\tilde \phi_j:=(-1)^{-\deg \phi_j}\phi_j$ and define a map $\tilde \Psi: \Jac(W^T)\to H_W$ by $\tilde \Psi(\tilde \phi_j):=\Psi(\phi_j)$. Let $\tilde{\mathbf{t}}$ denote the flat coordinate of $\tilde{\mc F}_0^{\rm SG}$, namely
\begin{align}\label{eqn-flatcoord}
   \tilde t_j&= (-1)^{1-\deg t_j} t_j.
\end{align}
  As a consequence, we have $\tilde{\mc F}_{0,  W^T}^{(3), \rm SG}=\mc F_{0,  W^T}^{(3), \rm SG}$ and $
      \tilde{\mc F}_{0,  W^T}^{(4), \rm SG}=-\mc F_{0,  W^T}^{(4), \rm SG}$.  Denote $\tilde \be_{\tilde \phi_j}:=\tilde \Psi(\tilde \phi_j)$. Then
      $\tilde \Psi$ defines a pairing-preserving ring isomorphism, which is read off from the identities
      $\langle \tilde \be_{\tilde\phi_i}, \tilde \be_{\tilde \phi_j}, \tilde \be_{\tilde \phi_k}\rangle_{0, 3}^{W}=\langle \tilde \phi_i, \tilde \phi_j, \tilde \phi_k\rangle_{0, 3}^{W^T, \tilde \zeta, \rm SG}$, Moreover,
    \begin{align}\label{eqn-4pt}
   \langle \tilde \be_{\widetilde{x_i}},   \tilde \be_{\widetilde{x_i}}, \tilde \be_{\widetilde{M_i^T/x_i^2}}, \tilde \be_{\tilde \phi_{\mu}}\rangle_{0, 4}^{W}=
       \langle \widetilde{x_i}, \widetilde{x_i}, \widetilde{M_i^T/x_i^2}, \tilde \phi_{\mu}\rangle_{0, 4}^{W^T, \tilde \zeta, \rm SG}.
  \end{align}
 From now on, we will simplify the notations by ignoring the symbol  $\tilde{ }$ and the superscript $\tilde \zeta$.
 \begin{comment}
   (here we notice $\phi_2=y, \phi_4=x, \phi_9=y^5, \phi_{12}=xy^5$)
   \begin{align*}
      \langle \be_{\phi_4}, \be_{\phi_4}, \be_{\phi_4}, \be_{\phi_{12}}\rangle_{0, 3}^{W, FJRW}&={1\over 3}=-\langle \phi_4, \phi_4, \phi_4, \phi_{12}\rangle_{0, 4}^{W^T, \rm SG},\\
       \langle \mathbf{e}_{\phi_2}, \mathbf{e}_{\phi_2}, \mathbf{e}_{\phi_9}, \mathbf{e}_{\phi_{12}}\rangle_{0, 3}^{W, FJRW}&={1\over 7}=-\langle \phi_2, \phi_2, \phi_9, \phi_{12}\rangle_{0, 4}^{W^T, \rm SG}.
   \end{align*}
\end{comment}
In addition, we will simply denote both $H_W$ and $\Jac(W^T)$ as $H$, and simply denote the correlators on both sides as $\langle\phi_{i_1},\cdots, \phi_{i_k}\rangle_{0, k}$ (or $\langle \phi_{i_1},\cdots, \phi_{i_k}\rangle$), whenever there is no risk of confusion.
\begin{comment}
For instance when referring to $\phi_{i_1}, \cdots, \phi_{i_k}\in H$ for A-side, we use  $\langle \phi_{i_1}, \cdots, \phi_{i_k}\rangle_{0, k}$ for the correlator $\langle \mathbf{e}_{\phi_{i_1}}, \cdots, \mathbf{e}_{\phi_{i_k}}\rangle_{0, k}^{W, G_W}$ with $\mathbf{e}_{i_j}$ mapped to
the (non-modified) specified basis $\phi_{j}$ via the isomorphism $\Psi$.
While for B-side, we use it for the correlator $\langle  {\phi_{i_1}}, \cdots, {\phi_{i_k}}\rangle_{0, k}^{W^T, \rm SG}$ with $\phi_{i_j}$ being the modified basis with respect to the modified potential.
\end{comment}
We have the following "Selection rule" for primary correlators.
\begin{lem}
A primary  correlator $\langle \phi_{i_1},\cdots, \phi_{i_k}\rangle_{0, k}$ on either A-side or  B-side is nonzero only if
    \begin{align}\label{eqn-deg-constraint}
      \sum_{j=1}^k\deg\phi_{i_j}=\hat c_{W^T}-3+k.
    \end{align}
 \end{lem}
 \begin{proof}
  The A-side case follows from formula \eqref{selection-non-primary} and    $\hat c_W=\hat c_{W^T}$.
      The   primary correlator   on B-side is given by $\pa_{t_{i_1}}\cdots\pa_{t_{i_k}}\mc F_{0, W^T}^{\rm SG}(0)$, where $\deg \phi_{i_j}=1-\deg t_{i_j}$. Then the statement follows, by noticing that  $\mc F_{0, W^T}^{\rm SG}(0)$ is weighted homogenous of degree $3-\hat c_{W^T}$.
 \end{proof}

A homogeneous  $\alpha\in H$ is called a \emph{primitive class} with respect to the specified basis $\{\phi_j\}$,  if it cannot be written as
$\alpha= \alpha_1 \bullet \alpha_2$ for $0<\deg\alpha_i<\deg\alpha$. A primary correlator $\langle \phi_{i_1}, \cdots, \phi_{i_k}\rangle_{0, k}$ is called \emph{basic} if at least $k-2$ insertions $\phi_{i_j}$ are primitive classes.
Now Theorem \ref{g=0-mirror} is  a direct consequence of the equalities \eqref{eqn-4pt}  and the following statement:
 \begin{lem}[Reconstruction Lemma]\label{reconstruction-lemma}
 If $W^T$ is one of the 14 exceptional singularities, then all the following hold.
\begin{enumerate}
\item \label{>5-point}
The prepotential $\mc F_0$  is uniquely determined from basic correlators $\LD\dots\RD_{0,k}$ with $k\leq 5$.

\item \label{5-point}
All basic correlators  $\LD\phi_{i_1}, \cdots, \phi_{i_5}\RD_{0,5}$ vanish.

\item \label{mirror-4pt}
All the 4-point basic correlators are uniquely determined from the formula \eqref{4-point}.

\end{enumerate}

\end{lem}

\noindent{\em Proof of \eqref{>5-point}:}
The potential function $\mc F_0$ satisfies the WDVV equation \eqref{WDVV} (hence the formula \eqref{eq:WDVV}). We can assume that $\LD \cdots \RD_{0, k}$ is not of type $\LD 1,\cdots \RD_{0, k}$, $k\geq4$, (otherwise it vanishes according to string equation, or the invariance of the primitive form along the $\phi_1$-direction where we notice $\phi_1=1$). Consider a correlator $\LD \cdots, \alpha_a,\alpha_b\bullet\alpha_c,\alpha_d \RD_{0, k}$, with last three insertions non-primitive.  By formula \eqref{eq:WDVV}, such correlator is the sum of $S_k$ together with three terms whose insertion replaces $\alpha_b\bullet\alpha_c$ with lower degree ones $\alpha_b$ or $\alpha_c$ at the same position. Repeating this will turn $\alpha_b\bullet\alpha_c$ into a primitive class, up to product of correlators with less number of insertions.
   By induction both on the degree of non-primitive classes and $k$, we can reduce any correlator to a linear combination of   basic correlators.

 Now we assume that $\langle\phi_{i_1}, \cdots, \phi_{i_k}\rangle_{0, k}$ is a nonzero basic correlator. Then we can write
  $\phi_{i_1}\bullet\cdots\bullet\phi_{i_k}=x^ay^bz^c$. It follows from  the degree constriant \eqref{eqn-deg-constraint} that
\begin{equation}\label{eq:deg}
\hat{c}_{W^T}-3+k=\sum_{j=1}^{k}\deg\phi_{i_j}=aq_x+bq_y+cq_z,
\end{equation}
Let us denote by $P$   the maximal numbers among the degree of a generator $x,y$ and $z$ (or $x$ and $ y$ if $W^T=W^T(x, y)$ is in two variables $x, y$ only). By direct calculations, we conclude
$$k\leq\frac{\hat{c}_{W^T}+1}{1-P}+2<6.$$

\noindent{\em Proof of \eqref{5-point}:} For $W^T=x^p+y^q$,   $x,y$ are generators for the ring structure $H$.
 The multiplications for all the insertions will be in a form of $x^a\bullet y^b$. By the degree constraint, a nonzero  basic correlator  $\LD\phi_{i_1}, \cdots, \phi_{i_k}\RD_{0,k}$ satisfies
\begin{equation}
\label{fermat-deg}
aq+bp=(k-1)pq-2p-2q,
\end{equation}
On the other hand, we assume the  first $k-2$ insertions to be primitive classes, so that they are  either $x$ or $y$. The top degree class $\phi_{\mu}=x^{p-2}\bullet y^{q-2}$ is of degree  $2-\frac{2}{p}-\frac{2}{q}$. Therefore we have the following inequalities required for non-vanishing correlator:
\begin{equation}\label{5F}
a\leq k-2+2(p-2),\quad b\leq k-2+2(q-2),\quad  a+b\leq k-2+2(p-2+q-2).
\end{equation}
It is easy to see that there is no $(a,b)$ satisfying both \eqref{fermat-deg} and \eqref{5F} if $k=5$. Hence $\LD\phi_{i_1}, \cdots, \phi_{i_5}\RD_{0,5}=0$. The arguments  for the remaining $W^T$ on B-side and all the $W$ on A-side are all similar and  elementary, details of which are left to the readers.

\noindent{\em Proof of \eqref{mirror-4pt}:}
Let us start with $W^T=x^p+y^q$, where we notice that $p,q$ are coprime.  The degree constraint  \eqref{fermat-deg} with $k=4$ implies that  $(a,b)=(2p-2,q-2)$ or $(p-2,2q-2)$. Thus the possibly nonzero basic correlators   are $\LD x,x,x^{p-2}y^i,x^{p-2}y^{q-2-i}\RD_{0,4}$, $i=0,\dots,q-2$. On the other hands, if formula \eqref{4-point} holds, then by WDVV equation \eqref{eq:WDVV}, we have
\begin{eqnarray*}
&&\LD x,x,x^{p-2}\bullet y^i,x^{p-2}y^{q-2-i}\RD\\
&=&
-\LD x,x^{p-2},y^i,x\bullet x^{p-2}y^{q-2-i}\RD+\LD x,x\bullet x^{p-2},y^i,x^{p-2}y^{q-2-i}\RD+\LD x,x,x^{p-2},y^i\bullet x^{p-2}y^{q-2-i}\RD
\\
&=&\LD x,x,x^{p-2},y^i\bullet x^{p-2}y^{q-2-i}\RD=\frac{1}{p}.
\end{eqnarray*}

For 2-Chain $W^T=x^py+y^q$, the degree constraint \eqref{eq:deg} tells us
$$a\frac{q-1}{pq}+b\frac{1}{q}=\hat{c}_W-3+k.$$
For   $k=4$, this implies that   $(a,b)=(2p-2,q)$ or $(p-2,2q-1)$. The basic correlators are $\LD x,x,x^{p-2}y^{1+i},x^{p-2}y^{q-1-i}\RD$ with $0\leq i\leq q-1$,  $\LD y,y,x^iy^{q-2},x^{p-2-i}y^{q-1}\RD$ with $0\leq i\leq p-2$ and $\LD x,y,x^iy^{q-1},x^{p-3-i}y^{q-1}\RD$ with $0\leq i\leq p-3$. The first two types are uniquely determined from the correlators which are listed in Proposition \ref{thm:FJRW}. For example, if $0< i< q-1$, since $p\,x^{p-1}y=\pa_{x}W^T=0$ in $\Jac(W^T)$, we have
$$
\LD x,x,x^{p-2}y\bullet y^i,x^{p-2}y^{q-1-i}\RD=\LD x,x,x^{p-2}y,x^{p-2}y^{q-1-i}\bullet y^i\RD.
$$
The last type is determined by
\begin{eqnarray*}
&&\LD x,y,x^iy^{q-1},x^{p-3-i}y^{q-1}\RD\\
&=&-\frac{1}{q}\LD x,y,x^iy^{q-1}, x^{p-2-i}\bullet x^{p-1}\RD\\
&=&-\frac{1}{q}\left(\LD x,y,x^{p-1},x^{p-2}y^{q-1}\RD+\LD x,y\bullet x^{p-1},x^{p-2-i},x^iy^{q-1}\RD\right)\\
&=&-\frac{1}{q}\LD x,y,x\bullet x^{p-2},x^{p-2}y^{q-1}\RD
=-\frac{1}{q}\LD x,x,x^{p-2}\bullet y, x^{p-2}y^{q-1}\RD=-\frac{1}{pq}
\end{eqnarray*}
Here we  use the relation $x^{p}+qy^{q-1}=\pa_{y}W^T=0$ in $\Jac(W^T)$ in the first equality.

For 2-Loop $W^T=x^3y+xy^4$,   the degree constraint \eqref{fermat-deg} with $k=4$ implies that  $(a,b)=(5,4)$ or $(3,7)$. If the formula \eqref{4-point} holds, namely if
$$\LD x,x,xy,x^2y^3\RD=\frac{3}{11},\qquad\LD y,y,xy^2,x^2y^3\RD=\frac{2}{11},$$
then we conclude $\LD x,y,x^2,x^2y^3\RD=\frac{3}{11}, \,\,\LD x,y,y^2,x^2y^3\RD=\frac{2}{11}$ and $\LD x,x,x^2y^2,xy^2\RD=\frac{2}{11}$ from a single WDVV equation for each correlator.
For the rest, we conclude $\LD x,x,xy^3,x^2y\RD=\frac{1}{11}$ and $\LD x,y,xy^3,xy^3\RD=-\frac{1}{11}$ by solving the following linear equations which  come from the WDVV equation,
$$\left\{
\begin{array}{l}
-3\LD x,x,x^2y,xy^3\RD+\LD x,x\bullet xy^3,y^3,y\RD=\LD x,x\bullet y^3,y,xy^3\RD,
\\
-4\LD x,y,xy^3,xy^3\RD=\LD x,y,x^2,x\bullet xy^3\RD+\LD x,y\bullet x^2,x,xy^3\RD.
\end{array}
\right.$$
Here the coefficient $-3$ (resp. $-4$) comes from $3x^2y+y^4=0$ (resp. $x^3+4xy^3=0$) in $\Jac(W^T)$.
Similarly, we conclude   $\LD x,y,x^2y^2,x^2y\RD=-\frac{1}{11}$ and $\LD y,y,x^2y^2,xy^3\RD=\frac{1}{11}$.

For $W^T=x^2y+y^4+z^3\in Q_{10}$, the number of 4-point basic correlators is 10. Three of them are the initial correlators in \eqref{4-point}, $\LD x,x,y,y^3z\RD$, $\LD y,y,y^2,y^3z\RD$, $\LD z,z,z,y^3z\RD$, the rest are
$\LD y,y,y^2z,y^3\RD,$ $\LD y,z,y^3,y^3\RD,$ $\LD z,z,yz,y^2z\RD,$ $\LD x,x,yz,y^3\RD,$ $\LD x,x,y^2,y^2z\RD,$ $\LD x,y,xz,y^3\RD,$ and $\LD z,z,xz,xz\RD.$
We have 7 WDVV equations to reconstruct them from the initial correlators,
$$
\left\{
\begin{aligned}
&4\LD y,y,y^2z,y^3\RD=\LD x,x,y,y^3z\RD,
\LD y,z,y^3,y^3\RD=\LD y,y,y^2z,y^3\RD-\LD y,y,y^2,y^3z\RD,\\
&\LD z,z,yz,y^2z\RD=\LD z,z,z,y^3z\RD,
\LD x,x,yz,y^3\RD=\LD x,x,y,y^3z\RD,
\LD x,x,y^2,y^2z\RD=\LD x,x,y,y^3z\RD,\\
&\LD x,y,xz,y^3\RD=\LD x,x,y,y^3z\RD,
\LD z,z,xz,xz\RD=-4\LD z,z,z,y^3z\RD.
\end{aligned}
\right.
$$

\begin{comment}
$$
\left\{
\begin{aligned}
&A_{4,4,6,8}=A_{2,4,4,10},\quad A_{4,4,5,9}=A_{2,4,4,10},\quad A_{2,4,7,8}=A_{2,4,4,10},\quad A_{2,4,4,10}=4A_{2,2,8,9},\\
&A_{2,3,8,8}=-A_{2,2,5,10}+A_{2,2,8,9}, \quad A_{3,3,6,9}=A_{3,3,3,10}, \quad A_{3,3,7,7}=-4A_{3,3,3,10}
\end{aligned}
\right.$$
\end{comment}

For other singularities of 3-variables,  all the basic $4$-point correlators are uniquely determined from the initial correlators in formula \eqref{4-point}, by the same technique. However, the discussion is more tedius. For example, there are 21 of 4-point basic correlators for type $S_{12}$ singularity $W^T=x^2y+y^2z+z^3x$. We can write down 18 WDVV equations carefully to determine all the 21 basic correlators from 3 correlators in the formula \eqref{4-point}. The details are skipped here.
\qed

\begin{comment}
For $F_{3,3,4}$,...

For $W=Q_{10}$, there are 10 of 4-point basic correlators. 5 of them contain broad sectors. We need 7 WDVV equations, $A_{4,4,6,8},A_{4,4,5,9},A_{2,4,7,8}$ equals $A_{2,4,4,10}$ and
$$A_{2,4,4,10}=4A_{2,2,8,9}, A_{2,3,8,8}=-A_{2,2,5,10}+A_{2,2,8,9}, A_{3,3,6,9}=A_{3,3,3,10}, A_{3,3,7,7}=-4A_{3,3,3,10}$$

For $Q_{12}^T$, there are 19 of 4-point basic correlators. 13 of them contain broad sectors. We need 16 WDVV equations.

For $C_{2,2,4}^T$, there are 14 of 4-point basic correlators, 7 of them contain broad sectors, 3 of them are nonconcave, $\LD\be_y,\be_y,\dots\RD$. We need 11 equations. This is a little bit complicated since the initial correlators are nonconcave. So in the computation section, we need one more WDVV.

For $C_{2,3,3}^T$, there are 15 of 4-point basic correlators, 7 of them contain broad sectors, 7 of them are concave. We need 12 equations.

For $L_{2,2,3}^T$, there are 21 of basic 4-point correlators, 9 of them are concave.
\end{comment}

\bigskip

\subsection{Mirror symmetry at higher genus}
In Section \ref{sec2}, we already constructed the total ancestor FJRW potential $\mathscr{A}^{\rm FJRW}_{W}$ for a pair $(W,G_W)$. Now we give the B-model total ancestor Saito-Givental potential $\mathscr{A}^{\rm SG}_{W^T}$.
Let $S$ be the universal unfolding of the isolated singularity $W^T$.
For a semisimple point ${\bf s}\in S$, Givental \cite{G2} constructed the following formula containing higher genus information of the Landau-Ginzburg B-model of $f$, (see \cite{G1, G2, CI} for more details)
$$\mathscr{A}^{\rm SG}_f({\bf s}):=\exp\left(-\frac{1}{48}\sum_{i=1}^{\mu}\log\Delta^i(\bf s)\right)\widehat{\Psi_{\bf s}}\widehat{R_{\bf s}}(\mathcal{T}).$$
Here $\mathcal{T}$ is the product of $\mu$-copies of Witten-Kontsevich $\tau$-function. $\Delta^i({\bf s}), \Psi_{\bf s}$ and $R_{\bf s}$ are data coming from the Frobenius manifold. The operators $\widehat{\cdot}$ are the so-called quantization operators. We call $\mathscr{A}^{\rm SG}_f({\bf s})$ the Saito-Givental potential for $f$ at point ${\bf s}$.
Teleman \cite{T} proved that $\mathscr{A}^{\rm SG}_f({\bf s})$ is uniquely determined by the genus 0 data on the Frobenius manifold. By definition, the coefficients in each genus-$g$ generating function of $\mathscr{A}^{\rm SG}_f({\bf s})$ is just meromorphic near the non-semisimple point ${\bf s}={\bf 0}$. Recently, using Eynard-Orantin recursion, Milanov \cite{M} proved $\mathscr{A}_{W^T}^{\rm SG}(\mbf{t})$ extends holomorphically at $\mbf{t}=\mbf{0}$. We denote such an extension by $\mathscr{A}_{W^T}^{\rm SG}$ and the Corollary \ref{main} follows from  Theorem \ref{g=0-mirror} and Teleman's theorem.

\subsection{Alternative representatives and the other direction}\label{sec:alternative}

%The stable equivalence  class $[f]$ of the function-germ $f$ may be represented by another  weighted homogeneous polynomial.
%Up to scalar factors of the monomials and up to the stable equivalence (by adding quadratic terms $x_k^2$ in additional variables $x_k$),

Although the theory of primitive forms depends only on the stable equivalence class of the singularity, the FJRW theory definitely depends on the choice of the polynomial together with the group.
For the exceptional unimodular singularities, in the following we list all the additional invertible weighted homogeneous polynomial representatives without quadratic terms $x_k^2$ in additional variables $x_k$ as follows (up to permutation symmetry among variables):
\begin{equation}\label{form:extra}
\left\{
\begin{array}{lll}
 E_{14}: x^3+y^8, & W_{12}: x^2y+y^2+z^5, & W_{13}: x^4y+y^4;\\
        Q_{12}: x^2y+y^5+z^3,& Z_{13}: x^3y+y^6, & U_{12}: x^2y+y^3+z^4;\\
U_{12}: x^2y+xy^2+z^4.&&
\end{array}
\right.
\end{equation}

\begin{comment} in Table \ref{tab-additional-exceptional-singularities}.
\begin{center}
    \begin{table}[H]
        \caption{\label{tab-additional-exceptional-singularities}  Additional normal forms of exceptional unimodular singularities
    }
        \begin{tabular}{|c|c||c|c||c|c||c|c|}
       \hline
         & Normal Form &  & Normal Form &  & Normal Form &  & Normal Form   \\ \hline \hline
       $E_{14}$ & $x^3+y^8$  & $W_{12}$ & $x^2y+y^2+z^5$   &  $W_{13}$ & $x^4y+y^4$ && \\              \hline
        $Q_{12}$ &  $x^2y+y^5+z^3$ & $Z_{13}$ & $x^3y+y^6$ & $U_{12}$ & $x^2y+y^3+z^4$ & $U_{12}$ & $x^2y+xy^2+z^4$       \\              \hline
     \end{tabular}
   \end{table}
 \end{center}

\begin{ques}
   Does the LG-LG mirror symmetry conjecture still hold if the B-side $W^T$ is given by other normal forms of the exceptional unimodular singularites?
\end{ques}

\noindent Here we are using $E_{14}$ and $E_{14}'$ to specify the two corresponding  normal forms in Table \ref{tab-exceptional-singularities} and Table \ref{tab-additional-exceptional-singularities} of type $E_{14}$. Similar notation conventions are taken for the other cases.

Such representatives $h$ are distinguished by their groups $G_h$ of diagonal symmetry.
\end{comment}

It is quite natural to investigate Conjecture \ref{conj} for all the weighted homogeneous polynomial representatives  on   the B-side.
\begin{thm}\label{addtional7-FJRW}
Conjecture \ref{conj}  is true if $W^T$ is given by any weighted homogenous polynomial representative of the exceptional unimodular singularities that $W^T$ is not $x^2y+xy^2+z^4$.
That is, there exists a mirror map, such that
$$\mathscr{A}_{W}^{\rm FJRW}=\mathscr{A}_{W^T}^{\rm SG}.$$
\end{thm}
\begin{proof}[Sketch of the proof]
Thanks to Corollary \ref{main}, it remains to show the case when $W^T$ is given by \eqref{form:extra}. By Proposition \ref{good-basis-uniqueness}, there is a unique good section.
Let us specify a weighted homogeneous basis $\{\phi_1, \cdots, \phi_\mu\}$ of $\Jac(W^T)$ as in Table \ref{invertible-singularities} for each atomic type and take product of such bases for mixed types.
\begin{comment}
\begin{itemize}
\item For Fermat singularity $x^p$, we choose a basis given by $x^i$, $i\leq p-2$.
\item For chain singularity $x^py+y^q$, we choose the basis
           $$\{x^{p-1}\}\bigcup\{x^{i}y^{j}~|~i \leq p-2,\,\, j\leq q-1\}.$$
\item For chain singularity $x^py+y^qz+z^r$, we choose the basis given by
       $$\{x^{p-1}z^{k}~|~ k\leq r-1\}\bigcup \{x^{i}y^{j}z^{k}~|~i \leq p-2, \,\,j\leq q-1,\,\, k\leq r-1\}.$$
\item For loop singularity $x^py+xy^q$, we choose a basis $x^{i}y^{j}$ with $i \leq p-1, j\leq q-1.$
\end{itemize}
\end{comment}
Then we could obtain the four-point function by direct calculations (see the link in the appendix for precise output).
An isomorphism $\Psi:  \Jac(W^T)\to H_{W}$ is chosen similarly as in Section \ref{sec2}.
We compute the corresponding four-point FJRW correlators as in Proposition \ref{thm-reformulathm1} by the same proof therein. If $W^T$ is not $x^2y+xy^2+z^4$, then the four-point FJRW correlators turn out to be the same as the B-side four-point correlator up to a sign.
These invariants completely determine the full data of the generating function at all genera on both sides, by exactly the same reconstruction technique as in the previous two subsections. Therefore, we conclude the statement.
\end{proof}

\begin{rmk}
If $W^T=x^2y+xy^2+z^4$, $H_{W}$ has broad ring generators $x\be_{J^8}$ and $y\be_{J^8}$. Our method does not apply to compute
$$\LD x\be_{J^8},x\be_{J^8},y\be_{J^8},\be_{J^{15}}\RD_{0}^W, \quad \LD y\be_{J^8},y\be_{J^8},x\be_{J^8},\be_{J^{15}}\RD_{0}^W, \quad \be_{\phi_{\mu}}=\be_{J^{15}}.$$
If $W^T=x^2y+y^2+z^5$, we may need a further rescaling on $\Psi(x)$ since we only know
$$\left(\LD\Psi(x),\Psi(x),\Psi(y),\Psi(yz^3)\RD_{0}^W\right)^2=2\LD\Psi(y),\Psi(y),\Psi(y),\Psi(y),\Psi(yz^3)\RD_{0}^W=\frac{1}{4}.$$
The first equality is a consequence of the WDVV equation and the second equality is a consequence of the orbifold GRR calculation with codimension $D=2$ (i.e., formula \eqref{top-euler}).
\end{rmk}

\begin{comment}
Since $(W^T)^T=W$, it   is also quite natural to ask the following question.

\begin{ques}\label{question-A-side}
    Does the LG-LG mirror symmetry conjecture   hold if the A-side $W$ is given by a normal form of the exceptional unimodular singularites?
\end{ques}

Therefore, part of the answer to Question \ref{question-A-side} is the following  immediate consequence of Theorem \ref{addtional7-FJRW}.

Let $W$ be a normal form representing the Arnold's strange duality as in Table \ref{tab-strangeduality}. Then   there exists a mirror map, which could be described explicitly,  such that
$$\mathscr{A}_{W}^{\rm FJRW}=\mathscr{A}_{W^T}^{\rm SG}.$$

\end{comment}

\subsubsection*{The other direction}
\begin{comment}
 So far we have started with Arnold's exceptional unimodular singularities on the B-side, and proved their equivalence with the FJRW-theory of their BHK mirrors. However, we can also start with the FJRW-theory associated to a normal form of the exceptional unimodular singularities, and ask for the Saito-Givental theory of their BHK mirrors.
\end{comment}

Among all the   representatives  $W$ on the A-side,   there are in total three cases for which $W^T$ is no longer exceptional unimodular. The corresponding $W^T$ is given by $x^3+xy^6$, $x^2+xy^5+z^3$, or $x^2+xy^3+z^4$. Let us end this section by the following remark, which gives a positive answer to Conjecture \ref{conj} for those representatives.
\begin{rmk}\label{rmk.4.5}
  \begin{enumerate}
    \item For the remaining three cases,  $W^T$ is no longer given by any one of the  exceptional unimodular singularities.
    \item A similar calculation as Proposition \ref{good-basis-uniqueness} shows that there exists a unique primitive form (up to a constant) for $x^2+xy^5+z^3$.
%    By a dimension formula in \cite{LLSaito}, the moduli space of primitive forms   consists of one point for $x^2+xy^5+z^3$.
    However, for the other two cases   $x^3+xy^6$ and $x^2+xy^3+z^4$, there is a whole one-dimensional family of choices of primitive forms.
    \item Let us specify a basis $\{\phi_1, \cdots, \phi_\mu\}$ of $\Jac(W^T)$ following  Table \ref{invertible-singularities}. It is easy to check that   $\{[\phi_1d^n\mathbf{x}], \cdots, [\phi_\mu d^n\mathbf{x}]\}$ form a good basis and specifies a choice of primitive form. A similar calculation shows that the B-side four-point function coincides with the A-side one (up to a sign as before), and they completely determine the full data of the generating functions at all genera by the same reconstruction technique again.
  \end{enumerate}
\end{rmk}

\section{Appendix A}
\subsection{The vector space isomorphisms}
Here we list the vector space isomorphism $\Psi:\Jac(W^T)\to( H_W)$ for the remaining cases of $W$ in Table \ref{tab-exceptional-singularities}.

\noindent$(1)$
\textbf{$3$-Fermat type.} $W=W^T=x^3+y^3+z^4\in U_{12}$.
%All nonzero $H_\gamma$ are narrow sectors.
%The FJRW ring  $(H_W,\bullet)$ is generated by $\{\be_{2,1,1},\be_{1,2,1},\be_{1,1,2}\}$.
We denote $\be_{i, j,k}:=1 \in H_\gamma$ for $\gamma=\left(\exp({2\pi\sqrt{-1}\,i\over 3}), \exp({2\pi\sqrt{-1}\,j\over 3}, \exp({2\pi\sqrt{-1}\,k\over 4})\right)\in G_W$.
The isomorphism $\Psi$ is given by
\begin{equation*}\label{mirror-fermat}
\Psi(x^{i-1}y^{j-1}z^{k-1})=\be_{i,j,k}, \quad 1\leq i\leq 3,1\leq j\leq 3,1\leq k\leq 4.
\end{equation*}

\begin{comment}

\noindent$(2)$ \textbf{2-Chain type.}
$W=x^p+xy^q$, with $(p, q)=(5, 3)$ or $(3, 5)$, where $W^T$ is of type $E_{13}$ or $Z_{11}$. Since $p-1$ and $q$ are coprime, $G_W=\LD J\RD\cong\mu_{pq}$. As shown in \cite{FS}, the vector space
   $H_W$ is spanned by the narrow elements  $\{\be_{J^{i}}~|~ 1\leq i\leq pq-1, p\nmid i\}$ together with a   broad element $y^{q-1}\be_{J^0}$. Let $(k,m)$ be the unique solution to the equations
\begin{equation*}\label{chain}
-(k-1)(p-1)\equiv (m-2)(p-1)\equiv 1 \mod (pq)
\end{equation*}
The FJRW ring $(H_W,\bullet)$ is generated by $\{\be_{J^k}, \be_{J^m}\}$. For the $\C$-basis, we get
\begin{equation*}\label{ring-2-chain}
\Psi(y^{q-1})=\mp qy^{q-1}\be_{J^0};\quad
\Psi(x^s y^t)=\be_{J^i}, \quad 1\leq i\leq pq-1, p\nmid i.
\end{equation*}
where $(s, t)\in\Z^2$ ($0\leq s\leq p-2, 0\leq t\leq q-1$) is the unique solution to the equation
$$i\equiv1+s(k-1)+t(m-1) \mod (pq).$$

\end{comment}

\noindent$(2)$ \textbf{Chain type.}
Let  $W=x^3+xy^5$. The mirror $W^T$ is of type $Z_{11}$. Note $G_W\cong \mu_{15}$.

    \begin{table}[h]
    \resizebox{\columnwidth}{!}{
\begin{tabular}{|c||c|c|c|c|c|c|c|c|c|c|c|}
   \hline
  $H_W$  & $\be_{J}$ & $\be_{J^{13}}$ & $\be_{J^{11}}$ & $\be_{J^{10}}$ & $\be_{J^{8}}$ & $\mp5y^4\be_{J^0}$  & $\be_{J^{7}}$   & $\be_{J^5}$ & $\be_{J^4}$ & $\be_{J^{2}}$ & $\be_{J^{14}}$    \\
\hline
 ${\rm Jac}(W^T)$  & $1$ & $y$ & $x$ & $y^2$ & $xy$ & $x^2$ & $y^3$ & $xy^2$ & $y^4$ & $xy^3$ & $xy^4$ \\
\hline
 \end{tabular}
 }
 \end{table}

Let  $W=x^3y+y^5$. The mirror $W^T$ is of type $E_{13}$. Note $G_W\cong \mu_{15}$.

    \begin{table}[h]
% \caption{\label{mirror-2-loop} $\Psi: \Jac(W^T)\overset{\cong}{\to}(H_W,\bullet)$ with $W^T$ of type $Z_{12}$ or $S_{12}$}
\resizebox{\columnwidth}{!}{
\begin{tabular}{|c||c|c|c|c|c|c|c|c|c|c|c|c|c|}
   \hline
  $H_W$  & $\be_{J}$ & $\be_{J^{13}}$ & $\be_{J^{12}}$ & $\be_{J^{11}}$ & $\be_{J^{9}}$  & $\be_{J^{8}}$  & $\mp3y^2\be_{J^0}$ & $\be_{J^7}$ & $\be_{J^6}$ & $\be_{J^{5}}$ & $\be_{J^{4}}$ & $\be_{J^{2}}$ & $\be_{J^{14}}$   \\
\hline
 ${\rm Jac}(W^T)$  & $1$ & $y$ & $y^2$ & $x$ & $y^3$ & $xy$ & $y^4$ & $xy^2$ & $x^2$ & $xy^3$ & $x^2y$ & $x^2y^2$ & $x^2y^3$ \\
\hline
 \end{tabular}
 }
 \end{table}

Let  $W=x^2y+y^3z+z^3$. The mirror $W^T$ is of type $Z_{13}$. Note $G_W\cong \mu_{18}$.

    \begin{table}[h]
% \caption{\label{mirror-2-loop} $\Psi: \Jac(W^T)\overset{\cong}{\to}(H_W,\bullet)$ with $W^T$ of type $Z_{12}$ or $S_{12}$}
\resizebox{\columnwidth}{!}{
\begin{tabular}{|c||c|c|c|c|c|c|c|c|c|c|c|c|c|}
   \hline
  $H_W$  & $\be_{J}$ & $\be_{J^{16}}$ & $\be_{J^{14}}$ & $\be_{J^{13}}$ & $\be_{J^{11}}$  & $\be_{J^{10}}$  & $\mp3y^2\be_{J^9}$ & $\be_{J^8}$ & $\be_{J^7}$ & $\be_{J^{5}}$ & $\be_{J^{4}}$ & $\be_{J^{2}}$ & $\be_{J^{17}}$   \\
   \hline
 ${\rm Jac}(W^T)$  & $1$ & $y$ & $z$ & $y^2$ & $yz$ & $x$ & $z^2$ & $y^2z$ & $xy$ & $xz$ & $xy^2$ & $xyz$ & $xy^2z$ \\
\hline
 \end{tabular}
 }
 \end{table}

Let  $W=x^2y+y^2z+z^4$. The mirror $W^T$ is of type $W_{13}$. Note $G_W\cong \mu_{16}$.

    \begin{table}[H]
% \caption{\label{mirror-2-loop} $\Psi: \Jac(W^T)\overset{\cong}{\to}(H_W,\bullet)$ with $W^T$ of type $Z_{12}$ or $S_{12}$}
\resizebox{\columnwidth}{!}{
\begin{tabular}{|c||c|c|c|c|c|c|c|c|c|c|c|c|c|}
\hline
  $H_W$  & $\be_{J}$ & $\be_{J^{14}}$ & $\be_{J^{13}}$ & $\be_{J^{11}}$ & $\be_{J^{10}}$  & $\be_{J^{9}}$  & $\mp2y\be_{J^8}$ & $\be_{J^7}$ & $\be_{J^6}$ & $\be_{J^{5}}$ & $\be_{J^{3}}$ & $\be_{J^{2}}$ & $\be_{J^{15}}$   \\
   \hline
 ${\rm Jac}(W^T)$  & $1$ & $z$ & $y$ & $z^2$ & $yz$ & $x$ & $z^3$ & $yz^2$ & $xz$ & $xy$ & $xz^2$ & $xyz$ & $xyz^2$ \\
\hline
 \end{tabular}
 }
 \end{table}

\noindent$(3)$ \textbf{Loop type.}
There is one 2-Loop of type $Z_{12}$: $W=W^T=x^3y+xy^4$ with $G_W\cong\mu_{11}$.

    \begin{table}[H]
% \caption{\label{mirror-2-loop} $\Psi: \Jac(W^T)\overset{\cong}{\to}(H_W,\bullet)$ with $W^T$ of type $Z_{12}$ or $S_{12}$}
\resizebox{\columnwidth}{!}{
\begin{tabular}{|c||c|c|c|c|c|c|c|c|c|c|c|c|}
   \hline
 $H_W$  & $\be_{J}$ & $\be_{J^8}$ & $\be_{J^6}$ & $\be_{J^4}$ & $\be_{J^2}$  & $x^2\be_{J^0}$ & $y^3\be_{J^0}$ & $\be_{J^9}$ & $\be_{J^7}$ & $\be_{J^{5}}$ & $\be_{J^{3}}$ & $\be_{J^{10}}$    \\
   \hline
 ${\rm Jac}(W^T)$  & $1$ & $y$ & $x$ & $y^2$ & $xy$  & $x^2$ & $y^3$ & $xy^2$ & $x^2y$ & $xy^3$ & $x^2y^2$ & $x^2y^3$  \\
   \hline
 \end{tabular}
 }
 \end{table}

There is one   3-Loop with $W^T$ of type $S_{12}$:  $W=x^2z+xy^2+yz^3$ with $G_W\cong\mu_{13}$.

    \begin{table}[h]
% \caption{\label{mirror-2-loop} $\Psi: \Jac(W^T)\overset{\cong}{\to}(H_W,\bullet)$ with $W^T$ of type $Z_{12}$ or $S_{12}$}
\resizebox{\columnwidth}{!}{
\begin{tabular}{|c||c|c|c|c|c|c|c|c|c|c|c|c|}
   \hline
  $H_W$  & $\be_{J}$ & $\be_{J^{11}}$ & $\be_{J^{10}}$ & $\be_{J^9}$ & $\be_{J^8}$  & $\be_{J^7}$ & $\be_{J^6}$ & $\be_{J^5}$ & $\be_{J^4}$ & $\be_{J^{3}}$ & $\be_{J^{2}}$ & $\be_{J^{12}}$    \\
   \hline
 ${\rm Jac}(W^T)$  & $1$ & $z$ & $x$ & $y$ & $z^2$ & $xz$ & $yz$ & $xy$ & $xz^2$ & $yz^2$ & $xyz$ & $xyz^2$ \\
\hline
 \end{tabular}
 }
 \end{table}

\noindent$(4)$ \textbf{Mixed type.}
Let  $W=x^2+xy^4+z^3$. The mirror $W^T$ is of type $Q_{10}$. Note $G_W\cong \mu_{24}$.

     \begin{table}[h]
     \resizebox{\columnwidth}{!}{
 \begin{tabular}{|c||c|c|c|c|c|c|c|c|c|c|}
   \hline
 $H_W$  & $\be_{J}$ & $\be_{J^{19}}$ & $\be_{J^{17}}$ & $\mp 4y^{3}\be_{J^{16}}$ & $\be_{J^{13}}$& $\be_{J^{11}}$   & $\mp 4y^{3}\be_{J^{8}}$ & $\be_{J^{7}}$ & $\be_{J^{5}}$  &  $\be_{J^{23}}$ \\
   \hline
 ${\rm Jac}(W^T)$  & $1$ & $y$ & $z$ & $x$ & $y^2$ & $yz$ & $xz$ & $y^3$ & $y^2z$ & $y^3z$     \\
 \hline
 \end{tabular}
 }
 \end{table}

Let $W=x^2y+y^4+z^3$. The mirror $W^T$ is of type $E_{14}$. Note $G_W\cong \mu_{24}$.
     \begin{table}[h]
     \resizebox{\columnwidth}{!}{
 \begin{tabular}{|c||c|c|c|c|c|c|c|c|c|c|c|c|c|c|}
   \hline
 $H_W$  & $\be_{J}$ & $\be_{J^{22}}$& $\be_{J^{19}}$& $\be_{J^{17}}$  &  $\mp 2y\be_{J^{16}}$ & $\be_{J^{14}}$ & $\be_{J^{13}}$  &$\be_{J^{11}}$&  $\be_{J^{10}}$ & $\mp 2y\be_{J^{8}}$ & $\be_{J^{7}}$ & $\be_{J^{5}}$ & $\be_{J^{2}}$ & $\be_{J^{23}}$    \\
   \hline
 ${\rm Jac}(W^T)$  & $1$ & $z$ & $z^2$ & $x$ & $y^2$  & $xz$ & $x$ & $yz^2$ & $yz$ & $y^2z$ & $yz^2$ & $xy$ & $xyz$ & $xyz^2$    \\
   \hline
 \end{tabular}
 }
 \end{table}

\begin{comment}
Let $W=x^2+xy^5+z^3$. The mirror $W^T$ is of type $Q_{12}$. Note $G_W\cong \mu_{30}$.
     \begin{table}[H]
  %\caption{\label{mirror-2-mixedq4} $\Psi: \Jac(W^T)\overset{\cong}{\to}(H_W,\bullet)$   with  $W^T$ of type $Q_{10}$ or $Q_{12}$}
 \begin{tabular}{|c||c|c|c|c|c|c|c|c|c|c|c|c|}
   \hline
 $H_W$  & $\be_{J}$ & $\be_{J^{13}}$& $\be_{J^{11}}$ & $y^{4}\be_{J^{10}}$ & $\be_{J^{25}}$ & $\be_{J^{23}}$  & $\be_{J^{7}}$ & $y^{4}\be_{J^{20}}$ & $\be_{J^{5}}$ & $\be_{J^{19}}$ & $\be_{J^{17}}$ & $\be_{J^{29}}$    \\
   \hline
 ${\rm Jac}(W^T)$  & $1$ & $y$ & $z$ & $x$ & $y^2$  & $yz$ & $y^3$ & $xz$ & $y^2z$ & $y^4$ & $y^3z$ & $y^4z$   \\
   \hline
 \end{tabular}
 \end{table}

\end{comment}

 \subsection{Four-point functions for exceptional unimodular singularities}
In the following, we provide the four-point functions $\mc F_0^{(4)}(\mathbf{t})$ of the Frobenius manifold structure associated to the primitive form $\zeta$ for all the remaining 13 cases in Table \ref{tab-exceptional-singularities}. We mark  the terms that give the B-side four-point invariants corresponding to \eqref{4-point} by using boxes. We also provide   the expression of  $\zeta$ up to order 3. We remind our readers of $\tilde{\mc F}_0^{(4)}(\tilde{\mathbf{t}})=-\mc F_0^{(4)}(\mathbf{t})$ as discussed  in section 4.1.   We obtain the list with the help of a computer. The codes are written in mathematica 8,  available at
\quad http://member.ipmu.jp/changzheng.li/index.htm

 \noindent$\bullet$ Type $E_{13}$:  \noindent\(f= x^3+ xy^5.\quad \{\phi_i\}_i=   \left\{1,y,y^2,x,y^3,x y,y^4,y^2 x,x^2,y^3 x,y x^2,x^2 y^2,y^3 x^2\right\}.\)
$$\zeta =1-\frac{4 }{75}s_{12} s_{13}-\frac{1}{25}x s_{13}^2  + O(\mathbf{s}^4).$$
\begin{align*}
 - \mc F^{(4)}_0&=  -\frac{3}{10} t_6 t_7^3-\frac{3}{5} t_5 t_7^2 t_8+\frac{1}{10} t_5 t_6 t_8^2+\frac{1}{15} t_3 t_8^3
                            +\frac{1}{90} t_6^3t_9+\frac{3}{5}t_5 t_6 t_7 t_9+\frac{2}{5} t_4 t_7^2 t_9  +\frac{1}{5} t_5^2 t_8 t_9  \\
                      &   \quad
                             +\frac{1}{15} t_4 t_6 t_8 t_9+\frac{2}{5} t_3 t_7 t_8 t_9-\frac{1}{10} t_4 t_5t_9^2
                                      -\frac{1}{15} t_3 t_6 t_9^2-\frac{1}{30} t_2 t_8 t_9^2  +\frac{1}{10} t_5 t_6^2 t_{10} \\
                     & \quad -\frac{3}{10} t_5^2 t_7 t_{10}-\frac{3}{10} t_3 t_7^2 t_{10}+\frac{1}{5} t_3 t_6 t_8 t_{10}
                     + \frac{1}{10} t_2 t_8^2t_{10}+\frac{1}{30} t_4^2 t_9 t_{10}+\frac{1}{5} t_3 t_5 t_9 t_{10}  \\
                      &   \quad +\frac{1}{5} t_2 t_7 t_9 t_{10}                             +\frac{1}{10} t_2 t_6 t_{10}^2+\frac{3}{10}t_5^2 t_6 t_{11}+\frac{1}{15} t_4 t_6^2 t_{11}+\frac{3}{5} t_4 t_5 t_7 t_{11} +\frac{2}{5} t_3 t_6 t_7 t_{11}\\
                      &\quad    +\frac{1}{15} t_4^2 t_8 t_{11} +\frac{2}{5}t_3 t_5 t_8 t_{11}
                                       +\frac{1}{5} t_2 t_7 t_8t_{11}-\frac{2}{15} t_3 t_4 t_9 t_{11}-\frac{1}{15} t_2 t_6 t_9 t_{11} +\frac{1}{10} t_3^2 t_{10}t_{11}\\
                      &   \quad
                               +\frac{1}{5}t_2 t_5 t_{10} t_{11}   -\frac{1}{30} t_2 t_4 t_{11}^2 +\frac{1}{5} t_4 t_5^2 t_{12}+\frac{1}{10} t_4^2 t_6 t_{12}  +\frac{2}{5} t_3 t_5 t_6 t_{12}+\frac{2}{5} t_3 t_4 t_7 t_{12}\\
                      &\quad
                         +\frac{1}{5}t_2 t_6 t_7 t_{12}+\frac{1}{5} t_3^2 t_8 t_{12}+\frac{1}{5} t_2 t_5 t_8 t_{12} -\frac{1}{15} t_2 t_4 t_9 t_{12}
                                      +\frac{1}{5} t_2 t_3 t_{10} t_{12} \fbox{$+\frac{2}{45}t_4^3 t_{13}$}\\
                       &   \quad  +\frac{1}{5} t_3 t_4 t_5 t_{13}+\frac{1}{10} t_3^2 t_6 t_{13}       +\frac{1}{5} t_2 t_5 t_6 t_{13}
                                   +\frac{1}{5} t_2 t_4 t_7 t_{13}+\frac{1}{5}t_2 t_3 t_8 t_{13}\fbox{$+\frac{1}{10} t_2^2 t_{10} t_{13}$}
\end{align*}

 \bigskip

 \noindent$\bullet$ Type $E_{14}$: \noindent\(f=  x^2+x y^4+z^3.\quad \{\phi_i\}_i=   \left\{1,y,x,y^2,x y,y^3,x y^2,z,y z,x z,y^2 z,x y z,y^3 z,x y^2 z\right\}.\)
 $$\zeta =  1+\frac{1}{64} s_{12}^2 s_{14}+\frac{1}{64} s_{10} s_{14}^2+\frac{1}{48} y s_{12} s_{14}^2+\frac{1}{192} y^2 s_{14}^3+ O(\mathbf{s}^4).$$
 \begin{align*}
    -\mc F^{(4)}_0&=-{1\over 16} t_{3}^2 t_{5} t_{9}+{1\over 8} t_{5}^2 t_{6}t_{9}+{1\over 4} t_{4}t_{5} t_{ 7} t_{ 9}+{1\over 4} t_{ 3} t_{ 6} t_{ 7} t_{ 9}
                       +{1\over 8}t_{ 2} t_{ 7}^{ 2} t_{ 9}-{1\over 8} t_{ 3}^{ 2} t_{ 4} t_{ 10}-{1\over 8} t_{ 2} t_{ 3} t_{ 5} t_{ 10}\\
                   &\quad    +{1\over 2} t_{ 4} t_{ 5} t_{ 6} t_{ 10}+{3\over 8} t_{ 3} t_{ 6}^{ 2} t_{ 10}+{1\over 8} t_{ 4}^{ 2} t_{ 7} t_{ 10}+{1\over 4} t_{ 2} t_{ 6} t_{ 7} t_{ 10}+{1\over 6} t_{ 8} t_{ 9}^{ 2} t_{ 10}-{1\over 24} t_{ 3}^{ 3} t_{ 11}+{1\over 4} t_{ 4} t_{ 5}^{ 2} t_{ 11}\\
                    &\quad  +{1\over 2} t_{ 3} t_{ 5} t_{ 6} t_{ 11}+{1\over 4} t_{ 3} t_{ 4} t_{ 7} t_{ 11}+{1\over 4} t_{ 2} t_{ 5} t_{ 7} t_{ 11}-{1\over 4} t_{ 6}^{ 2} t_{ 7} t_{ 11}+{1\over 6} t_{ 8}^{ 2} t_{ 10} t_{ 11}-{1\over 6} t_{ 9}^{ 2} t_{ 11}^{ 2}-1/9 t_{ 8} t_{ 11}^{ 3}\\
                    &\quad -{1\over 16} t_{ 2} t_{ 3}^{ 2} t_{ 12}+{1\over 4} t_{ 4}^{ 2} t_{ 5} t_{ 12}+{1\over 2} t_{ 3} t_{ 4} t_{ 6} t_{ 12}+{1\over 4} t_{ 2} t_{ 5} t_{ 6} t_{ 12}-{1\over 6} t_{ 6}^{ 3} t_{ 12}+{1\over 4} t_{ 2} t_{ 4} t_{ 7} t_{ 12}+{1\over 6} t_{ 8}^{ 2} t_{ 9} t_{ 12}\\
                    &\quad +{1\over 2} t_{ 3} t_{ 4} t_{ 5} t_{ 13}+{1\over 8} t_{ 2} t_{ 5}^{ 2} t_{ 13}+{3\over 8} t_{ 3}^{ 2} t_{ 6} t_{ 13}-{1\over 2} t_{ 5} t_{ 6}^{ 2} t_{ 13}+{1\over 4} t_{ 2} t_{ 3} t_{ 7} t_{ 13}-{1\over 9} t_{ 9}^{ 3} t_{ 13}-{1\over 6} t_{ 8}^{ 2} t_{ 13}^{ 2}
                     \fbox{$+{1\over 18} t_{ 8}^{ 3} t_{ 14}$}\\
                    &\quad -{2\over 3} t_{ 8} t_{ 9} t_{ 11} t_{ 13}+{1\over 8} t_{ 3} t_{ 4}^{ 2} t_{ 14}+{1\over 4} t_{ 2} t_{ 4} t_{ 5} t_{ 14}+{1\over 4} t_{ 2} t_{ 3} t_{ 6} t_{ 14}-{1\over 4} t_{ 4} t_{ 6}^{ 2} t_{ 14}\fbox{$+{1\over 8} t_{ 2}^{ 2} t_{ 7} t_{ 14}$}-{1\over 2} t_{ 4} t_{ 6} t_{ 7} t_{ 13}
 \end{align*}

 \bigskip

 \noindent$\bullet$ Type $Z_{11}$: \noindent\(f=  x^3 y+y^5.\quad \{\phi_i\}_i=\left\{1,y,x,y^2,x y,x^2,y^3,x y^2,y^4,x y^3,x y^4\right\}\).
           $$\zeta =  1+\frac{17}{675} s_{10} s_{11}^2+\frac{2}{81} y s_{11}^3+ O(\mathbf{s}^4).$$
\begin{align*}
    - \mc F^{(4)}_0&= -\frac{5}{18} t_5 t_6^3+\frac{1}{3} t_4 t_6^2 t_7+\frac{1}{15} t_4 t_5 t_7^2-\frac{1}{90} t_3 t_7^3
                           +\frac{1}{18} t_5^3 t_8+\frac{2}{15}t_4^2 t_7 t_8+\frac{1}{3} t_3 t_6 t_7 t_8\\
                     &\quad +\frac{1}{10} t_2 t_7^2 t_8
                        +\frac{1}{6} t_3 t_5 t_8^2+\frac{1}{30} t_4^2 t_5 t_9+\frac{1}{3} t_3 t_5 t_6t_9+\frac{1}{3} t_2 t_6^2 t_9 -\frac{1}{15} t_3 t_4 t_7 t_9+\frac{1}{15} t_2 t_5 t_7 t_9\\
                     &\quad+\frac{1}{15} t_2 t_4 t_8 t_9 -\frac{1}{30} t_2 t_3 t_9^2 +\frac{1}{18} t_4^3 t_{10}+\frac{1}{6}t_3 t_5^2 t_{10}
                               +\frac{1}{3} t_3 t_4 t_6 t_{10}+\frac{1}{5} t_2 t_4 t_7 t_{10}\\
                     &\quad        +\frac{1}{6} t_3^2 t_8 t_{10}+\frac{1}{30} t_2^2 t_9 t_{10}
                             +\frac{1}{15}t_2 t_4^2 t_{11}\fbox{$+\frac{1}{6} t_3^2 t_5 t_{11}$}
                                          +\frac{1}{3} t_2 t_3 t_6 t_{11}\fbox{$+\frac{1}{15} t_2^2 t_7 t_{11}$}
\end{align*}

\bigskip

 \noindent$\bullet$ Type $Z_{12}$: \noindent\(f= x^3 y+x y^4.\quad \{\phi_i\}_i=\left\{1,y,x,y^2,x y,x^2,y^3,x y^2,x^2 y,x y^3,x^2 y^2,x^2 y^3\right\}\).

$$\zeta =  1-\frac{6 }{121}s_{11} s_{12}-\frac{5 }{121}y s_{12}^2+\frac{29 }{1331}s_{10} s_{12}^2+\frac{9 }{1331}x s_{12}^3+ O(\mathbf{s}^4).$$
\begin{align*}
   - \mc F^{(4)}_0&=-\frac{10}{33} t_5 t_6^3+\frac{5}{22} t_5 t_6^2 t_7+\frac{7}{22} t_5 t_6 t_7^2-\frac{7}{22} t_5 t_7^3
                                 +\frac{3}{11} t_4 t_6^2 t_8+\frac{4}{11}t_4 t_6 t_7 t_8\\
                  &\quad -\frac{6}{11} t_4 t_7^2 t_8+\frac{1}{11} t_4 t_5 t_8^2+\frac{2}{11} t_3 t_6 t_8^2-\frac{1}{22} t_3 t_7 t_8^2+\frac{1}{22} t_2 t_8^3
                                      +\frac{1}{66}t_5^3 t_9\\
                   &\quad -\frac{2}{11} t_4 t_5 t_6 t_9-\frac{2}{11} t_3 t_6^2 t_9+\frac{6}{11} t_4 t_5 t_7 t_9+\frac{1}{11} t_3 t_6 t_7 t_9+\frac{4}{11} t_3 t_7^2t_9+\frac{2}{11} t_4^2 t_8 t_9\\
                   &\quad   +\frac{1}{11} t_3 t_5 t_8 t_9-\frac{1}{11} t_2 t_6 t_8 t_9 +\frac{3}{11} t_2 t_7 t_8 t_9
                          -\frac{1}{11} t_3 t_4 t_9^2-\frac{1}{22} t_2 t_5 t_9^2+\frac{1}{22} t_4 t_5^2 t_{10}\\
                   &\quad +\frac{1}{22}t_4^2 t_6 t_{10}+\frac{4}{11} t_3 t_5 t_6 t_{10}+\frac{7}{22} t_2 t_6^2 t_{10}     -\frac{3}{22} t_4^2 t_7 t_{10}
                         -\frac{1}{11} t_3 t_5 t_7 t_{10}+\frac{1}{11}t_2 t_6 t_7 t_{10}\\
                   &\quad  -\frac{3}{22} t_2 t_7^2 t_{10}
                           -\frac{1}{11} t_3 t_4 t_8t_{10}+\frac{1}{11} t_2 t_5 t_8 t_{10}   +\frac{1}{22} t_3^2 t_9 t_{10}+\frac{1}{11}t_2 t_4 t_9 t_{10}
                    -\frac{1}{22} t_2 t_3 t_{10}^2\\
                    &\quad  +\frac{5}{22} t_4^2 t_5 t_{11}     +\frac{1}{11} t_3 t_5^2 t_{11}+\frac{2}{11} t_3 t_4 t_6 t_{11}   -\frac{1}{11} t_2 t_5 t_6 t_{11}  +\frac{5}{11} t_3 t_4t_7 t_{11}+\frac{3}{11} t_2 t_5 t_7 t_{11} \\
                  & \quad +\frac{1}{11} t_3^2 t_8 t_{11}
                          +\frac{3}{11} t_2 t_4 t_8 t_{11}-\frac{1}{11} t_2 t_3 t_9 t_{11}+\frac{1}{22}t_2^2 t_{10} t_{11}
              +\frac{1}{11} t_3 t_4^2 t_{12}\\
              &\quad\fbox{$+\frac{3}{22} t_3^2 t_5 t_{12}$}+\frac{2}{11} t_2 t_4 t_5 t_{12}
                          +\frac{3}{11} t_2 t_3 t_6 t_{12}+\frac{2}{11}t_2 t_3 t_7 t_{12}\fbox{$+\frac{1}{11} t_2^2 t_8 t_{12}$}
\end{align*}

\noindent\(                   {}\qquad\quad\,\,\\
                            {}\qquad\quad\,\,
                {}\qquad\)

\bigskip

\noindent$\bullet$ Type $Z_{13}$: \noindent\(f=x^2+x y^3+y z^3.\quad \{\phi_i\}_i=\left\{1,y,z,y^2,y z,x,z^2,y^2 z,x y,x z,x y^2,x y z,x y^2 z\right\}\).
$$\zeta = 1+\frac{7}{486} s_{12}^2 s_{13}+\frac{7}{486} s_{10} s_{13}^2+\frac{5 }{243}y s_{12} s_{13}^2+\frac{2}{729} y^2 s_{13}^3+ O(\mathbf{s}^4).$$
 \begin{align*}
   - \mc F^{(4)}_0&= -{1\over 12} t_{6}^{2} t_{7}^{2}-{1\over 6} t_{5} t_{7}^{3}-{5\over 108} t_{6}^{3} t_{8}-{1\over 6} t_{5}^{2} t_{8}^{2}
                       -{1\over 9} t_{3} t_{8}^{3}-{1\over 18} t_{5} t_{6}^{2} t_{9}+{1\over 3} t_{4} t_{7}^{2} t_{9}+{4\over 9} t_{4} t_{6} t_{8} t_{9}\\
             &\quad +{1\over 3} t_{3} t_{7} t_{8} t_{9}+{1\over 9} t_{4} t_{5} t_{9}^{2}+{1\over 36} t_{3} t_{6} t_{9}^{2}
                     +{1\over 6} t_{2} t_{8} t_{9}^{2}+{1\over 18} t_{5}^{3} t_{10}-{1\over 6} t_{4} t_{6}^{2} t_{10}-{1\over 6} t_{3} t_{6} t_{7} t_{10}\\
             &\quad +{1\over 3} t_{3} t_{5} t_{8} t_{10}-{1\over 6} t_{2} t_{6} t_{9} t_{10}
            +{1\over 9} t_{4} t_{5} t_{6} t_{11}+{1\over 36} t_{3} t_{6}^{2} t_{11}+{1\over 3} t_{3} t_{5} t_{7} t_{11}+{2\over 9} t_{2} t_{4} t_{6} t_{13}\\
             &\quad +{1\over 3} t_{2} t_{4} t_{9} t_{12}-{1\over 9} t_{4}^{2} t_{8} t_{11}+{1\over 9} t_{2} t_{6} t_{8} t_{11}-{1\over 9} t_{3} t_{4} t_{9} t_{11}+{1\over 9} t_{2} t_{5} t_{9} t_{11}+{1\over 9} t_{2} t_{4} t_{10} t_{11}\\
             &\quad+{1\over 3} t_{2} t_{7}^{2} t_{11} -{1\over 24} t_{3}^{2} t_{10}^{2} +{1\over 6} t_{3} t_{5}^{2} t_{12}+{5\over 18} t_{4}^{2} t_{6} t_{12}-{1\over 12} t_{2} t_{6}^{2} t_{12}+{1\over 3} t_{3} t_{4} t_{7} t_{12}+{1\over 6} t_{3}^{2} t_{8} t_{12}\\
             &\quad+{1\over 18} t_{2}^{2} t_{11} t_{12}-{2\over 27} t_{4}^{3} t_{13}\fbox{$+{1\over 6} t_{3}^{2} t_{5} t_{13}$}+{1\over 3} t_{2} t_{3} t_{7} t_{13}\fbox{$+{1\over 9} t_{2}^{2} t_{9} t_{13}$}-{1\over 18} t_{2} t_{3} t_{11}^{2}+{2\over 9} t_{4}^{2} t_{9} t_{10}
    \end{align*}

\bigskip

\noindent$\bullet$ Type $W_{12}$: \noindent\(f=x^4+y^5.\quad \{\phi_i\}_i=\left\{1,y,x,y^2,x y,x^2,y^3,x y^2,x^2 y,x y^3,x^2 y^2,x^2 y^3\right\}\).
$$\zeta =   1-\frac{1}{20}s_{11} s_{12}-\frac{1}{20}y s_{12}^2+ O(\mathbf{s}^4).$$
\begin{align*}
 -  \mc F^{(4)}_0&= \frac{1}{20} t_5^2 t_7^2+\frac{1}{8} t_5 t_6^2 t_8+\frac{1}{5} t_4 t_5 t_7 t_8+\frac{1}{10} t_4^2 t_8^2+\frac{1}{10} t_2 t_7 t_8^2
                          +\frac{1}{8}t_5^2 t_6 t_9\\
                &\quad+\frac{1}{10} t_4^2 t_7 t_9+\frac{1}{10} t_2 t_7^2 t_9
                          +\frac{1}{4} t_3 t_6 t_8 t_9+\frac{1}{8} t_3 t_5 t_9^2+\frac{1}{10} t_4^2 t_5t_{10}
                            +\frac{1}{8} t_3 t_6^2t_{10}\\
                &\quad+\frac{1}{5} t_2 t_5 t_7 t_{10}+\frac{1}{5} t_2 t_4 t_8 t_{10}+\frac{1}{20} t_2^2 t_{10}^2
                   +\frac{1}{15} t_4^3t_{11}+\frac{1}{4} t_3 t_5 t_6 t_{11}\\
                   &\quad+\frac{1}{5} t_2 t_4 t_7 t_{11}+\frac{1}{8} t_3^2 t_9 t_{11}
                           +\frac{1}{10} t_2 t_4^2 t_{12}\fbox{$+\frac{1}{8} t_3^2t_6 t_{12}$}\fbox{$+\frac{1}{10} t_2^2 t_7 t_{12}$}
\end{align*}

\bigskip

\noindent$\bullet$ Type $W_{13}$: \noindent\(f= x^2+x y^2+y z^4.\quad \{\phi_i\}_i= \left\{1,z,y,z^2,y z,x,z^3,y z^2,x z,x y,x z^2,x y z,x y z^2\right\}\).
 $$\zeta =  1-\frac{5 }{64}s_{11} s_{13}+\frac{15 s_{12}^2 s_{13}}{1024}-\frac{ y s_{13}^2}{128}+\frac{3 s_{10} s_{13}^2}{256}
                +\frac{11 z s_{12}s_{13}^2 }{512}+\frac{3z^2 s_{13}^3 }{512}+ O(\mathbf{s}^4).$$
\begin{align*}
    -\mc F^{(4)}_0 &=  -{3\over 32} t_{6}^{2} t_{7}^{2} -{1\over  6} t_{5} t_{7}^{3} -{1\over 48} t_{6}^{3} t_{8} -{1\over  4} t_{4} t_{7}^{2} t_{8}
               -{1\over 8} t_{5}^{2} t_{8}^{2}-{3\over 3 2} t_{5} t_{6}^{2} t_{9} -{1\over  4} t_{4} t_{6} t_{7} t_{9} +{1\over  4} t_{4} t_{5} t_{8} t_{9}\\
           &\quad  +{1\over 8} t_{2} t_{8}^{2} t_{9} -{1\over 1 6} t_{4}^{2} t_{9}^{2} -{1\over 8} t_{3} t_{6} t_{9}^{2} -{1\over 1 6} t_{2} t_{7} t_{9}^{2}
                                   +{1\over 1 6} t_{5}^{2} t_{6} t_{10} +{1\over 1 6} t_{4} t_{6}^{2} t_{10}+{1\over  2} t_{4} t_{5} t_{7} t_{10}\\
            &\quad +{3\over 8} t_{3} t_{7}^{2} t_{10}+{1\over 8} t_{4}^{2} t_{8} t_{10}+{1\over 8} t_{3} t_{6} t_{8} t_{10}
                   +{1\over  4} t_{2} t_{7} t_{8} t_{10}+{1\over 8} t_{3} t_{5} t_{9} t_{10}+{1\over 1 6} t_{2} t_{6} t_{9} t_{10} \\
            &\quad -{1\over 8} t_{3} t_{4} t_{10}^{2}-{1\over 1 6} t_{2} t_{5} t_{10}^{2}+{1\over 8} t_{4} t_{5}^{2} t_{11} -{1\over 1 6} t_{4}^{2} t_{6} t_{11} -{1\over 8} t_{3} t_{6}^{2} t_{11} -{1\over 8} t_{2} t_{6} t_{7} t_{11}\\
            &\quad +{1\over  4} t_{2} t_{5} t_{8} t_{11} -{1\over 8} t_{2} t_{4} t_{9} t_{11}+{1\over 1 6} t_{3}^{2} t_{10} t_{11}-{1\over 3 2} t_{2}^{2} t_{11}^{2}+{1\over  4} t_{4}^{2} t_{5} t_{12}+{1\over  4} t_{3} t_{5} t_{6} t_{12}\\
            &\quad +{1\over 3 2} t_{2} t_{6}^{2} t_{12}+{1\over  2} t_{3} t_{4} t_{7} t_{12}+{1\over  4} t_{2} t_{5} t_{7} t_{12}+{1\over  4} t_{2} t_{4} t_{8} t_{12}+{1\over 8} t_{3}^{2} t_{9} t_{12} -{1\over 8} t_{2} t_{3} t_{10} t_{12}\\
            &\quad +{1\over 8} t_{3} t_{4}^{2} t_{13}+{1\over  4} t_{2} t_{4} t_{5} t_{13}\fbox{$+{3\over 1 6} t_{3}^{2} t_{6} t_{13}$}+{1\over  4} t_{2} t_{3} t_{7} t_{13}\fbox{$+{1\over 8} t_{2}^{2} t_{8} t_{13}$}
\end{align*}

\bigskip

\noindent$\bullet$ Type $Q_{10}$: \noindent\(f=x^2 y+y^4+z^3.\quad \{\phi_i\}_i=\left\{1,y,z,x,y^2,y z,x z,y^3,y^2 z,y^3 z\right\}\).
$$\zeta =1+\frac{3}{128} s_9 s_{10}^2+\frac{11 }{384}y s_{10}^3+ O(\mathbf{s}^4).$$
\begin{align*}- \mc F^{(4)}_0&= \frac{1}{24} t_5^3 t_6+\frac{1}{18} t_3 t_6^3+\frac{1}{4} t_4 t_5^2 t_7-\frac{1}{3} t_3^2 t_7^2+\frac{1}{4} t_4^2 t_6 t_8
                               +\frac{1}{8}t_2 t_5 t_6 t_8+\frac{1}{2} t_2 t_4 t_7 t_8\\
                       &\quad +\frac{1}{4} t_4^2 t_5 t_9 +\frac{1}{8} t_2 t_5^2 t_9+\frac{1}{6} t_3^2 t_6 t_9+\frac{1}{16} t_2^2 t_8t_9
                             \fbox{$+\frac{1}{18} t_3^3 t_{10}$}\fbox{$+\frac{1}{4} t_2 t_4^2 t_{10}$}\fbox{$+\frac{1}{16} t_2^2 t_5 t_{10}$}
\end{align*}

\bigskip

\noindent$\bullet$ Type $Q_{11}$: \noindent\(f=x^2 y+y^3 z+z^3.\quad \{\phi_i\}_i=\left\{1,y,z,x,y^2,y z,z^2,x z,y^2 z,y z^2,y^2 z^2\right\}\).
  $$\zeta= 1-\frac{5 }{108}s_{10} s_{11}-\frac{1}{24}y s_{11}^2+\frac{13}{648} s_9 s_{11}^2
                         +\frac{25 }{1944}z s_{11}^3+ O(\mathbf{s}^4).$$
\begin{align*} -\mc F^{(4)}_0&=  \frac{1}{36} t_5 t_6^3+\frac{1}{4} t_5^2 t_6 t_7+\frac{1}{36} t_3 t_6^2 t_7-\frac{1}{24} t_4^2 t_7^2-\frac{1}{9} t_3 t_5 t_7^2
                               -\frac{1}{18}t_2 t_6 t_7^2+\frac{1}{2} t_4 t_5 t_6 t_8\\
                        &\quad -\frac{1}{6} t_3 t_4 t_7 t_8   -\frac{1}{4} t_3^2 t_8^2-\frac{1}{12} t_5^3 t_9+\frac{1}{4} t_4^2 t_6 t_9+\frac{1}{12}t_2 t_6^2 t_9
                                   +\frac{1}{36} t_3^2 t_7 t_9\\
                        &\quad +\frac{1}{6} t_2 t_5 t_7 t_9+\frac{1}{2} t_2 t_4 t_8 t_9 +\frac{1}{4} t_4^2 t_5 t_{10}  +\frac{1}{4} t_3 t_5^2 t_{10}
                                  +\frac{1}{12}t_3^2 t_6 t_{10}+\frac{1}{3} t_2 t_5 t_6 t_{10}\\
                        &\quad -\frac{1}{9} t_2 t_3 t_7 t_{10}+\frac{1}{12} t_2^2 t_9 t_{10}\fbox{$+\frac{5}{108} t_3^3 t_{11}$}\fbox{$+\frac{1}{4}t_2 t_4^2 t_{11}$}  +\frac{1}{6} t_2 t_3 t_5 t_{11}\fbox{$+\frac{1}{12} t_2^2 t_6 t_{11}$}
\end{align*}

\bigskip

\noindent$\bullet$ Type $Q_{12}$: \noindent\(f= x^2 y+x y^3+z^3.\quad \{\phi_i\}_i=\left\{1,x,y,x y,y^2,x y^2,z,x z,y z,x y z,y^2 z,x y^2 z\right\}\).
$$\zeta = 1+\frac{1}{75} s_{10}^2 s_{12}+\frac{1}{75} s_8 s_{12}^2+\frac{1}{50} y s_{10} s_{12}^2+\frac{1}{25}   s_{11}s_{12}^2+ O(\mathbf{s}^4).$$
\begin{align*}
   - \mc F^{(4)}_0&= -{3\over 10} t_{2}^{2} t_{4} t_{8}-{1\over 10} t_{3} t_{4}^{2} t_{8}+{1\over 5} t_{2} t_{4} t_{5} t_{8}+{3\over 10} t_{4} t_{5}^{2} t_{8}+{2\over 5} t_{2} t_{3} t_{6} t_{8}+{1\over 5} t_{3} t_{5} t_{6} t_{8}\fbox{$+{1\over 10} t_{3}^{2} t_{4} t_{12}$}\\
     &\quad-{1\over 10} t_{2} t_{4}^{2} t_{9}+{1\over 5} t_{4}^{2} t_{5} t_{9}+{1\over 5} t_{2}^{2} t_{6} t_{9}+{1\over 5} t_{3} t_{4} t_{6} t_{9}+{1\over 5} t_{2} t_{5} t_{6} t_{9}-{1\over 5} t_{5}^{2} t_{6} t_{9}+{1\over 6} t_{7} t_{8} t_{9}^{2}-{1\over 36}t_{9}^{4}\\
     &\quad +{1\over 5} t_{2} t_{3} t_{5} t_{12}-{1\over 5} t_{2} t_{3} t_{4} t_{10}+{1\over 10} t_{2}^{2} t_{5} t_{10}+{2\over 5} t_{3} t_{4} t_{5} t_{10}+{3\over 10} t_{2} t_{5}^{2} t_{10}-{1\over 5} t_{5}^{3} t_{10}+{1\over 10} t_{3}^{2} t_{6} t_{10}\\
     &\quad-{1\over 10} t_{2}^{3} t_{10}+{1\over 10} t_{2}^{2} t_{4} t_{11}+{1\over 5} t_{3} t_{4}^{2} t_{11}+{3\over 5} t_{2} t_{4} t_{5} t_{11}-{3\over 5} t_{4} t_{5}^{2} t_{11}+{1\over 5} t_{2} t_{3} t_{6} t_{11}-{2\over 5} t_{3} t_{5} t_{6} t_{11}\\
     &\quad +{1\over 6} t_{7}^{2} t_{8} t_{11} -{1\over 3} t_{7} t_{9}^{2} t_{11}-{1\over 6} t_{7}^{2} t_{11}^{2}\fbox{$+{1\over 5} t_{2}^{2} t_{3} t_{12}$}-{1\over 5} t_{3} t_{5}^{2} t_{12}\fbox{$+{1\over 18} t_{7}^{3} t_{12}$} +{1\over 6} t_{7}^{2} t_{9} t_{10}-{1\over 4} t_{7}^{2} t_{8}^{2}
\end{align*}

\bigskip

\noindent$\bullet$ Type $S_{11}$: \noindent\(f=x^2 y+y^2 z+z^4.\quad \{\phi_i\}_i= \left\{1,z,x,y,z^2,x z,y z,z^3,x z^2,y z^2,y z^3\right\}\).
$$\zeta = 1-\frac{3 }{64}s_{10} s_{11}-\frac{7 }{128}z s_{11}^2+\frac{9}{512} s_8 s_{11}^2
                         +\frac{5 }{1024}y s_{11}^3+ O(\mathbf{s}^4).$$
\begin{align*}- \mc F^{(4)}_0&=    -\frac{5}{32} t_5^2 t_6^2+\frac{1}{48} t_5^3 t_7+\frac{1}{4} t_4 t_6^2 t_7+\frac{1}{4} t_3 t_6 t_7^2-\frac{1}{16} t_4 t_5^2 t_8
                                -\frac{1}{8}t_3 t_5 t_6 t_8\\
                            &\quad -\frac{1}{16} t_2 t_6^2 t_8+\frac{1}{8} t_4^2 t_7 t_8  +\frac{1}{16} t_2 t_5 t_7 t_8-\frac{1}{32} t_3^2 t_8^2-\frac{1}{16} t_2 t_4 t_8^2- \frac{1}{16} t_3 t_5^2 t_9\\
                            &\quad
                                 -\frac{1}{2} t_2 t_5 t_6 t_9+\frac{1}{2} t_3 t_4 t_7 t_9-\frac{1}{8} t_2 t_3 t_8 t_9
                              -\frac{1}{8} t_2^2 t_9^2+\frac{1}{8}t_4^2 t_5 t_{10}+\frac{1}{8} t_2 t_5^2 t_{10}\\
                            &\quad +\frac{1}{2} t_3 t_4 t_6 t_{10}
                                  +\frac{1}{4} t_3^2 t_7 t_{10}+\frac{1}{32} t_2^2 t_8 t_{10}\fbox{$+\frac{1}{4}t_3^2 t_4 t_{11}$}
                                  \fbox{$+\frac{1}{8} t_2 t_4^2 t_{11}$}\fbox{$+\frac{3}{32} t_2^2 t_5 t_{11}$}
\end{align*}

\bigskip

\noindent$\bullet$ Type $S_{12}$: \noindent\(f= x^2 y+y^2 z+x z^3.\quad \{\phi_i\}_i= \left\{1,z,x,y,z^2,x z,y z,x y,x z^2,y z^2,x y z,x y z^2\right\}\).
$$\zeta =  1-\frac{12 s_{10} s_{12}}{169}+\frac{30 s_{11}^2 s_{12}}{2197}-\frac{2 x s_{12}^2}{169}+\frac{20 s_8 s_{12}^2}{2197}+\frac{93 z s_{11}
s_{12}^2}{4394}+\frac{9 z^2 s_{12}^3}{2197}+ O(\mathbf{s}^4).$$
\begin{align*}
   - \mc F^{(4)}_0&=   -\frac{5 }{156}t_6^4+\frac{1}{13} t_5 t_6^2 t_7-\frac{1}{13} t_5^2 t_7^2-\frac{1}{13} t_4 t_6 t_7^2
                                 -\frac{1}{26} t_3 t_7^3+\frac{5}{26}t_5^2 t_6 t_8\\
                &\quad +\frac{1}{26} t_4 t_6^2 t_8+\frac{2}{13} t_4 t_5 t_7 t_8  +\frac{1}{13} t_3 t_6 t_7 t_8+\frac{1}{26} t_2 t_7^2 t_8
                         -\frac{3}{52} t_4^2t_8^2-\frac{2}{13} t_3 t_5 t_8^2\\
                & \quad   -\frac{1}{13} t_2 t_6 t_8^2-\frac{1}{39} t_5^3 t_9  -\frac{2}{13} t_4 t_5 t_6 t_9-\frac{1}{13} t_3 t_6^2 t_9
                     +\frac{1}{13} t_4^2 t_7 t_9+\frac{1}{13} t_3 t_5 t_7 t_9\\
                &\quad +\frac{1}{13} t_2 t_6 t_7 t_9+\frac{1}{13} t_3 t_4 t_8 t_9 +\frac{1}{13}t_2 t_5 t_8 t_9
                       -\frac{1}{26} t_3^2 t_9^2-\frac{1}{13} t_2 t_4 t_9^2 -\frac{1}{13} t_4 t_5^2 t_{10}\\
                &\quad   +\frac{1}{26} t_4^2 t_6 t_{10}+\frac{1}{13} t_3t_5 t_6 t_{10}
                                    +\frac{2}{13} t_2 t_6^2 t_{10}-\frac{3}{13} t_3 t_4 t_7 t_{10}-\frac{2}{13} t_2 t_5 t_7 t_{10}+\frac{1}{26} t_3^2 t_8 t_{10}\\
                &\quad +\frac{1}{13} t_2 t_4 t_8 t_{10}+\frac{1}{13} t_2 t_3 t_9 t_{10}-\frac{1}{26} t_2^2 t_{10}^2+\frac{5}{26} t_4^2 t_5 t_{11}
                        +\frac{7}{26} t_3 t_5^2t_{11}+\frac{3}{13} t_3 t_4 t_6 t_{11}\\
                &\quad +\frac{4}{13} t_2 t_5 t_6 t_{11}  +\frac{3}{26} t_3^2 t_7 t_{11}+\frac{1}{13} t_2 t_4 t_7 t_{11}-\frac{2}{13}t_2 t_3 t_8 t_{11}
                                      +\frac{1}{26} t_2^2 t_9 t_{11}\fbox{$+\frac{5}{26} t_3^2 t_4 t_{12}$}\\
                &\quad \fbox{$+\frac{2}{13} t_2 t_4^2 t_{12}$} +\frac{3}{13} t_2 t_3 t_5 t_{12}\fbox{$+\frac{3}{26}t_2^2 t_6 t_{12}$}
\end{align*}

\bigskip

\noindent$\bullet$ Type $U_{12}$: \noindent\(f=x^3+y^3+z^4.\quad \{\phi_i\}_i= \left\{1,z,x,y,z^2,x z,y z,x y,x z^2,y z^2,x y z,x y z^2\right\}\).

 $$\zeta = 1+\frac{1}{72} s_{11}^2 s_{12}+\frac{1}{72} s_8 s_{12}^2+\frac{1}{36} z s_{11} s_{12}^2+\frac{1}{72} z^2 s_{12}^3+ O(\mathbf{s}^4).$$
\begin{align*} -\mc F^{(4)}_0&=   \frac{1}{8} t_5^2 t_6 t_7+\frac{1}{6} t_3 t_6^2 t_8+\frac{1}{6} t_4 t_7^2 t_8+\frac{1}{4} t_2 t_5 t_7 t_9
                                          +\frac{1}{6} t_3^2 t_8 t_9+\frac{1}{4}t_2 t_5 t_6 t_{10}+\frac{1}{6} t_4^2 t_8 t_{10}\\
                     &\quad +\frac{1}{8} t_2^2 t_9 t_{10}+\frac{1}{8} t_2 t_5^2 t_{11}+\frac{1}{6} t_3^2 t_6 t_{11}
                           +\frac{1}{6}t_4^2 t_7 t_{11}\fbox{$+\frac{1}{18} t_3^3 t_{12}$}\fbox{$+\frac{1}{18} t_4^3 t_{12}$}\fbox{$+\frac{1}{8} t_2^2 t_5 t_{12}$}
\end{align*}

\footnotesize
\noindent\textit{Acknowledgments.}

 The authors would like to thank Huai-Liang Chang, Kentaro Hori, Hiroshi Iritani, Todor Eliseev Milanov, Atsushi Takahashi, and especially Yongbin Ruan, for helpful discussions and valuable comments.
 %%The first author is partially supported by JSPS Grant-in-Aid for Young Scientists (B) No. 25870175.
The authors also thank the   referee(s) for the careful reading and valuable comments. C. L. is supported by IBS-R003-D1. S.L. is partially supported by NSF DMS-1309118. K.S. is partially supported by  JSPS Grant-in-Aid for Scientific Research (A) No. 25247004.
In addition, Y.S. would like to thank Alessandro Chiodo, Jeremy Gu\'er\'e, Weiqiang He for helpful discussions.
S.L. would like to thank Kavli IPMU, and Y.S. would like to thank IAS of HKUST for their hospitality. Part of the work was done during their visit.

\end{document}